\documentclass[11pt]{amsart}
\usepackage{amssymb,amsmath}

\theoremstyle{plain}
\newtheorem{thrm}{Theorem}[section]
\newtheorem{lemma}[thrm]{Lemma}
\newtheorem{prop}[thrm]{Proposition}

\newtheorem{rmrk}[thrm]{Remark}
\newtheorem{dfn}[thrm]{Definition}

\numberwithin{equation}{section}

\usepackage{hyperref}
\hypersetup{colorlinks=false,allcolors=false,hidelinks}

\begin{document}


\newcommand{\SL}{\mathcal L^{1,p}( D)}
\newcommand{\Lp}{L^p( Dega)}
\newcommand{\CO}{C^\infty_0( \Omega)}
\newcommand{\Rn}{\mathbb R^n}
\newcommand{\Rm}{\mathbb R^m}
\newcommand{\R}{\mathbb R}
\newcommand{\Om}{\Omega}
\newcommand{\Hn}{\mathbb H^n}
\newcommand{\N}{\mathbb N}
\newcommand{\aB}{\alpha B}
\newcommand{\eps}{\epsilon}
\newcommand{\BVX}{BV_X(\Omega)}
\newcommand{\p}{\partial}
\newcommand{\IO}{\int_\Omega}
\newcommand{\bG}{\mathbb{G}}
\newcommand{\bg}{\mathfrak g}
\newcommand{\Bux}{\mbox{Box}}
\newcommand{\al}{\alpha}
\newcommand{\til}{\tilde}
\newcommand{\nuX}{\boldsymbol{\nu}^X}
\newcommand{\bN}{\boldsymbol{N}}
\newcommand{\nh}{\nabla_H}
\newcommand{\deh}{\Delta_H}
\newcommand{\rh}{|\nabla_H \rho|^2}
\newcommand{\uh}{|\nabla_H u|^2}
\newcommand{\Gp}{G_{D,p}}
\newcommand{\n}{\boldsymbol \nu}
\newcommand{\ve}{\varepsilon}
\newcommand{\dsh}{|\nabla_H \rho| d\sigma_H}
\newcommand{\la}{\lambda}
\newcommand{\vf}{\varphi}
\newcommand{\rhh}{|\nabla_H \rho|}
\newcommand{\Ba}{\mathcal{B}_\alpha}
\newcommand{\Za}{Z_\alpha}
\newcommand{\ra}{\rho_\alpha}
\newcommand{\na}{\nabla_\alpha}
\newcommand{\vt}{\vartheta}
\newcommand{\us}{\R^{n+1}_+}

\newcommand{\BB}{B(r)}
\newcommand{\BBone}{B(1)}
\newcommand{\divv}{{\rm div}}


\title[Regularity of the singular set in the fractional obstacle problem]
{Structure and regularity of the singular set in the obstacle problem for the fractional Laplacian}

\author{Nicola Garofalo}
\address{Dipartimento di Ingegneria Civile, Edile e Ambientale (DICEA) \\ Universit\`a di Padova\\ 35131 Padova, ITALY}

\email[Nicola Garofalo]{rembrandt54@gmail.com}


\author{Xavier Ros-Oton}
\address{The University of Texas at Austin, Department of Mathematics, 2515 Speedway, Austin, TX 78751, USA} \email[Xavier Ros-Oton]{ros.oton@math.utexas.edu}

\thanks{NG was partially supported by a grant from the University of Padova ``Progetti d'Ateneo 2013". XR was partially supported by NSF grant DMS-1565186 and by MINECO grant MTM2014-52402-C3-1-P (Spain)}



\maketitle

\begin{abstract}
We study the singular part of the free boundary in the obstacle problem for the fractional Laplacian, \   $\min\bigl\{(-\Delta)^su,\,u-\varphi\bigr\}=0$ in $\R^n$, for general obstacles $\varphi$.
Our main result establishes the complete structure and regularity of the singular set.
To prove it, we construct new monotonicity formulas of Monneau-type that extend those in \cite{GP} to all $s\in(0,1)$.
\end{abstract}

\tableofcontents

\section{Introduction and main results}

The goal of this paper is to study the structure and regularity of the singular part of the free boundary in the obstacle problem for the fractional Laplacian.
Given a smooth function $\varphi :\R^n\to \R$,  this problem consists in finding a function $u$ defined in $\Rn$ such that
\begin{equation}\label{pb}
\left\{\begin{array}{rcl}
\displaystyle\min\bigl\{ u-\varphi,\,(-\Delta)^su\bigr\}&=&0\quad \textrm{in}\ \R^n,\vspace{1mm}\\
\displaystyle\lim_{|x|\to\infty}u(x)&=&0,
\end{array}\right.
\end{equation}
where for any $s\in(0,1)$ the symbol $(-\Delta)^s$ denotes the nonlocal  $s$-Laplacian defined by
\begin{equation}\label{fl}
\qquad\qquad (-\Delta)^s u(x)= \frac{\gamma_{n,s}}{2} \int_{\R^n} \frac{2u(x)-u(x+z) - u(x-z)}{|z|^{n+2s}} dz, 
\end{equation}
and $\gamma_{n,s}>0$ is given by
\[
\gamma_{n,s} = \frac{s 2^{2s} \Gamma(\frac n2 + s)}{\pi^{\frac n2} \Gamma(1-s)},
\]
see for instance \cite{L}. We note in passing that $\gamma_{n,s}$ is so chosen that $\widehat{(-\Delta)^s u}(\xi) = (2\pi |\xi|)^{2s} \hat u(\xi)$, for every $u\in \mathcal S(\Rn)$, where $\hat u(\xi) = \int_{\Rn} e^{-2i\pi \xi\cdot x} u(x) dx$ is the Fourier transform of $u$. 

The function $\vf$ in \eqref{pb} denotes the obstacle, the set 
\[
 \Lambda_\vf(u) = \{x\in \Rn\,:\, u(x) = \vf(x)\}
 \]
 is the so-called \emph{coincidence set}, and its topological boundary $\Gamma_\vf(u) = \partial \Lambda_\vf(u)$ is the so-called \emph{free boundary}. 
When $\varphi = 0$ we simply write $\Lambda(u)$ and $\Gamma(u)$.

An important motivation for studying the obstacle problem \eqref{pb} comes from Probability, where \eqref{pb} arises in optimal stopping problems for stochastic processes with jumps.
In particular, such type of models are used in Mathematical Finance, see \cite{CT}.
Problem \eqref{pb} also appears in other contexts, such as the study of the regularity of minimizers of interaction energies in kinetic equations; see \cite{CDM}.
In case $s=\frac12$, it appears as well in elasticity and in the study of semipermeable membranes, see \cite{DL}.
In such case it is also known as Signorini, or thin obstacle, problem, see \cite{F} and \cite{Fr}.

\addtocontents{toc}{\protect\setcounter{tocdepth}{1}} 

\subsection{Known results}

The regularity of solutions and free boundaries for \eqref{pb} was first studied in \cite{S} and \cite{CSS}.
The main results of \cite{CSS} establish that, when the obstacle $\varphi\in C^{2,1}$, then the solution $u$ possesses the optimal regularity $C^{1+s}(\R^n)$, and that at any free boundary point $x_0\in \Gamma_\varphi(u)$ one has the following dichotomy:
\begin{itemize}
\item[(a)] either \qquad\quad $0<c\,r^{1+s}\leq \sup_{B_r(x_0)}(u-\varphi)\leq C\,r^{1+s}$;\quad \vspace{2mm}
\item[(b)] or \qquad\qquad\qquad\qquad\, $0\leq\sup_{B_r(x_0)}(u-\varphi)\leq C\,r^2$.
\end{itemize}
Moreover, it was proved that the set of points satisfying (a) is an open subset of the free boundary, and it is locally a $C^{1,\alpha}$ graph.

In the existing literature, the free boundary points are usually subdivided into
three categories:
\begin{itemize}
\item[(i)] the set of \emph{regular points}, i.e., those satisfying (a);
\item[(ii)] the set $\Sigma_\varphi$ of \emph{singular points}, consisting of those free boundary points at which the coincidence set $\Lambda_\vf(u)$ has zero $n$-dimensional density, i.e.,
\[\lim_{r\downarrow 0}\frac{\bigl|\Lambda_\vf(u) \cap B_r(x_0)\bigr|}{|B_r(x_0)|}=0;\]
\item[(iii)] those free boundary points which are neither regular, nor singular.
\end{itemize}

For instance, the function $u(x) = (x_1^+)^{1+s}\in C^{1,s}(\R^n)$ is a solution of the obstacle problem corresponding to zero obstacle and satisfying (a) above at all points of its free boundary. As explained above, the set of regular points was studied in \cite{CSS}, where Caffarelli, Salsa, and Silvestre proved that this is an open subset of the free boundary and it is locally  a $C^{1,\alpha}$ hypersurface.
Recently, the set of regular free boundary points has been proved to be $C^\infty$ independently by \cite{KRS} and \cite{JN} (in fact, the method on \cite{KRS} shows that the regular free boundary is a real-analytic hypersurface whenever $\varphi$ is analytic); see also \cite{KPS,DS} for the case $s = \frac 12$.

On the other hand, in the case $s=\frac12$ the set $\Sigma_\varphi$ of {singular points} was studied  by the first named author and Petrosyan in \cite{GP}.
The main results of \cite{GP} establish that, for problem \eqref{pb} with $s=\frac12$, the blow-up at any singular point is a unique homogeneous polynomial of degree $2m$, and that the set of singular points is contained in a countable union of $C^1$ manifolds.

Furthermore, Barrios, Figalli and the second named author proved in \cite{BFR} that, when the obstacle $\varphi$ satisfies
\begin{equation}\label{concavity}
\Delta \varphi\leq0\qquad \textrm{in}\quad\, \{\varphi>0\}\subset\subset\R^n,
\end{equation}
then regular and singular points exhaust all possible free boundary points, and showed that under the assumption \eqref{concavity} singular points are either isolated or locally contained inside a $C^1$ submanifold. 
In particular, under the assumption \eqref{concavity} there are no free boundary points in the above category (iii).

Finally, in case of {zero obstacle}, very recently Focardi and Spadaro \cite{FSpadaro} established for the first time a regularity result for the set of points (iii).
By using methods from Geometric Measure Theory, they establish a regularity result for the free boundary up to a set of null $\mathcal H^{n-1}$ measure.

The aim of this paper is to study the complete structure and regularity of the set $\Sigma_\varphi$ of \emph{singular points} in the problem \eqref{pb} above and for general obstacles $\varphi$. 
As explained in detail below, we will prove uniqueness of blow-ups at \emph{all} singular free boundary points, and deduce that the singular set is contained in a countable union of $C^1$ manifolds.
Our results are new even in the case of zero obstacle, and extend the results of \cite{GP} to all $s\in(0,1)$.
We emphasize that the set $\Sigma_\varphi(u)$ is not necessarily a small part of the free boundary $\Gamma_\varphi(u)$, and that in fact the whole free boundary could coincide with $\Sigma_\varphi(u)$.

\subsection{Main results}

We first study the obstacle problem \eqref{pb} when the obstacle $\varphi$ is \emph{real-analytic}.
In that case, after an appropriate transformation (see Lemma \ref{analytic}) we can reduce our analysis to the case in which the obstacle is zero. For any free boundary point $x_0\in \Gamma(u)$ the blow-up of $u$ at $x_0$ is homogeneous of degree $\kappa\in[1+s,\infty)$.

We denote by $\Gamma_\kappa(u)$ the set of free boundary points at which the homogeneity of blow-up's is $\kappa$.
We denote $\Sigma(u)$ the set of \emph{singular} free boundary points, and we set $\Sigma_\kappa(u)=\Sigma(u)\cap\Gamma_\kappa(u)$.
Moreover, given an homogeneous polynomial $p_{2m}$ of degree $2m$ in $\R^n$, we will denote
\[d(p_{2m}):=\textrm{dim}\bigl\{\xi\in \R^n\mid \xi\cdot \nabla p_{2m}(x) = 0\quad\textrm{for every}\ x\in\R^n\bigr\}.\]
Then, we have the following result on the complete structure and regularity of the singular set $\Sigma(u)$.

\begin{thrm}\label{th-analytic}
Let $u$ be the solution of the obstacle problem \eqref{pb}, with $\varphi:\R^n\to\R$ satisfying
\begin{equation}\label{obstacle-analytic}
\varphi\quad \textrm{is analytic in}\quad \{\varphi>0\},\quad \textrm{and}\quad\varnothing\neq \{\varphi>0\}\subset\subset\R^n.
\end{equation}
Then,
\[\Sigma(u)=\bigcup_{m=1}^\infty \Sigma_{2m}(u).\]
Moreover, the blow-up of $u$ at any $x_0\in \Sigma_{2m}(u)$ is a {unique} homogeneous polynomial $p_{2m}^{x_0}$ of degree $2m$, and
\[\Sigma_{2m}(u)=\bigcup_{d=1}^{n-1}\Sigma_{2m}^d(u),\]
where
\begin{equation}\label{Sigma-2m-d}
\qquad\qquad\qquad\qquad\Sigma_{2m}^d(u):=\bigl\{x_0\in \Sigma_{2m}(u)\mid d(p_{2m}^{x_0})=d\bigr\},\qquad d=0,1,...,n-1.
\end{equation}
Furthermore, every set $\Sigma_{2m}^d(u)$ is contained in a countable union of $d$-dimensional $C^1$ manifolds.
\end{thrm}

When the obstacle $\varphi$ is not analytic but only $C^{k,\gamma}$, then we establish a similar result on the structure and regularity of the singular set.
It reads as follows.

\begin{thrm}\label{th-main}
Let $u$ be the solution of the obstacle problem \eqref{pb}, with $\varphi:\R^n\to\R$ satisfying
\begin{equation}\label{obstacle}
\varphi\in C^{k,\gamma}(\R^n),\quad \textrm{with}\quad k\geq2,\quad \gamma\in(0,1),\quad \textrm{and}\quad\varnothing\neq \{\varphi>0\}\subset\subset\R^n.
\end{equation}
Then, for every free boundary point $x_0\in \Gamma(u)$ we have:
\begin{itemize}
\item[(a)] either any blow-up of $u$ at $x_0$ is homogeneous of degree $\kappa$ for some $\kappa<k+\gamma$
\item[(b)] or \ $\sup_{B_r(x_0)}(u-\varphi)=o(r^{k+\gamma-\varepsilon})$ for all $\varepsilon>0$.
\end{itemize}
Given $\kappa<k+\gamma$, let $\Sigma_\kappa(u)$ be the set of singular points at which any blow-up of $u$ is homogeneous of degree $\kappa$.
Then $\kappa=2m$ for some positive integer $m$.
Moreover, the blow-up of $u$ at any $x_0\in \Sigma_{2m}(u)$ is a {unique} homogeneous polynomial $p_{2m}^{x_0}$ of degree $2m$, and
\[\Sigma_{2m}(u)=\bigcup_{d=1}^{n-1}\Sigma_{2m}^d(u),\]
where $\Sigma_{2m}^d(u)$ is defined as in \eqref{Sigma-2m-d}.
Furthermore, every set $\Sigma_{2m}^d(u)$ is contained in a countable union of $d$-dimensional $C^1$ manifolds.
\end{thrm}

As said above, Theorems \ref{th-analytic} and \ref{th-main} extend to all $s\in(0,1)$ the results of \cite{GP} for $s=\frac12$.
Moreover, even in the case $s=\frac12$, Theorem \ref{th-main} above slightly improves on the analogue result in \cite{GP}, in the sense that it requires less regularity on the obstacle $\varphi$.
Indeed, to study $\Sigma_{2m}(u)$, in \cite{GP} it was required that $\varphi$ was at least $C^{2m+1,1}$, while here we only need $\varphi\in C^{2m,\gamma}$ for some small $\gamma>0$.

\subsection{Plan of the paper}

The paper is organized as follows.
In Section \ref{S:sb} we study the relation between the fractional Laplacian and the  operators $\mathcal B_\alpha$ in  \eqref{bg0} below.
Using such relation, we find in Section \ref{S:op} new monotonicity formulas of Weiss and Monneau type for the fractional obstacle problem.
In Section \ref{sec4} we use these monotonicity formulas to establish the complete structure and regularity of the set of singular free boundary points in the fractional obstacle problem with zero obstacle.
In Section \ref{sec5} we prove Theorem \ref{th-analytic} by using the results for the zero obstacle case.
In Section \ref{sec6} we start the study of the obstacle problem \eqref{pb} with nonzero $C^{k,\gamma}$ obstacles $\varphi$, and establish a generalized Almgren frequency formula for the problem.
In Section \ref{sec7} we establish a new Monneau-type monotonicity formula for the problem. Finally, in Section \ref{sec8} we prove Theorem \ref{th-main}.

\section{The extension operator $L_a$ and the operator $\mathcal B_\alpha$}

In their extension paper \cite{CS} Caffarelli and Silvestre investigated the connection between the nonlocal operator $(-\Delta)^s$ and certain (local) degenerate elliptic operators. More precisely, they showed that $(-\Delta)^s$ can be recovered as a weighted Dirichlet-to-Neumann map of a certain degenerate elliptic operator $L_a$, where $a = 1- 2s$  (the so-called \emph{extension operator}), see \eqref{cs} and \eqref{cs2} below.    The aim of this section is to further explore such connection and also, at the same time, prepare the ground for some new developments that constitute the core of our work. In particular, we will further explore the link between the fractional Laplacian and another degenerate operator $\Ba$, where now $\alpha = \frac{a}{1-a}$, see \eqref{csb} and \eqref{csb2} below. We will show that, in the range $0<s<\frac 12$, various old and new results  about the nonlocal operator $(-\Delta)^s$, and its local extension counterpart $L_a$, can be directly deduced from corresponding ones for $\Ba$. We mention that, in such range of values of $s$, we have $\alpha>0$, and thus the differential operator $\Ba$ is the prototype of a class of equations introduced by S.~Baouendi in his Ph. D. Dissertation that continues to be much studied nowadays. It is our hope that by emphasizing the beautiful link between these two areas of PDE's, the results in this paper will encourage a new interaction between two seemingly disjoint communities: that of workers in subelliptic equations, and that of workers in nonlocal PDE's.

\subsection{The extension operator}\label{SS:ext}
Let $0<s<1$ and consider in $\Rn$ the fractional Laplacian of order $s$ defined by \eqref{fl} above. In \cite{CS} Caffarelli and Silvestre considered the extension problem in $\R^{n+1}_+$ with variables $(x,y)$, $x\in \Rn, y>0$,
\begin{equation}\label{cs}
\begin{cases}
L_a u(x,y) = \operatorname{div}(y^a \nabla u)(x,y) = 0\ \ \ \text{in}\ \ \R^{n+1}_+,\ \ \ \  a = 1-2s,
\\
u(x,0) = \vf(x),\ \ \ \ \ x\in \Rn,
\end{cases}
\end{equation}
and they proved that there exists a constant $C = C(n,s)>0$ such that
\begin{equation}\label{cs2}
C (-\Delta)^s \vf(x) = - \underset{y\to 0^+}{\lim} y^a D_yu(x,y).
\end{equation}
We mention that, in probability, the extension idea had already been proposed in the 1969 paper \cite{MO}.

Since $0<s<1$, we see from \eqref{cs} that $-1<a<1$. On the other hand, given a number $a\in (-1,1)$, we can consider, independently from $(-\Delta)^s$, the operator $L_a$ in \eqref{cs}. It is clear that we can write it in non-divergence form
\begin{equation}\label{la}
L_a = y^a \left(\Delta_x + D_{yy} + \frac ay D_y\right).
\end{equation}
In order to eliminate the drift term in \eqref{la} in \cite{CS} the authors introduced a change of variable $\Phi: \R^{n+1}_+ \to \R^{n+1}_+$ in the following form
\begin{equation}\label{phi}
(x,z) = \Phi(x,y) = (x,h(y)).
\end{equation}
Now, given a function $u(x,z)$ defined for $(x,z)\in \us$, we define a function $\tilde u(x,y)$, with $(x,y) \in \us$, by the formula
\begin{equation}\label{tu}
\tilde u(x,y) := u(\Phi(x,y)) = u(x,h(y)).
\end{equation}
A simple computation gives
\[
L_a \tilde u(x,y) = y^a\left[\Delta_x u(x,h(y)) + \left(h''(y) + \frac ay h'(y)\right) D_z u(x,h(y)) + h'(y)^2 D_{zz} u(x,h(y))\right].
\]
From this equation it is clear that if we choose the function $h(y)$ to satisfy the differential equation
\begin{equation}\label{hde}
h''(y) + \frac ay h'(y)  \equiv 0,
\end{equation}
then we obtain
\begin{equation}\label{lab}
L_a \tilde u(x,y) =  (h^{-1}(z))^a\left[\Delta_x u(x,z) + h'(h^{-1}(z))^2 D_{zz} u(x,z)\right].
\end{equation}
Integrating \eqref{hde} we obtain $h(y) = A y^{1-a}$ for some $A\in \R\setminus \{0\}$. At this point we choose $A$ in such a way that $h'(h^{-1}(z)) = z^{-\frac{a}{1-a}}$, which gives $A = (1-a)^{-(1-a)}$. Since with this choice we have $h'(y) = (1-a)^a y^{-a}$, we conclude from \eqref{lab} that
\begin{align}\label{lab2}
L_a \tilde u(x,y) & =  (1-a)^a z^{\frac{a}{1-a}} \left[\Delta_x u(x,z) + z^{-\frac{2a}{1-a}} D_{zz} u(x,z)\right].
\end{align}
From \eqref{lab2} we recognize that the mapping $u\to \tilde u$ defined by \eqref{tu}, where $h:(0,\infty)\to (0,\infty)$ is
the strictly increasing function given by
\begin{equation}\label{h}
h(y) = \left(\frac{y}{1-a}\right)^{1-a},\ \ \ \text{with inverse}\ \ \ \ h^{-1}(z) = (1-a) z^{\frac{1}{1-a}},
\end{equation}
converts in a one-to-one, onto fashion, solutions of the equation $L_a$ with respect to the variables $(x,y)\in \us$ into solutions with respect to the variables $(x,z)\in \us$ of the equation
\begin{equation}\label{bg0}
\Delta_x u(x,z) + z^{-\frac{2a}{1-a}} D_{zz} u(x,z) = 0,\ \ \ \ \ \ \ \ \ \ \ \  \ \  \alpha = \frac{a}{1-a}.
\end{equation}
This is the equation (1.8) in \cite{CS}.

We note for later purpose that the Jacobian determinant of the  diffeomorphism $\Phi$, with $h$ as in \eqref{h}, is given by
\begin{equation}\label{j}
|J_\Phi(x,y)| = (1-a)^a y^{-a}.
\end{equation}

\subsection{$(-\Delta)^s$ met Salah Baouendi}\label{S:sb}
Although this aspect went unnoticed in \cite{CS}, we can also write the equation \eqref{lab2} above in the following form
\begin{align}\label{csnb}
L_a \tilde u(x,y) & =  (1-a)^a z^{-\frac{a}{1-a}} \left[D_{zz} u(x,z) + z^{\frac{2a}{1-a}} \Delta_x u(x,z)\right].
\end{align}
When $0<s<1/2$ we have $a = 1-2s>0$, and therefore also
\begin{equation}\label{alpha}
\alpha = \frac{a}{1-a} = \frac{1-2s}{2s} = \frac{1}{2s} - 1>0.
\end{equation}
 In such situation the differential operator within square brackets in the right-hand side of \eqref{csnb} is a special case of the family of operators in $\Rn_x\times \Rm_z$ given by
 \begin{equation}\label{ba}
 \Ba=\Delta_z+ |z|^{2\alpha}\Delta_x,\ \ \ \ \ \ \alpha >0.
\end{equation}
Such operators are degenerate elliptic along the $n$-dimensional subspace $M = \Rn \times  \{0\}_{\R^m} $. Nowadays, operators such as \eqref{ba} are known as Baouendi-Grushin operators. They were first introduced by S. Baouendi in 1967 in his Doctoral Dissertation \cite{B} under Malgrange. M. Vishik learnt about his results during a visit to Malgrange and obtained permission to suggest to his student Grushin some questions connected with the hypoellipticity of \eqref{ba}, see \cite{Gr1}, \cite{Gr2}.  A decade later, in the early 80's, Franchi and Lanconelli introduced a class of operators which are modeled on \eqref{ba}, and they pioneered the study of the fine properties of their weak solutions, such as H\"older continuity and Harnack type inequalities, through a deep analysis of the natural control distance associated with the relevant operators, see \cite{FL1}-\cite{FL5}, and also the subsequent work \cite{FS}. At the same time, Fabes, Kenig and Serapioni, and Fabes, Jerison and Kenig published their pioneering papers \cite{FKS}, \cite{FJK} on a general class of degenerate-elliptic operators that include the operator $L_a$ in \eqref{cs} above (as we will see, in retrospect these two lines of research are not unrelated). Yet, another decade later, the first named author established in \cite{G} strong unique continuation for the Baouendi-Grushin operator \eqref{ba}. His main result was the discovery of some monotonicity formulas of Almgren type for the operator $\Ba$ in \eqref{ba} above. Such formulas, that have been recently further generalized in the work \cite{GR}, constitute motivational ground for the results in the present paper.

In the case $s = \frac 12$, we have $a = \alpha = 0$, and as it is well-known the extension operator $L_a$ in \eqref{cs} above is just the standard Laplacian in the variables $(x,y)\in \R^{n+1}_+$. In such case, the obstacle problem for $(-\Delta)^{1/2}$ corresponds to the classical Signorini problem.

In the range $\frac 12 < s < 1$ we have $-1<a<0$, and $-\frac 12 < \alpha <0$. In this case, the operator within square brackets in the right-hand side of \eqref{csnb} is no longer of Baouendi type, since it now presents itself in the form $\Ba u = D_{zz} u + z^{-2|\alpha|} \Delta_x u$. Nonetheless, even if now $-\frac 12 < \alpha <0$, the operator $\Ba$ continues to be scale invariant with respect to the anisotropic dilations \eqref{dilB} below. Furthermore, even in this negative range of $\alpha$, the transformation \eqref{tu} above converts, in a one-to-one, onto fashion, solutions of  $\Ba u = 0$ into solutions of $L_a\tilde u  = 0$.  Since for any $-1<a<1$ the function $\omega(x,y) = |y|^a$  is an $A_2$ weight of Muckenhoupt, we infer that for every $0<s<1$ the operator $L_a$ in \eqref{cs} belongs to the class of degenerate elliptic operators introduced in \cite{FKS}. Therefore, in particular, weak solutions in the class $W^{1,2}(|y|^a dx dy)$ of \eqref{cs} satisfy a Harnack inequality, they are locally H\"older continuous and in fact they are real analytic for $y>0$. It follows that from the regularity properties for $L_a$ established in \cite{CSS} we obtain corresponding regularity properties for $\Ba$ that are necessary to carry through all relevant computations in this paper.

\subsection{Back to the extension problem}\label{S:bext} If we now return to \eqref{cs} keeping \eqref{lab2} and \eqref{ba} in mind, we see that the Caffarelli-Silvestre extension problem can be alternatively formulated in the following way: given a function $\vf\in C^\infty_0(\Rn)$, consider the problem
\begin{equation}\label{csb}
\begin{cases}
\Ba u(x,z) = D_{zz} u + z^{2\alpha} \Delta_x u = 0\ \ \ \text{in}\ \ \R^{n+1}_+,\ \ \ \  \alpha = \frac{1}{2s} -1,
\\
u(x,0) = \vf(x),\ \ \ \ \ x\in \Rn.
\end{cases}
\end{equation}
Then, one has with $C >0$ as in \eqref{cs2}
\begin{equation}\label{csb2}
C (-\Delta)^s \vf(x) = - (2s)^{1-2s} \underset{z\to 0^+}{\lim} D_z u(x,z).
\end{equation}
The proof of \eqref{csb2} follows immediately by \eqref{cs2} above, the chain rule, and by the fact that $y^a h'(y) = (1-a)^a = (2s)^{1-2s}$.

The equations \eqref{csb} and \eqref{csb2} prove the following remarkable fact that, for the operator in the right-hand side of \eqref{lab2}, was already noted in \cite{CS}.

\begin{prop}\label{P:flp}
Given any $s\in (0,1)$, the fractional Laplacian $(-\Delta)^s$ in $\Rn$ can  be interpreted as the Dirichlet-to-Neumann map of the operator $\Ba$ in $\us$, where $\alpha = \frac{1}{2s} - 1$.
\end{prop}

\subsection{Some properties of the operator $\Ba$}\label{S:sp}

 Let $\alpha >0$. Returning to the  operator $\Ba$ in \eqref{ba}, we next recall some facts from the cited 1993 work \cite{G}. If we equip $\Rn_x\times \R_z^{m}$ with the following non-isotropic dilations
\begin{align}\label{dilB}
 \delta_\lambda(x,z)&=(\lambda^{\alpha+1}x, \lambda z),\ \ \ \ \ \ \lambda>0,
\end{align}
then we say that a function $u$ is $\delta_\la$-homogeneous (or simply, homogeneous) of degree $\kappa$ if
\[
u(\delta_\la(x,z)) = \la^\kappa u(x,z),\ \ \ \ \ \la>0.
\]
It is straightforward to verify that the partial differential operator $\Ba$ is $\delta_\la$-homogeneous of degree two, i.e.,
\begin{equation*}
 \Ba(\delta_\lambda \circ u)=\lambda^2\delta_\lambda\circ (\Ba u).
\end{equation*}

The infinitesimal generator of the dilations \eqref{dilB} is given by
\begin{equation}\label{Zagen}
 Z_\alpha= (\alpha+1)\sum_{i=1}^n x_i \partial_{x_i} + \sum_{j=1}^m z_j\partial_{z_j}.
\end{equation}
It is easy to verify that $u$ is homogeneous of degree $\kappa$ if and only if
\begin{equation}\label{Zak}
\Za u = \kappa u.
\end{equation}

We note that Lebesgue measure in $\Rn_x\times \Rm_z$ changes according to the equation
\begin{align*}
 d(\delta_\lambda(x,z))&=\lambda^{(\alpha+1)n+ m} dx dz,
\end{align*}
which motivates the definition of the \emph{homogeneous dimension} for the number
\begin{align}\label{Qa}
 Q=Q_\alpha=(\alpha+1)n + m.
\end{align}

In the study of the operators \eqref{ba} the following pseudo-gauge introduced in \cite{G} plays an important role
\begin{align}\label{ra}
 \rho_\alpha(x,z) & = \left((\alpha+1)^2 |x|^2 + |z|^{2(\alpha+1)}\right)^{\frac1{2(\alpha+1)}}.
\end{align}
We clearly have
\begin{equation}\label{gaugehom}
\rho_\alpha(\delta_\la(x,z)) = \lambda \rho_\alpha(x,z),
\end{equation}
i.e., the pseudo-gauge is homogeneous of degree one. This gives in particular,
\[
Z_\alpha \rho_\alpha = \rho_\alpha.
\]
The pseudo-ball and sphere centered at the origin with radius $r>0$ are respectively defined as
\begin{align}\label{BSr}
 B_{\rho_\alpha}(r)&=\{(x,z)\in \Rn_x\times \Rm_z \mid \rho_\alpha(x,z)<r\},\ \ \ \ \ S_{\rho_\alpha}(r) = \p B_{\rho_\alpha}(r).
\end{align}
In \cite{G} it was proved that, with $Q$ as in \eqref{Qa}, and $C_\alpha>0$ given by
\[
C_\alpha^{-1} = (Q + 2\alpha)(Q-2) \int_{\Rn_x\times \Rm_z} \frac{|z|^{\alpha} dx dz}{\left[((\alpha +1)^2 |x|^2 + |z|^{\alpha+1} + 1)\right]^{1 +\frac{Q+2\alpha}{2(\alpha+1)}}},
\]
 the function
\begin{equation}\label{Ga}
 \Gamma(x,z)=\frac{C_{\alpha}}{\ra(x,z)^{Q-2}}
\end{equation}
is a fundamental solution for $-\Ba$ with singularity at $(0,0)$. Since the operator is invariant with respect to translations along $M = \Rn \times  \{0\}_{\R^m} $, from \eqref{Ga} we immediately obtain the fundamental solution for $\Ba$ with singularity at any point of the subspace $M$.

\subsection{Further interplay between the operators $\Ba$ and $L_a$}\label{S:fi}

Henceforth in this paper, we focus on the situation in which $m=1$, so that $\Ba$ is given in $\Rn_x\times \R_z$ by
\begin{equation}\label{bg}
\Ba u = D_{zz} u + |z|^{2\alpha} \Delta_x u = 0,
\end{equation}
see \eqref{csnb} above. In this case, the infinitesimal generator \eqref{Zagen} of the dilations \eqref{dilB} becomes
\begin{equation}\label{Za}
 Z_\alpha= (\alpha+1)\sum_{i=1}^n x_i \partial_{x_i} +  z\partial_{z}.
\end{equation}

Functions that satisfy $\Ba u=0$ may have a certain degree of singularity along the subspace $M = \Rn_x\times\{0\}_z$. Given an open set $\Om \subset \R^{n}_x\times \R_z$, it is thus necessary to introduce the following more intrinsic classes of ``smooth  functions''
\begin{align*}
 \Gamma^1_\alpha(\Omega)=\{f\in C(\Omega)\mid f,X_jf\in C(\Omega)\},
\end{align*}
where $X_j f$ indicates the weak derivative along the
not necessarily smooth vector fields
\begin{align}\label{XBa}
 X_j=\begin{cases}
 |z|^\alpha \partial_{x_{j}},\ \ \ j = 1, 2,\ldots,n,
 \\
\p_z,\ \ \ \ \ \ \ \ \ j=n+1.
     \end{cases}
\end{align}
  From the expression \eqref{XBa} of the $X_j$'s it is clear that outside the subspace $M$ we have
\begin{equation}\label{div}
\operatorname{div} X_j = 0,\ \ \ \ \ \ \ j=1,...,n+1.
\end{equation}
We also set
\[
\Gamma^2_\alpha(\Omega)=\{f\in C(\Omega)\mid f,X_jf\in \Gamma^1_\alpha(\Omega)\}.
\]
Thus, classical solutions to $\Ba u=0$ in $\Omega$ are taken to be of class $\Gamma^2_\alpha(\Omega)$. Notice that membership in such class implies, in particular, that $D_{zz} u$ exists and it is continuous in $\Om$.

\begin{lemma}\label{L:comm}
For every $j=1,...,n+1$, we have for the commutator of $X_j$ and $Z_\alpha$
\[
[X_j,Z_\alpha] = X_j,
\]
outside $M$.
\end{lemma}

\begin{proof}
It is left to the reader.

\end{proof}

In the transformation \eqref{tu} above, homogeneities transform according to the following proposition.

\begin{prop}\label{P:hom}
A function in the variables $(x,z)\in \R^{n+1}$ is homogeneous of degree $\kappa$ with respect to the non-isotropic dilations \eqref{dilB} above if and only if the function $\tilde u$ defined by \eqref{tu} is homogeneous of degree
\begin{equation}\label{tk}
\tilde \kappa \overset{def}{=} (1-a)\kappa
\end{equation} with respect to the standard Euclidean dilations $\tilde \delta_\mu(x,y) = (\mu x,\mu y)$. Accordingly, if we indicate with $\tilde Z \tilde u(x,y) = \langle x,\nabla_x \tilde u(x,y)\rangle + y \p_y \tilde u(x,y)$ the infinitesimal generator of the Euclidean dilations $\tilde \delta_\mu$,  then one has the formula
\begin{equation}\label{gd}
(\alpha + 1) \tilde Z \tilde u(x,y) = Z_\alpha u(x,z),
\end{equation}
where $z = h(y)$. We can rewrite \eqref{gd} as $(\alpha + 1) \tilde Z \tilde u = \widetilde{Z_\alpha u}$.
\end{prop}

\begin{proof}
Let $(x,z) = (x,h(y))$, with $h$ given by \eqref{h} above. We have from \eqref{tu}
\begin{align*}
\tilde u(\lambda^\beta x, \lambda^\gamma y) & = u(\lambda^\beta x,\lambda^{(1-a)\gamma} h(y)) = u(\lambda^{\alpha+1}x,\lambda h(y)),
\end{align*}
provided we choose $\beta = \alpha +1$ and $\gamma = \frac{1}{1-a}$. If now $u$ is homogeneous of degree $\kappa$ with respect to the dilations \eqref{dilB}, we obtain from the latter equation
\[
\tilde u(\lambda^\beta x, \lambda^\gamma y) = \lambda^\kappa u(x,h(y)) = \lambda^\kappa \tilde u(x,y).
\]
Keeping in mind that $\alpha = \frac{a}{a-1}$, we conclude that $\beta = \alpha + 1 = \frac{1}{1-a} = \gamma$. Therefore, letting $\mu = \lambda^{\frac{1}{1-a}}$, we find
\[
\tilde u(\mu x,\mu y) = \mu^{(1-a)\kappa} \tilde u(x,y) = \mu^{\tilde \kappa}\ \tilde u(x,y),
\]
which gives the sought for conclusion. The proof of \eqref{gd} now follows from \eqref{tu}, \eqref{h} and an application of the chain rule.

\end{proof}

Since it will be relevant subsequently, we write explicitly the homogeneous dimension associated with the operator $\Ba$ in \eqref{bg}. In such case we have $m = 1$, $\alpha + 1 = \frac{1}{1-a}$, and so
\begin{equation}\label{hd}
Q = 1 + \frac{n}{1-a}.
\end{equation}
We have the following simple yet important fact whose verification is left to the reader.

\begin{lemma}\label{L:rhod}
Let $m=1$ and for any $a\in (-1,1)$ let $\alpha = \frac{a}{1-a}$. Then, for any $(x,y)\in \us$ we have
\begin{equation}\label{rd}
\rho_\alpha(x,h(|y|)) = h(d_e(x,y)),
\end{equation}
where the function $h$ is given by \eqref{h} above, and
\[d_e(x,y) = (|x|^2 + y^2)^{1/2}\]
indicates the standard Euclidean distance in $\R^{n+1}$. The equation \eqref{rd} implies in particular that
\begin{equation}\label{balls}
B_{\rho_\alpha}(h(r)) = \Phi(B_e(r)),\ \ \ \ \ \ \ \ r>0,
\end{equation}
where $\Phi$ is the diffeomorphism given by \eqref{phi} above, and
\[B_e(r) = \{(x,y)\in \R^{n+1}\mid d_e(x,y)<r\}.\]
is the Euclidean ball in $\R^{n+1}$ of radius $r$ and centered at the origin.
\end{lemma}

In view of \eqref{lab2}, \eqref{bg} it is clear that if we consider the function in $\us$ given by
\[
\tilde \Gamma(x,y) = \Gamma(x,h(y)),
\]
then $\tilde \Gamma$ is a solution of $L_a u = 0$ in $\us$. Notice that from \eqref{rd}, \eqref{Ga} we have
\[
\tilde \Gamma(x,y) = \frac{C_{\alpha}}{\ra(x,h(y))^{Q-2}} = \frac{C_{\alpha}}{h(d_e(x,y))^{Q-2}} = \frac{(1-a)^{(1-a)(Q-2)}C_{\alpha}}{d_e(x,y)^{(1-a)(Q-2)}} = \frac{\tilde C_a}{d_e(x,y)^{n+a-1}},
\]
where in the last equality we have used the above expression \eqref{hd} of the homogeneous dimension associated with the dilations \eqref{dilB} in $\R^{n+1}$. Now, the function $\tilde \Gamma(x,y)$ is precisely the fundamental solution of the Laplacian
\[
\Delta = \Delta_x + D_{yy} + \frac{a}{y} D_y
\]
in the fractional dimension
\begin{equation}\label{tQ}
\tilde Q = n+1+a,
\end{equation}
 found by Caffarelli and Silvestre in formula (2.1) in \cite{CS}. We note explicitly the following connection between $Q$ in \eqref{hd} and $\tilde Q$
 \begin{equation}\label{QtQ}
 (1-a)Q = \tilde Q - 2a.
 \end{equation}

We close this section with a simple proposition that unravels the connection between certain integrals on the pseudo-balls and spheres  $B_{\rho_\alpha}(r)$ and $S_{\rho_\alpha}(r)$ in the space of the variables $(x,z)$, and corresponding integrals on the Euclidean balls and spheres in the variables $(x,y)$.

\begin{prop}\label{P:integrals}
Let $f$ be a continuous function in the space $\R^{n+1}$ with the variables $(x,z)$, even in $z$, and let $\tilde f(x,y) = f(\Phi(x,y)) = f(x,h(|y|))$. Then, we have for every $r>0$
\begin{align}\label{si}
\int_{B_{\rho_\alpha}(h(r))} f(x,z) dx dz =   (1-a)^a \int_{B_e(r)} \tilde f(x,y) |y|^{-a} dx dy.
\end{align}
We also have
\begin{equation}\label{si2}
h'(r) \int_{S_{\rho_\alpha}(h(r))} \frac{f(x,z)}{|\nabla \rho_\alpha(x,z)|} d H_n(x,z) = (1-a)^a\int_{S_e(r)} \tilde f(x,y) |y|^{-a} dH_n(x,y).
\end{equation}
\end{prop}

\begin{proof}
Using \eqref{balls} we obtain
\[
\int_{B_{\rho_\alpha}(h(r))} f(x,z) dx dz =  \int_{\Phi(B_e(r))} f(x,z) dx dz.
\]
By the change of variable formula and \eqref{j}, the right-hand side of the latter equation equals that of \eqref{si}. To establish \eqref{si2} we observe that Federer's coarea formula (Cavalieri's principle) gives
\begin{equation}\label{coarea1}
\int_{B_{\rho_\alpha}(h(r))} f(x,z) dx dz = \int_0^{h(r)} \int_{S_{\rho_\alpha}(\tau)} \frac{f(x,z)}{|\nabla \rho_\alpha(x,z)|} d H_n(x,z) d\tau,
\end{equation}
where we have denoted by $H_n$ the $n$-dimensional Hausdorff measure in $\R^{n+1}$. Similarly, we find
\begin{equation}\label{coarea2}
 \int_{B_e(r)} \tilde f(x,y) |y|^{-a} dx dy =  \int_0^r \int_{S_e(t)} \tilde f(x,y) |y|^{-a} dH_n(x,y) dt.
 \end{equation}
Combining \eqref{si} with \eqref{coarea1} and \eqref{coarea2}, and differentiating with respect to $r$ we obtain
\eqref{si2}.
\end{proof}

\subsection{Almgren type monotonicity formulas for the operators $\Ba$, and their counterparts for the Caffarelli-Silvestre extension}\label{S:bg}

In order to keep the connection between the extension operator $L_a$ and the Baouendi operator $\Ba$, throughout this section we assume that $\alpha>0$. In the paper \cite{G} the first named author introduced a frequency function associated with $\Ba$, and proved that such frequency is monotone nondecreasing on solutions of $\Ba u = 0$, see Theorem \ref{T:almgrenBa} below. On the other hand, a version of the Almgren type monotonicity formula for the extension operator $L_a$ in \eqref{cs} above played an extensive role also in the mentioned recent work \cite{CSS} on the obstacle problem for the fractional Laplacian. In this section we gather some new monotonicity properties of the operators $\Ba$ that we use to establish some new results about the Caffarelli-Silvestre extension operator $L_a$ in \eqref{cs} above. We keep focusing on the situation in which $m=1$, and thus $\nabla_z = D_z$.

For $u, v\in \Gamma^1_\alpha(\R^{n+1})$, we define the $\alpha$-gradient of $u$ to be
\begin{equation*}
 \nabla_\alpha u= \left(|z|^\alpha \nabla_x u, D_z u\right)
\end{equation*}
and we set
\[
\langle\na u,\na v\rangle =  |z|^{2\alpha} \langle\nabla_x u,\nabla_x v\rangle + D_z u D_z v.
\]
The square of the length of $\nabla_\alpha u$ is
\begin{equation}\label{carre}
 |\nabla_\alpha u|^2= |z|^{2\alpha} |\nabla_x u|^2 + (D_z u)^2.
\end{equation}
The following lemma, collects the identities (2.12)-(2.14) in \cite{G}.

\begin{lemma}\label{l:gradalpharho}
Let $\ra$ be the pseudo-gauge in \eqref{ra} above. One has in $\R^{n+1}\setminus\{0\}$,
 \begin{equation}\label{nara}
  \psi_{\alpha} \overset{def}{=} |\na \ra|^2 = \frac{|z|^{2\alpha}}{\ra^{2\alpha}}.
 \end{equation}
Moreover, given a function $u$ one has
 \begin{equation}\label{nara2}
 \langle\na u,\na \ra\rangle = \frac{\Za u}{\ra} \psi_\alpha,
 \end{equation}
 where $Z_\alpha$ is the vector field in \eqref{Za} above.
 \end{lemma}

 Given a number $R>0$ and function $u\in \Gamma^1_\alpha(B_{\rho_\alpha}(R))$, we now define the Dirichlet integral, height, and frequency of $u$, respectively, as:
\begin{align}
 D(u,r) & = \int_{B_{\rho_\alpha}(r)}|\nabla_\alpha u|^2 dx dz,\ \ \ 0<r<R, \label{D}\\
 H(u,r) & = \int_{S_{\rho_\alpha}(r)} u^2\frac{\psi_\alpha}{|\nabla \ra|}\ dH_{n}\\
N(u,r) & = \frac{rD(u,r)}{H(u,r)}.
\end{align}

Suppose that $u$ is even in $z$ and consider the function $\tilde u(x,y) = u(\Phi(x,y)) = u(x,h(|y|))$, defined in $B_e(R)$ by \eqref{tu} above, and introduce now the following quantities
\begin{align}
\tilde D(\tilde u,r) & = \int_{B_e(r)}  |\nabla \tilde u|^2 |y|^a dx dy,\ \ \ 0<r<R,
 \label{tD}\\
\tilde H(\tilde u,r) & = \int_{S_{e}(r)} u^2 |y|^a  dH_{n},
\label{tH}\\
\tilde N(\tilde u,r) &=  \frac{r \tilde D(\tilde u,r)}{\tilde H(\tilde u,r)}.
\label{tN}
\end{align}
Here, $B_e(r)$ is the Euclidean ball of radius $r$ in $\R^{n+1}$ with center at the origin, and $S_e(r)=\partial B_e(r)$.
We have the following result.

\begin{lemma}\label{L:hth}
For every $0<r<R$ one has
\[
\tilde H(\tilde u,r) =  r^a H(u,h(r)).
\]
\end{lemma}

\begin{proof}

We choose $f = u^2 \psi_\alpha$  in \eqref{si2} above, obtaining
\begin{align*}
h'(r) H(u,h(r)) & = h'(r) \int_{S_{\rho_\alpha}(r)} u^2\frac{\psi_\alpha}{|\nabla \ra|}\ dH_{n}  = (1-a)^a \int_{S_{e}(r)} \tilde u^2 \tilde \psi_\alpha |y|^{-a}\ dH_{n}
\end{align*}
From \eqref{nara} and \eqref{rd} we obtain
\[
\tilde \psi_\alpha(x,y) = \frac{h(|y|)^{2\alpha}}{\ra^{2\alpha}(x,h(|y|))} =  \frac{h(|y|)^{2\alpha}}{h(d_e(x,y))^{2\alpha}}.
\]
Substituting in the above equation we find
\[
h(r)^{2\alpha} h'(r) H(u,h(r)) =  (1-a)^a \int_{S_{e}(r)} \tilde u^2  |y|^{-a} h(|y|)^{2\alpha}\ dH_{n}.
\]
Keeping in mind that $\alpha = \frac{a}{1-a}$, we obtain
\[
h(r)^{2\alpha} h'(r) = (1-a)^{-a} r^a,\ \ \ \ \ |y|^{-a} h(|y|)^{2\alpha} = \frac{|y|^a}{(1-a)^{2a}}.
 \]
Substitution in the latter equation yields the desired conclusion.

\end{proof}

Next, we have the following.

\begin{lemma}\label{L:dtd}
For every $0<r<R$ one has
\[
\tilde D(\tilde u,r) = (1-a)^{a}  D(u,h(r)).
\]
\end{lemma}

\begin{proof}
From formula \eqref{si} we obtain
\[
D(u,h(r)) = \int_{B_{\rho_\alpha}(h(r))} |\na u|^2 dx dz =  (1-a)^a \int_{B_e(r)} \widetilde{|\na u|^2} |y|^{-a} dx dy.
\]
From the chain rule one obtains
\[
D_z u(x,h(|y|)) = \frac{1}{h'(|y|)} D_y \tilde u(x,y),\ \ \ \ \ \nabla_x u(x,h(|y|)) = \nabla_x \tilde u(x,y).
\]
This gives
\begin{align*}
\widetilde{|\na u|^2}(x,y) & = |\na u|^2(x,h(|y|)) = D_z u(x,h(|y|))^2 + h(|y|)^{2\alpha} |\nabla_x u(x,h(|y|))|^2
\\
& =  \frac{1}{h'(|y|)^2} D_y \tilde u(x,y)^2 + h(|y|)^{2\alpha} |\nabla_x \tilde u(x,y)|^2 = \frac{|y|^{2a}}{(1-a)^{2a}} \left[D_y \tilde u(x,y)^2 + \nabla_x \tilde u(x,y)|^2\right]
\\
& = \frac{|y|^{2a}}{(1-a)^{2a}} |\nabla \tilde u(x,y)|^2.
\end{align*}
Inserting this formula in the right-hand side of the latter equation, we find
\[
D(u,h(r)) = (1-a)^{-a} \int_{B_e(r)} |\nabla \tilde u|^2 |y|^{a} dx dy = (1-a)^{-a} \tilde D(\tilde u,r).
\]
This gives the desired conclusion.

\end{proof}

Finally, combining Lemmas \ref{L:hth} and \ref{L:dtd}, we obtain the following notable result.

\begin{prop}\label{P:ntn}
For every $0<r<R$ one has
\[
\tilde N(\tilde u,r) = (1-a) N(u,h(r)).
\]
In particular, if either one of the limits $\tilde N(\tilde u,0^+) = \underset{r\to 0^+}{\lim} \tilde N(\tilde u,r)$, $N(u,0^+) = \underset{r\to 0^+}{\lim} N(u,r)$ exists, then also the other does and we have (see Proposition \ref{P:hom} above)
\[
\tilde N(\tilde u,0^+) = (1-a)  N(u,0^+).
\]
\end{prop}

We now recall the following Almgren type monotonicity formula for solutions of the operator $\Ba$ in \eqref{bg} established in \cite{G}.

\begin{thrm}\label{T:almgrenBa}
Let $u$ be a solution of $\Ba u = 0$ in $B_{\rho_\alpha}(r_0)$, and suppose that for no $r\in (0,r_0)$ we have $u\equiv 0$ in $B_{\rho_\alpha}(r)$. Then, the function $r\to N(u,r)$ is nondecreasing in $(0,r_0)$. Furthermore, $N(u,\cdot)\equiv \kappa$ if and only if $u$ is homogeneous of degree $\kappa$ with respect to the dilations \eqref{dilB} above.
\end{thrm}

If we combine Theorem \ref{T:almgrenBa} with \eqref{balls} in Lemma \ref{L:rhod} and with Proposition \ref{P:ntn}, we immediately obtain the following result for the operator $L_a$ in \eqref{cs} above.

\begin{thrm}\label{T:almgrenEO}
Let $\tilde u$ be a solution of the Caffarelli-Silvestre extension operator $L_a \tilde u = 0$ in $B_{e}(R_0)$ that is even in $y$, and suppose that for no $r\in (0,R_0)$ we have $\tilde u\equiv 0$ in $B_e(r)$. Then, the function $r\to \tilde N(\tilde u,r)$ is nondecreasing on $(0,R_0)$. Furthermore, $N(\tilde u,\cdot) \equiv \tilde \kappa$ if and only if $\tilde u$ is homogeneous of degree $\tilde \kappa$ with respect to the standard Euclidean dilations $\tilde \delta_\lambda(x,y) = (\la x,\la y)$.
\end{thrm}

We will soon return to Theorems \ref{T:almgrenEO} and \ref{T:almgrenBa} in the next sections in the more general setting of solutions to the thin obstacle problem for the operators $\Ba$ and $L_a$.

\section{Monotonicity formulas of Weiss and Monneau type: the case of zero obstacle}\label{S:op}

In the groundbreaking paper \cite{CSS} it was shown that if $u$ is the unique solution to the obstacle problem \eqref{pb} above for $s\in (0,1)$, then by extending $u$ to the upper half-space  $\R^{n+1}_+$, and then reflecting it evenly in the variable $y$, one obtains a solution $\tilde u$ to the following problem, equivalent to \eqref{pb}, where the parameter $a\in (-1,1)$ is related to $s$ by the formula $a = 1-2s$ (see \eqref{cs} above),
\begin{equation}\label{sb}
\begin{cases}
L_a \tilde u = \operatorname{div}(|y|^a \nabla \tilde u) = 0\ \ \ \ \ \ \ \ \text{in}\ \R^{n+1}_+\cup \R^{n+1}_-,
\\
\tilde u(x,0) \ge \vf(x),\ \ \ \ \ \ \ \ \ \ \ \ \ \ \ \ \text{for}\ x\in \Rn,
\\
\tilde u(x,-y) = \tilde u(x,y),\ \ \ \ \ \ \ \ \ \  \ \text{for}\ x\in \Rn, y\in \R,
\\
- \underset{y\to 0^+}{\lim} y^a D_y\tilde u(x,y) \ge 0, \ \ \ \ \  \text{for}\ x\in \Rn,
\\
\underset{y\to 0^+}{\lim} y^a D_y \tilde u(x,y) = 0, \ \ \ \ \ \ \ \text{on the set where}\ \tilde u(x,0) > \vf(x).
\end{cases}
\end{equation}
Using such equivalence between the nonlocal problem \eqref{pb} and the local problem \eqref{sb}, the authors proved one of their main results, namely the optimal $C^{1+s}$ regularity of the solution $u$ of \eqref{pb} (the almost optimal regularity had been previously established in \cite{S}). They also proved the regularity of the free boundary at regular points. One of the central tools in \cite{CSS} was an Almgren type truncated monotonicity formula for the problem \eqref{sb}. In the case when the obstacle $\vf \equiv 0$ truncation is not necessary and the  corresponding (pure) monotonicity formula, see Theorem 6.1 in \cite{CS}, generalizes Theorem \ref{T:almgrenEO} above to solutions of the obstacle problem \eqref{sb}.

As we have mentioned above,  in their 2007 paper \cite{CS} the authors were not aware of the 1993 work \cite{G} and the Almgren type monotonicity Theorem \ref{T:almgrenBa} above. The link between such result and Theorem 6.1 in \cite{CS} became apparent later, see Remark 3.2 in \cite{CSS}. For our subsequent objectives, we will now further discuss this aspect.

In $\R^{n+1}$ with variables $(x,z)$ consider the pseudo-ball $B_{\rho_\alpha}(r)  = \{(x,z)\in \R^{n+1}\mid \rho_\alpha(x,z) <r\}$ centered at the origin with radius $r>0$, see \eqref{ra} above, and denote by $B_{\rho_\alpha}^+(r) = \{(x,z)\in \us \mid z>0\}\cap B_{\rho_\alpha}(r)$ the upper pseudo-ball, with analogous definition for $B_{\rho_\alpha}(r)^-$. We will denote by $B_{\rho_\alpha}(r)' = B_{\rho_\alpha}(r)\cap \{(x,0)\in \R^{n+1}\}$ the thin ball in $\Rn$. If a function $\tilde u$ solves the problem \eqref{sb} above in a ball $B_e(R_0)$, then, after an even reflection in $z$,
the function $u$ defined by \eqref{tu} above verifies
\begin{equation}\label{sb2}
\begin{cases}
\Ba u = D_{zz} u + z^{2\alpha} \Delta_x u = 0\ \ \ \ \ \  \text{in}\ B_{\rho_\alpha}(r_0)^+ \cup B_{\rho_\alpha}(r_0)^-,
\\
u \ge \vf,\ \ \ \ \ \ \ \ \ \ \ \ \ \ \ \ \ \ \ \ \ \ \ \ \ \ \ \ \ \ \ \  \text{in}\ B_{\rho_\alpha}(r_0)',
\\
u(x,-z) = u(x,z),\ \ \ \ \ \ \ \ \ \ \ \ \ \ \ \ \text{for}\ x\in \Rn, z\in \R
\\
- D_z^+ u + D_z^- u  \ge 0, \ \ \ \ \ \ \ \ \ \ \ \ \ \ \ \ \text{in}\ B_{\rho_\alpha}(r_0)',
\\
u (- D_z^+ u + D_z^- u)= 0, \ \ \ \ \ \ \ \ \ \ \ \ \text{in}\ B_{\rho_\alpha}(r_0)',
\end{cases}
\end{equation}
where, $r_0 = h^{-1}(R_0)$. This is a thin obstacle  problem for the degenerate elliptic operator $\Ba$ with obstacle $\vf(x)$ concentrated on the thin set $B_{\rho_\alpha}(r_0)'$.

\subsection{Almgren type monotonicity formulas}\label{SS:almgren}
We will use the connection between the problems \eqref{sb} and \eqref{sb2} to study the singular part of the free boundary in the problem \eqref{pb} above. With this objective in mind we now recall some facts that will be useful in the sequel. Similarly to what we did in Section \ref{S:bg}, we assume here that $0<\alpha < \infty$ to keep the connection between $L_a$ and $\Ba$. Subsequently, we will free ourselves of this restriction and consider the full range $-\frac 12 < \alpha < \infty$. This corresponds to the full range $0<s<1$ for the operator $(-\Delta)^s$.

The following proposition follows from an easy adaptation of the proof of formula (4.36) in \cite{G}.
\begin{prop}\label{P:diri}
 Let $u$ be a solution of \eqref{sb2} with $\vf\equiv 0$. Then, for every $0<r<r_0$,
 \begin{align}\label{Ds}
  D(u,r)&=\int_{S_{\rho_\alpha}(r)}u\left(\frac{Z_\alpha u}{r}\right)\frac{\psi_\alpha}{|\nabla \ra|}dH_{n}.
 \end{align}
\end{prop}

\begin{proof}
For every $\ve >0$ we introduce the sets
\begin{align}\label{es}
B^+_{\rho_\alpha}(r,\ve) & = B_{\rho_\alpha}(r) \cap \{(x,z)\in \R^{n+1}\mid z>\ve\},
\\
B^-_{\rho_\alpha}(r,\ve) & = B_{\rho_\alpha}(r) \cap \{(x,z)\in \R^{n+1}\mid z<\ve\},
\notag\\
L^{\pm}(r,\ve) & = B_{\rho_\alpha}(r) \cap \{(x,z)\in \R^{n+1}\mid z = \pm \ve\}.
\notag
\end{align}
Since for the variational solution $u$ of \eqref{sb2} we have $|\na u|\in L^2(B_{\rho_\alpha}(r))$, by Lebesgue dominated convergence we have
\begin{align*}
D(u,r) = \int_{{B_{\rho_\alpha}(r)}} |\na u|^2   = \underset{\ve\to 0^+}{\lim} \left\{\int_{{B^+_{\rho_\alpha}(r,\ve)}} |\na u|^2 +  \int_{{B^-_{\rho_\alpha}(r,\ve)}} |\na u|^2 \right\}.
\end{align*}
In $B^\pm_{\rho_\alpha}(r,\ve)$ we have from the first equation in \eqref{sb2}
\[
\Ba(u^2/2) = u \Ba u + |\na u|^2 = |\na u|^2.
\]
We thus find
\begin{align*}
& \int_{{B_{\rho_\alpha}(r)}} |\na u|^2  = \underset{\ve\to 0^+}{\lim} \left\{\sum_{j=1}^{n+1} \int_{{B^+_{\rho_\alpha}(r,\ve)}} X_j X_j (u^2/2) + \sum_{j=1}^{n+1}  \int_{{B^-_{\rho_\alpha}(r,\ve)}} X_j X_j (u^2/2) \right\}
\\
& = \underset{\ve\to 0^+}{\lim} \left\{\sum_{j=1}^{n+1} \int_{{\p B^+_{\rho_\alpha}(r,\ve)}} u \langle X_j,\nu\rangle  X_j u + \sum_{j=1}^{n+1}  \int_{{\p B^-_{\rho_\alpha}(r,\ve)}} u \langle X_j,\nu\rangle  X_j u  \right\}
\\
& = \underset{\ve\to 0^+}{\lim} \left\{\int_{{S^+_{\rho_\alpha}(r,\ve)}} u \langle\na u,\na \rho_\alpha\rangle \frac{dH_{n}}{|\nabla \ra|}+ \int_{{S^-_{\rho_\alpha}(r,\ve)}} u \langle\na u,\na \rho_\alpha\rangle \frac{dH_{n}}{|\nabla \ra|}\right\}
\\
& + \underset{\ve\to 0^+}{\lim} \left\{- \int_{{L^+(r,\ve)}} u D_z u dx + \int_{{L^-(r,\ve)}} u D_z u dx\right\}
\\
& = \int_{{S_{\rho_\alpha}(r)}} u \langle\na u,\na \rho_\alpha\rangle \frac{dH_{n}}{|\nabla \ra|} + \int_{{B'_{\rho_\alpha}(r)}} u(-D^+_z u + D^-_z u) dx
\\
& = \int_{S_{\rho_\alpha}(r)} u\left(\frac{Z_\alpha u}{r}\right)\frac{\psi_\alpha}{|\nabla \ra|}dH_{n} + \int_{{B'_{\rho_\alpha}(r)}} u(-D^+_z u + D^-_z u) dx,
\end{align*}
where in the last equality we have used \eqref{nara2}. By the last equation in \eqref{sb2} we now have in $B'_{\rho_\alpha}(r)$
\[
u(-D^+_z u + D^-_z u)  = 0.
\]
In conclusion, we have proved that
\[
\int_{S_{\rho_\alpha}(r)} w\left(\frac{Z_\alpha w}{r}\right)\frac{\psi_\alpha}{|\nabla \ra|}dH_{n}  = \int_{{B_{\rho_\alpha}(r)}} |\na u|^2,
\]
which is \eqref{Ds} above.

\end{proof}

The proof of the next result follows that of Lemma 4.1 in \cite{G} and from Proposition \ref{P:diri} above.

\begin{lemma}[First variation of the height]\label{l:hprime}
 Let $u$ be a solution of \eqref{sb2} with $\vf\equiv 0$. Then,
 \[
  H'(u,r)=\frac{Q-1}{r}H(u,r)+2D(u,r),\hspace{1cm}\textrm{ for }r\in (0,r_0).
 \]
\end{lemma}

By using the coarea formula, we have the following for the derivative of $D(u,r)$:
\begin{equation}\label{D'u}
 D'(u,r)=\int_{S_{\rho_\alpha}(r)} |\nabla_\alpha u|^2\frac{dH_{n}}{|\nabla\ra|}.
\end{equation}

The next result, which is inspired to Corollary 2.3 in \cite{G}, provides a basic property of solutions of \eqref{sb2}.

\begin{lemma}[First variation of the energy]\label{l:dprime}
 Let $u$ be a solution of \eqref{sb2} with $\vf\equiv 0$. Then,
 \begin{equation}\label{D'u2}
  D'(u,r)=2\int_{{S_{\rho_\alpha}(r)}}\left(\frac{Z_\alpha u}{r}\right)^2\frac{\psi_\alpha}{|\nabla \ra|}dH_{n}+\frac{Q-2}{r}D(u,r).
 \end{equation}
\end{lemma}

\begin{proof}
For every $\ve>0$ we consider the sets in \eqref{es} above. We have
\begin{align*}
& \int_{B^+_{\rho_\alpha}(r,\ve)} \operatorname{div}(|\na u|^2 Z_\alpha) = Q \int_{B^+_{\rho_\alpha}(r,\ve)} |\na u|^2 +  \int_{B^+_{\rho_\alpha}(r,\ve)} Z_\alpha(|\na u|^2)
\\
& =  Q \int_{B^+_{\rho_\alpha}(r,\ve)} |\na u|^2 + 2 \sum_{j=1}^{n+1} \int_{B^+_{\rho_\alpha}(r,\ve)} Z_\alpha(X_j u) X_j u
\\
& =  Q \int_{B^+_{\rho_\alpha}(r,\ve)} |\na u|^2 + 2 \sum_{j=1}^{n+1} \int_{B^+_{\rho_\alpha}(r,\ve)} X_j(Z_\alpha u) X_j u -  2 \sum_{j=1}^{n+1} \int_{B^+_{\rho_\alpha}(r,\ve)} [X_j,Z_\alpha] u X_j u
\\
& =  Q \int_{B^+_{\rho_\alpha}(r,\ve)} |\na u|^2 + 2 \sum_{j=1}^{n+1} \int_{\p B^+_{\rho_\alpha}(r,\ve)} \langle X_j,\nu\rangle Z_\alpha u X_j u
\\
& - 2 \int_{B^+_{\rho_\alpha}(r,\ve)} Z_\alpha u \Ba u -  2 \int_{B^+_{\rho_\alpha}(r,\ve)} |\na u|^2
\\
& =  (Q-2) \int_{B^+_{\rho_\alpha}(r,\ve)} |\na u|^2 + 2 \sum_{j=1}^{n+1} \int_{\p B^+_{\rho_\alpha}(r,\ve)} \langle X_j,\nu\rangle Z_\alpha u X_j u,
\end{align*}
where in the second to the last equality we have used integration by parts and Lemma \ref{L:comm}, whereas in the last we have used the first equation in \eqref{sb2}. Now we observe that $\p B^+_{\rho_\alpha}(r,\ve) = S^+_{\rho_\alpha}(r,\ve) \cup L^+(r,\ve)$, and that we have $\nu = \frac{\nabla \rho_\alpha}{|\na \rho_\alpha|}$ on $S^+_{\rho_\alpha}(r,\ve)$, whereas we have $\nu = - e_{n+1}$ on $L^+(r,\ve)$.
This gives
\begin{align*}
& \sum_{j=1}^{n+1} \int_{\p B^+_{\rho_\alpha}(r,\ve)} \langle X_j,\nu\rangle Z_\alpha u X_j u = \sum_{j=1}^{n+1} \int_{S^+_{\rho_\alpha}(r,\ve)} \langle X_j,\nu\rangle Z_\alpha u X_j u
\\
& + \sum_{j=1}^{n+1} \int_{L^+(r,\ve)} \langle X_j,\nu\rangle Z_\alpha u X_j u =  \int_{S^+_{\rho_\alpha}(r,\ve)}  Z_\alpha u \frac{\langle\na u,\na \rho_\alpha\rangle}{|\na \rho_\alpha|}
\\
&  - \int_{L^+(r,\ve)} Z_\alpha u D_z u dx = \frac 1r \int_{S^+_{\rho_\alpha}(r,\ve)}  (Z_\alpha u)^2 \frac{\psi_\alpha}{|\na \rho_\alpha|}
- \int_{L^+(r,\ve)} Z_\alpha u D_z u dx,
\end{align*}
where in the last equality we have used \eqref{nara2} above. We conclude that
\begin{align*}
& \int_{B^+_{\rho_\alpha}(r,\ve)} \operatorname{div}(|\na u|^2 Z_\alpha) =  (Q - 2) \int_{B^+_{\rho_\alpha}(r,\ve)} |\na u|^2
\\
& + \frac 2r \int_{S^+_{\rho_\alpha}(r,\ve)}  (Z_\alpha u)^2 \frac{\psi_\alpha}{|\na \rho_\alpha|}
- 2 \int_{L^+(r,\ve)} Z_\alpha u D_z u dx.
\end{align*}
Recalling that from \eqref{Za} above we have $Z_\alpha u = (1+\alpha) \langle x,\nabla_x u\rangle + z D_zu$, we conclude
\begin{align}\label{Dup}
& \int_{B^+_{\rho_\alpha}(r,\ve)} \operatorname{div}(|\na u|^2 Z_\alpha) =  (Q - 2) \int_{B^+_{\rho_\alpha}(r,\ve)} |\na u|^2 + \frac 2r \int_{S^+_{\rho_\alpha}(r,\ve)}  (Z_\alpha u)^2 \frac{\psi_\alpha}{|\na \rho_\alpha|}
\\
&
- 2 (1+\alpha) \int_{L^+(r,\ve)}  \langle x,\nabla_x u\rangle  D_z u dx - 2 \ve \int_{L^+(r,\ve)} (D_z u)^2 dx.
\notag
\end{align}
On the other hand, the divergence theorem gives
\begin{align*}
& \int_{B^+_{\rho_\alpha}(r,\ve)} \operatorname{div}(|\na u|^2 Z_\alpha) = \int_{\p B^+_{\rho_\alpha}(r,\ve)} |\na u|^2 \langle Z_\alpha,\nu\rangle = \int_{S^+_{\rho_\alpha}(r,\ve)} |\na u|^2 \langle Z_\alpha,\nu\rangle
\\
& + \int_{L^+(r,\ve)} |\na u|^2 \langle Z_\alpha,\nu\rangle = \int_{S^+_{\rho_\alpha}(r,\ve)} |\na u|^2 \frac{Z_\alpha \rho_\alpha}{|\na \rho_\alpha|}  -  \int_{L^+(r,\ve)} |\na u|^2 \langle Z_\alpha,e_{n+1}\rangle
\end{align*}
Since by \eqref{ra} the function $\rho_\alpha$ is homogeneous of degree one with respect to the non-isotropic dilations \eqref{dilB}, we have $Z_\alpha \rho_\alpha = \rho_\alpha = r$ on $S_{\rho_\alpha}(r)$. Keeping in mind that $\langle Z_\alpha,e_{n+1}\rangle = z = \ve$ on $L^+(r,\ve)$, and recalling that \eqref{carre} gives $|\nabla_\alpha u|^2=(D_z u)^2+ |z|^{2\alpha} |\nabla_x u|^2$, we thus obtain
\begin{align}\label{Dup2}
& \int_{B^+_{\rho_\alpha}(r,\ve)} \operatorname{div}(|\na u|^2 Z_\alpha) =  r \int_{S^+_{\rho_\alpha}(r,\ve)} |\na u|^2 \frac{dH_n}{|\na \rho_\alpha|}
\\
& -  \ve \int_{L^+(r,\ve)} (D_z u)^2 -  \ve^{1+2\alpha} \int_{L^+(r,\ve)} |\nabla_x u|^2.
\notag
\end{align}
Letting $\ve\to 0^+$ in \eqref{Dup} with \eqref{Dup2}, and combining the resulting equations we find
\begin{align}\label{Dup3}
& r \int_{S^+_{\rho_\alpha}(r)} |\na u|^2 \frac{dH_n}{|\na \rho_\alpha|} = (Q - 2) \int_{B^+_{\rho_\alpha}(r)} |\na u|^2 + \frac 2r \int_{S^+_{\rho_\alpha}(r)}  (Z_\alpha u)^2 \frac{\psi_\alpha}{|\na \rho_\alpha|}
\\
&
- 2 (1+\alpha) \int_{B'_{\rho_\alpha}(r)}  \langle x,\nabla_x u\rangle  D^+_z u\ dx.
\notag
\end{align}
Arguing in a similar way on the set $B^-_{\rho_\alpha}(r,\ve)$ we find
\begin{align}\label{Dup4}
& r \int_{S^-_{\rho_\alpha}(r)} |\na u|^2 \frac{dH_n}{|\na \rho_\alpha|} = (Q - 2) \int_{B^-_{\rho_\alpha}(r)} |\na u|^2 + \frac 2r \int_{S^-_{\rho_\alpha}(r)}  (Z_\alpha u)^2 \frac{\psi_\alpha}{|\na \rho_\alpha|}
\\
&
+ 2 (1+\alpha) \int_{B'_{\rho_\alpha}(r)}  \langle x,\nabla_x u\rangle  D^-_z u\ dx.
\notag
\end{align}
If we now combine \eqref{Dup3} with \eqref{Dup4} we finally obtain
\begin{align}\label{Dup5}
& r \int_{S_{\rho_\alpha}(r)} |\na u|^2 \frac{dH_n}{|\na \rho_\alpha|} = (Q - 2) \int_{B_{\rho_\alpha}(r)} |\na u|^2 + \frac 2r \int_{S_{\rho_\alpha}(r)}  (Z_\alpha u)^2 \frac{\psi_\alpha}{|\na \rho_\alpha|}
\\
&
+ 2 (1+\alpha) \int_{B'_{\rho_\alpha}(r)}  \langle x,\nabla_x u\rangle (-D_z^+ u + D^-_z u)\ dx.
\notag
\end{align}
We now evaluate the integral on the thin ball  $B'_{\rho_\alpha}(r)$ in the right-hand side of \eqref{Dup5}. We have
\begin{align*}
& \int_{B'_{\rho_\alpha}(r)}  \langle x,\nabla_x u\rangle (-D_z^+ u + D^-_z u)\ dx = \int_{B'_{\rho_\alpha}(r)\cap \{u>0\}}  \langle x,\nabla_x u\rangle (-D_z^+ u + D^-_z u)\ dx
\\
& + \int_{B'_{\rho_\alpha}(r)\setminus \{u>0\}}  \langle x,\nabla_x u\rangle (-D_z^+ u + D^-_z u)\ dx =  0.
\end{align*}
The former integral in the right-hand side vanishes since on the set $B'_{\rho_\alpha}(r)\cap \{u>0\}$ we must have $-D_z^+ u + D^-_z u = 0$ by the fifth equation in \eqref{sb2}. At every point $(x,0)\in B'_{\rho_\alpha}(r)\setminus \{u>0\}$ we must instead have $\nabla_x u(x,0) = 0$, and thus the latter integral in the right-hand side vanishes as well. Keeping in mind \eqref{D'u}, we see that we have finally proved \eqref{D'u2}, thus completing the proof of the lemma.

\end{proof}

We are now able to extend Theorem \ref{T:almgrenBa} above to solutions of the thin obstacle problem \eqref{sb2}.

\begin{thrm}[Almgren type monotonicity for $\Ba$]\label{T:almgrenbgo}
Let $u$ be a solution of \eqref{sb2} with $\vf\equiv 0$. Then, $r\to N(u,r)$ is monotone non-decreasing in $(0,r_0)$. Furthermore, $N(u,r) \equiv \kappa$ in $(0,r_0)$ if and only if $u$ is homogeneous of degree $\kappa$ in $B_{\rho_\alpha}(r_0)$ with respect to the non-isotropic dilations \eqref{dilB} above.
\end{thrm}

\begin{proof}
Using Lemmas \ref{l:hprime}, \ref{l:dprime} and Proposition \ref{P:diri}, we find
\begin{equation}\label{logN}
\frac{d}{dr} \log N(r) = 2 \frac{\int_{S_{\rho_\alpha}(r)}\left(\frac{Z_\alpha u}{r}\right)^2\frac{\psi_\alpha}{|\nabla \ra|}dH_{n}}{\int_{S_{\rho_\alpha}(r)}u\left(\frac{Z_\alpha u}{r}\right)\frac{\psi_\alpha}{|\nabla \ra|}dH_{n}} - 2 \frac{\int_{S_{\rho_\alpha}(r)}u\left(\frac{Z_\alpha u}{r}\right)\frac{\psi_\alpha}{|\nabla \ra|}dH_{n}}{\int_{S_{\rho_\alpha}(r)} u^2\frac{\psi_\alpha}{|\nabla \ra|}dH_{n}}\ge 0,
\end{equation}
where in the last inequality we have used Cauchy-Schwarz inequality. The implication that if $u$ is homogeneous of degree $\kappa$, then $N(u,r) \equiv \kappa$ follows directly from Proposition  \ref{P:diri} and \eqref{Zak}. The reverse implication is subtler since it uses the full strength of the monotonicity in \eqref{logN} (for a different proof of such implication see Remark \ref{R:sh} below). If $N(u,r) \equiv \kappa$, then $\frac{d}{dr} \log N(r)  \equiv 0$ and thus there is equality in the inequality in \eqref{logN}. This means that there is equality in the Cauchy-Schwarz inequality, hence for every $r\in (0,r_0)$ there exists $\gamma(r)$ such that $Z_\alpha u = \gamma(r) u$ on $S_{\rho_\alpha}(r)$. Inserting this information in the identity in Proposition  \ref{P:diri} we conclude that
\[
D(r) = \frac{\gamma(r)}{r} \int_{S_{\rho_\alpha}(r)}u^2 \frac{\psi_\alpha}{|\nabla \ra|}dH_{n} = \frac{\gamma(r)}{r} H(r).
\]
This implies $\kappa \equiv N(r) = \gamma(r)$ for every $r\in (0,r_0)$, and therefore $Z_\alpha u = \kappa u$. According to \eqref{Zak}
this implies that $u$ is homogeneous of degree $\kappa$ with respect to the non-isotropic dilations \eqref{dilB} above.

\end{proof}

Combining Theorem \ref{T:almgrenbgo} with Proposition \ref{P:ntn} we obtain the Caffarelli-Silvestre monotonicity theorem for solutions of \eqref{sb}.

\begin{thrm}[Almgren type monotonicity for $L_a, -1<a<1$]\label{T:almgrenext}
Let $\tilde u$ be a solution of \eqref{sb} in $B_e(R_0)$ with $\vf\equiv 0$. Then, $r\to \tilde N(\tilde u,r)$ is monotone non-decreasing in $(0,R_0)$. Furthermore, $\tilde N(\tilde u,r) \equiv \tilde \kappa$ in $(0,R_0)$ if and only if $\tilde u$ is homogeneous of degree $\tilde \kappa$ in $B_e(R_0)$ with respect to the standard Euclidean dilations.
\end{thrm}

\begin{proof}
Assume first that $0\le a <1$, and thus $0 \le \alpha = \frac{a}{1-a}<\infty$. Let $\tilde u$ be as in the statement of the theorem and consider the function $u(x,z)$ defined by \eqref{tu}. With $r_0 = h^{-1}(R_0)$, such function satisfies the problem \eqref{sb2} above in the pseudo-ball $B_{\rho_\alpha}(r_0)$ with obstacle $\vf \equiv 0$. From Theorem \ref{T:almgrenbgo} we infer that $r\to N(u,h(r))$ is non-decreasing in $(0,R_0)$. The desired conclusion now follows from Proposition \ref{P:ntn}. In the remaining range $-1<a<0$, corresponding to $-\frac 12<\alpha<0$, the proof can be obtained either by adapting the regularization procedure for the operator $\Ba$ as in in \cite{G}, or working directly on the operator $L_a$ as in \cite{CS}.

\end{proof}

\subsection{Weiss type monotonicity}\label{SS:weiss}
In this section we  establish a one-parameter family of monotonicity formulas similar to those in Theorem~1.4.1 in \cite{GP}. Since we want to maintain the connection between $L_a$ and $\Ba$, similarly to the way we have obtained Theorem \ref{T:almgrenext}, we first consider the range $0\le \alpha<\infty$ for $\Ba$.

\begin{thrm}[Weiss type monotonicity for $\Ba, 0\le \alpha <\infty$]\label{T:weissB}
 Let $u$ be a solution of \eqref{sb2} with $\vf\equiv 0$. For every $\kappa> 0$ we define for $0<r<r_0$
 \begin{equation}\label{WBa}
  \mathcal{W}_\kappa(u,r)=\frac{1}{r^{Q-2+2\kappa}}D(u,r)-\frac{\kappa}{r^{Q-1+2\kappa}}H(u,r),
 \end{equation}
 where $Q = n(\alpha +1) + 1$ (see \eqref{hd} above). Then,
 \begin{equation}\label{W'Ba}
  \frac{d}{dr}\mathcal{W}_\kappa(u,r)=\frac{2}{r^{Q+2\kappa}}\int_{{S_{\rho_\alpha}(r)}}(Z_\alpha u-\kappa u)^2\frac{\psi_\alpha}{|\nabla \ra|}dH_{n}.
 \end{equation}
 Consequently, $r\mapsto \mathcal{W}_\kappa(u,r)$ is a non-decreasing function in $(0,r_0)$, and $\mathcal{W}_\kappa(u,\cdot)$ is constant if and only if $u$ is homogeneous of degree $\kappa$ in $B_{\rho_\alpha}(r_0)$.
\end{thrm}

\begin{proof}

With Proposition \ref{P:diri}, Lemmas \ref{l:hprime} and \ref{l:dprime} in hands, the proof of Theorem \ref{T:weissB} proceeds as follows.
Differentiating \eqref{WBa} we obtain
\begin{align*}
 \frac{d}{dr}\mathcal{W}_\kappa(u,r) & = \frac{1}{r^{Q-1+2\kappa}} \left\{-(Q-2+2\kappa) D(u,r) + r D'(u,r) + \frac{\kappa(Q-1+2\kappa)}{r} H(u,r) - \kappa H'(u,r)\right\}
\\
& = \frac{1}{r^{Q-1+2\kappa}} \bigg\{-(Q-2+2\kappa) D(u,r) + r \left[\frac{2}{r^2} \int_{S_{\rho_\alpha}(r)} (Z_\alpha u)^2 \frac{\psi_\alpha}{|\nabla \ra|}dH_{n} + \frac{Q-2}{r} D(u,r)\right]
\\
& + \frac{\kappa(Q-1+2\kappa)}{r} H(u,r) - \kappa \left[\frac{Q-1}{r}H(u,r)+2D(u,r)\right]\bigg\}
\\
& = \frac{2}{r^{Q+2\kappa}}  \left\{\int_{S_{\rho_\alpha}(r)} (Z_\alpha u)^2 \frac{\psi_\alpha}{|\nabla \ra|}dH_{n}  + \int_{S_{\rho_\alpha}(r)} u^2 \frac{\psi_\alpha}{|\nabla \ra|}dH_{n} - 2\kappa \int_{S_{\rho_\alpha}(r)} u Z_\alpha u \frac{\psi_\alpha}{|\nabla \ra|}dH_{n}\right\}
\\
& = \frac{2}{r^{Q+2\kappa}}\int_{{S_{\rho_\alpha}(r)}}(Z_\alpha u-\kappa u)^2\frac{\psi_\alpha}{|\nabla \ra|}dH_{n}.
 \end{align*}

\end{proof}

\begin{rmrk}\label{R:sh}
We note that the proof of the last statement in Theorem \ref{T:almgrenbgo} can also be derived from Theorem \ref{T:weissB}. In fact, it suffices to  write \eqref{WBa} as follows
\begin{equation}\label{WkBa}
\mathcal W_\kappa(u,r) = \frac{H(u,r)}{r^{Q-1+2\kappa}} \big(N(u,r) - \kappa\big).
\end{equation}
By the hypothesis that $N(u,r) = \kappa$ for $0<r<r_0$, we see that $\mathcal W_\kappa(u,\cdot) =0$ in $(0,r_0)$, and therefore $\frac{d}{dr} \mathcal W_\kappa(u,\cdot) = 0$. In view of \eqref{W'Ba} this gives $\Za u = \kappa u$ in $B_{\rho_\alpha}(r_0)$.
\end{rmrk}

Combining Theorem \ref{T:weissB} with Lemmas \ref{L:hth} and \ref{L:dtd} we obtain the following one-parameter family of monotonicity formulas.

\begin{thrm}[Weiss type monotonicity for $L_a, -1<a<1$]\label{T:weissext}
Let $\tilde u$ be a solution of \eqref{sb} in $B_e(R_0)$ with $\vf\equiv 0$. For every $\tilde \kappa > 0$ and $0<r<R_0$ we define
\begin{equation}\label{tW}
\tilde{\mathcal W}_{\tilde \kappa}(\tilde u,r) = \frac{1}{r^{\tilde Q-2+2\tilde \kappa}}\tilde D(\tilde u,r)-\frac{\tilde \kappa}{r^{\tilde Q-1+2\tilde \kappa}}\tilde H(\tilde u,r),
 \end{equation}
where $\tilde Q =  n+1+a$ (see \eqref{tQ} above). Then,
 \begin{equation}\label{tW'}
  \frac{d}{dr}\tilde{\mathcal W}_{\tilde \kappa}(\tilde u,r)=\frac{2}{r^{\tilde Q+2\tilde \kappa}}\int_{{S_{e}(r)}}(\tilde Z \tilde u-\tilde \kappa \tilde u)^2 |y|^a dH_{n}.
 \end{equation}
 Consequently, $r\mapsto \mathcal{\tilde W}_{\tilde \kappa}(\tilde u,r)$ is a non-decreasing function in $(0,r_0)$, and $\tilde{\mathcal W}_{\tilde \kappa}(\tilde u,\cdot)$ is constant if and only if $\tilde u$ is a standard homogeneous function of degree $\tilde \kappa$ in $B_{e}(R_0)$.
\end{thrm}

\begin{proof}
Suppose first that $0\le a<1$, so that $0\le \alpha <\infty$. Let $\tilde u$ be as in the statement of the theorem and consider the function $u(x,z)$ defined by \eqref{tu}. The even extension in $z$ of such function satisfies the problem \eqref{sb2} above with obstacle $\vf \equiv 0$. In view of Theorem \ref{T:weissB}
the functional in \eqref{WBa} is monotonically increasing, and therefore such is also $r\to  \mathcal{W}_\kappa(u,h(r))$, where $h$ is as in \eqref{h} above and $\kappa = \tilde \kappa/(1-a)$, see \eqref{tk} above. Now, using \eqref{h}, \eqref{hd}, \eqref{tQ}, and Lemmas \ref{L:hth} and \ref{L:dtd} we find
\begin{align}\label{Ws}
\mathcal{W}_\kappa(u,h(r)) & =\frac{1}{h(r)^{Q-2+2\kappa}} D(u,h(r)) -\frac{\kappa}{h(r)^{Q-1+2\kappa}}H(u,h(r))
\\
& = \frac{(1-a)^{(1-a)(Q-2+2\kappa)}}{r^{(1-a)(Q-2+2\kappa)}} \frac{\tilde D(\tilde u,r)}{(1-a)^{a}} -  \kappa \frac{(1-a)^{(1-a)(Q-1+2\kappa)}}{r^{(1-a)(Q-1+2\kappa)}}\frac{\tilde H(\tilde u,r)}{r^a}
\notag\\
& = (1-a)^{\tilde Q-a -2+2\tilde \kappa}\left\{ \frac{1}{r^{\tilde Q-2+2\tilde \kappa}} \tilde D(\tilde u,r) - \frac{\tilde \kappa}{r^{\tilde Q-1+2\tilde \kappa}} \tilde H(\tilde u,r)\right\}
\notag\\
& = (1-a)^{\tilde Q-a -2+2\tilde \kappa} \tilde{\mathcal W}_{\tilde \kappa}(\tilde u,r).
\notag
\end{align}
Since we know that $r\to \mathcal W_\kappa(u,h(r))$ is monotone non-decreasing, the latter equation already tells us that also $r\to \tilde{\mathcal W}_{\tilde \kappa}(\tilde u,r)$ is monotone non-decreasing. Furthermore, the chain rule and \eqref{W'Ba} give
\begin{align*}
\frac{d}{dr} \tilde{\mathcal W}_{\tilde \kappa}(\tilde u,r) & = (1-a)^{-\tilde Q+a +2-2\tilde \kappa} h'(r) \left(\frac{d}{d\tau}\mathcal{W}_\kappa(u,\cdot)\right)(h(r))
\\
& = (1-a)^{-\tilde Q+a +2-2\tilde \kappa} \frac{2}{h(r)^{Q+2\kappa}} h'(r)  \int_{{S_{\rho_\alpha}(h(r))}}(Z_\alpha u-\kappa u)^2\frac{\psi_\alpha}{|\nabla \ra|}dH_{n}.
 \end{align*}
 We now use \eqref{si2} and argue as in the proof of Lemma \ref{L:hth} to find
 \begin{align*}
& h'(r)  \int_{{S_{\rho_\alpha}(h(r))}}(Z_\alpha u-\kappa u)^2\frac{\psi_\alpha}{|\nabla \ra|}dH_{n} =
(1-a)^a \int_{{S_{e}(r)}}(\widetilde{Z_\alpha u} -\kappa \tilde u)^2 \tilde \psi_\alpha |y|^{-a} dH_{n}
\\
& = \frac{(1-a)^a}{(1-a)^{2a} h(r)^{2\alpha}} \int_{{S_{e}(r)}}((\alpha + 1) \tilde Z \tilde u -\kappa \tilde u)^2   |y|^{a} dH_{n}
\\
& = \frac{1}{(1-a)^{a + 2} h(r)^{2\alpha}} \int_{{S_{e}(r)}}(\tilde Z \tilde u -\tilde \kappa \tilde u)^2  |y|^{a} dH_{n}
\end{align*}
where in the second to the last equality we have used \eqref{gd}. We conclude that
\begin{align*}
\frac{d}{dr} \tilde{\mathcal W}_{\tilde \kappa}(\tilde u,r) & = (1-a)^{-\tilde Q+a +2-2\tilde \kappa} \frac{2}{(1-a)^{a + 2} h(r)^{Q + 2 \kappa + 2\alpha}} \int_{{S_{e}(r)}}(\tilde Z \tilde u -\tilde \kappa \tilde u)^2 |y|^{a} dH_{n}
\\
& = (1-a)^{-\tilde Q-2\tilde \kappa} \frac{2 (1-a)^{(1-a)(Q+2\kappa +2a/(1-a))}}{r^{(1-a)(Q + 2 \kappa + 2\alpha)}} \int_{{S_{e}(r)}}(\tilde Z \tilde u -\tilde \kappa \tilde u)^2 |y|^{a} dH_{n}
\\
& = \frac{2}{r^{\tilde Q + 2 \tilde\kappa}} \int_{{S_{e}(r)}}(\tilde Z \tilde u -\tilde \kappa \tilde u)^2 |y|^{a} dH_{n},
\end{align*}
which proves \eqref{tW'}. The last part of the theorem follows immediately from \eqref{tW'}.

The proof of the theorem in the remaining range $-1<a<0$, for which $-\frac 12 <\alpha <0$, now follows easily by mimicking the above proof of Theorem \ref{T:weissB}. What we mean by this is that, once we have the correct guess \eqref{tW} of the Weiss functional for $L_a$ in the case $0\le a<1$, all we need to do is to use the same  arguments that we have used to prove \eqref{W'Ba}.

\end{proof}

\subsection{Global solutions of polynomial growth}\label{SS:pol}

It is interesting here to analyze, in the range $0\le \alpha <\infty$, global solutions of \eqref{sb2} having polynomial growth.

\begin{dfn}\label{D:shBa}
Let $\kappa\ge 0$. We denote by $\mathfrak P_{\alpha,\kappa}(\R^{n+1})$ the space of all functions $P_\kappa\in \Gamma^2_\alpha(\R^{n+1})$ such that $\Ba P_\kappa = 0$ and $\Za P_\kappa = \kappa P_\kappa$. The elements of such space will be called $\Ba$-\emph{solid harmonics} of degree $\kappa$. We indicate by $\mathfrak P^+_{\alpha,\kappa}(\R^{n+1})$ the space of those functions $P_\kappa\in \mathfrak P_{\alpha,\kappa}(\R^{n+1})$ such that $P_\kappa(x,0) \ge 0$ for which $P_\kappa(x,-z) = P_\kappa(x,z)$.
\end{dfn}

Let us note explicitly that $D_z P_\kappa(x,0) = 0$. We emphasize that, for $P_\kappa\in \mathfrak P_{\alpha,\kappa}(\R^{n+1})$, the number $\kappa$ needs not be an integer. For instance, if $A = \frac{(\alpha+1)(2\alpha+1)}{n}$, then the function $P_\kappa(x,z) =  A |x|^2 - z^{2(\alpha + 1)}$ is a solution of $\Ba f = 0$, homogeneous of degree $\kappa = 2(\alpha+1)$. Thus, $P_\kappa \in \mathfrak P_{\alpha,\kappa}(\R^{n+1})$. We also have $P_\kappa \in \mathfrak P^+_{\alpha,\kappa}(\R^{n+1})$.

\begin{prop}\label{P:cm}
For every $\kappa\ge 0$ the space $\mathfrak P_{\alpha,\kappa}(\R^{n+1})$ is finite-dimensional.
\end{prop}

\begin{proof}
Let $u\in  \mathfrak P_{\alpha,\kappa}(\R^{n+1})$. Since $\Za u = \kappa u$ in $\R^{n+1}$,
for any $(x,z)\in \R^{n+1}$ such that $\ra(x,z)\ge 1$ we must have
\[
u(x,z) = \ra(x,z)^\kappa u(\delta_{\ra(x,z)^{-1}}(x,z)),
\]
and therefore
\[
|u(x,z)| \le \left(\underset{S_{\rho_\alpha(1)}}{\max}\ |u|\right)  \ra(x,z)^\kappa.
\]
This estimate gives
\begin{equation}\label{growth}
\underset{r\ge 1}{\sup}\ \left(\frac{1}{r^\kappa} \underset{B_{\rho_\alpha}(r)}{\sup} |u|\right) <\infty.
\end{equation}
With \eqref{growth} in hands, we can now invoke the Colding-Minicozzi type theorem at the end of the paper \cite{KL} by Kogoj and Lanconelli to conclude that $\mathfrak P_{\alpha,\kappa}(\R^{n+1})$ is finite-dimensional.

\end{proof}

It is quite notable that solid harmonics of different degrees enjoy the following orthogonality property. A similar orthogonality property fails  for the solid harmonics in the Heisenberg group $\Hn$.

\begin{prop}\label{P:ortho}
For every $\kappa\not= \mu$, let $P_\kappa\in \mathfrak P_{\alpha,\kappa}(\R^{n+1})$ and $P_\mu \in \mathfrak P_{\alpha,\mu}(\R^{n+1})$. Then, for every $r>0$ one has
\[
\int_{{S_{\rho_\alpha}(r)}} P_\kappa P_\mu \frac{\psi_\alpha}{|\nabla \ra|}dH_{n} = 0.
\]
\end{prop}

\begin{proof}
By formula (2.30) in \cite{G} we have
\begin{align*}
0\ & = \int_{B_{\rho_\alpha}(r)} \big(P_\kappa\ \Ba P_\mu - P_\mu\ \Ba P_\kappa\big) dx dz
\\
& = \int_{{S_{\rho_\alpha}(r)}} \bigg[P_\kappa\langle\na P_\mu,\na \ra\rangle - P_\mu\langle\na P_\kappa,\na \ra\rangle \bigg]\frac{dH_{n}}{|\nabla \ra|}
\end{align*}
Using the equation \eqref{nara2} in the latter identity we find
\[
\langle\na P_\kappa,\na \ra\rangle = \frac{\Za P_\kappa}{\ra} \psi_\alpha =  \kappa P_\kappa \frac{\psi_\alpha}{\ra}
\]
and similarly,
\[
\langle\na P_\mu,\na \ra\rangle = \frac{\Za P_\mu}{\ra} \psi_\alpha = \mu P_\mu \frac{\psi_\alpha}{\ra}.
\]
Combining the last three equations we obtain
\[
\frac{\mu - \kappa}{r} \int_{{S_{\rho_\alpha}(r)}} P_\kappa P_\mu \frac{\psi_\alpha}{|\nabla \ra|}dH_{n} = 0.
\]
Since $\kappa\not= \mu$, the desired conclusion follows.

\end{proof}

\begin{thrm}\label{T:N=k-B}
 Let $u\not\equiv 0$ be a solution of \eqref{sb2} in $\R^{n+1}$ with $\vf \equiv 0$. If for some number $\kappa> 0$ we have $N(u,r)\equiv \kappa$, then $u\in \mathfrak P_{\alpha,\kappa}(\R^{n+1})$.
\end{thrm}

\begin{proof}
From \eqref{WkBa} in Remark \ref{R:sh} we conclude that $\Za u = \kappa u$ in $\R^{n+1}$, and thus $u\in \mathfrak P_{\alpha,\kappa}(\R^{n+1})$.

\end{proof}

\subsection{Monneau type monotonicity formulas}\label{SS:monneau}

Our next result represents a generalization to solutions of the problem \eqref{sb2} with zero obstacle $\vf$ of the basic monotonicity Theorem 1.4.3 established in \cite{GP} for solutions of the lower-dimensional obstacle problem for the standard Laplacian $\Delta$. We will use such result to obtain a new monotonicity theorem for the problem \eqref{sb} above. For the case of the single homogeneity $\kappa = 2$, and in connection with solutions of the classical obstacle problem for $\Delta$, this monotonicity theorem was first proved by Monneau in \cite{M}.

\begin{thrm}[One-parameter monotonicity formulas of Monneau type for $\Ba, 0\le \alpha <\infty$]\label{T:MBa}
Let $u$ be a solution to \eqref{sb2} in $B_{\rho_\alpha}(r_0)$ with $\vf \equiv 0$, and denote by $\kappa = N(u,0^+)$. Let $P_\kappa\in \mathfrak P^+_{\alpha,\kappa}(\R^{n+1})$,  and consider the functional defined for $0<r<r_0$
\begin{equation}\label{M}
\mathcal M_\kappa(u,P_\kappa,r) = \frac{1}{r^{Q-1+2\kappa}} \int_{{S_{\rho_\alpha}(r)}} (u - P_\kappa)^2 \frac{\psi_\alpha}{|\nabla \ra|}dH_{n}.
\end{equation}
Then,
\begin{equation}\label{M'}
\frac{d}{dr} \mathcal M_\kappa(u,P_\kappa,r) \ge \frac{2}{r} \mathcal W_\kappa(u,r),
\end{equation}
and therefore by \eqref{WkBa} and \eqref{Nk}, $r \to  \mathcal M_\kappa(u,P_\kappa,r)$ is non-decreasing in $(0,r_0)$.
\end{thrm}

\begin{proof}
We begin with some preliminary considerations. Suppose that $u$ be a solution to \eqref{sb2} in $B_{\rho_\alpha}(r_0)$  with $\vf \equiv 0$. According to Theorem \ref{T:almgrenbgo} the frequency $N(u,\cdot)$ is monotone non-decreasing, and therefore the limit
\[
N(u,0^+) = \underset{r\to 0^+}{\lim}\ N(u,r),
\]
exists. If $\kappa = N(u,0^+)$, then again by Theorem \ref{T:almgrenbgo} we know that
\begin{equation}\label{Nk}
N(u,r) \ge \kappa,\ \ \ \ \ 0<r<r_0.
\end{equation}

By formulas \eqref{Nk} and \eqref{WkBa} we thus find
\begin{equation}\label{M2}
\mathcal W_\kappa(u,r) \ge 0,\ \ \ \ \ r>0.
\end{equation}
Therefore, the nondecreasing character of $r\to \mathcal M_\kappa(u,P_\kappa,r)$ will follow once we establish formula \eqref{M'}. We turn to this objective now.

The last part of Theorem  \ref{T:weissB}
 guarantees that $N(P_\kappa,r) \equiv \kappa$, and so, again by \eqref{WkBa}, we conclude that
\[
\mathcal W_\kappa(P_\kappa,r) \equiv 0.
\]
This observation allows to write, with $w = u - P_\kappa$,
\begin{align}\label{M3}
\mathcal W_\kappa(u,r) & = \mathcal W_\kappa(u,r) - \mathcal W_\kappa(P_\kappa,r)
\\
& = \frac{1}{r^{Q-2+2\kappa}} \int_{{B_{\rho_\alpha}(r)}} (|\na w|^2 + 2 \langle\na w,\na P_\kappa\rangle)dx dz
\notag\\
& - \frac{\kappa}{r^{Q-1+2\kappa}} \int_{{S_{\rho_\alpha}(r)}} (w^2 + 2 w P_\kappa) \frac{\psi_\alpha}{|\nabla \ra|}dH_{n}
\notag\\
& = \mathcal W_\kappa(w,r)  + \frac{2}{r^{Q-2+2\kappa}} \int_{{B_{\rho_\alpha}(r)}} \langle\na w,\na P_\kappa\rangle dx dz
\notag
\\
& - \frac{2\kappa}{r^{Q-1+2\kappa}} \int_{{S_{\rho_\alpha}(r)}} w P_\kappa \frac{\psi_\alpha}{|\nabla \ra|}dH_{n}.
\notag
\end{align}
We now claim that for each $r\in (0,r_0)$
\begin{equation}\label{c}
\frac{2}{r^{Q-2+2\kappa}} \int_{{B_{\rho_\alpha}(r)}} \langle\na w,\na P_\kappa\rangle dx dz - \frac{2\kappa}{r^{Q-1+2\kappa}} \int_{{S_{\rho_\alpha}(r)}} w P_\kappa \frac{\psi_\alpha}{|\nabla \ra|}dH_{n} = 0.
\end{equation}
To prove \eqref{c} we need to compute $\int_{{B_{\rho_\alpha}(r)}} \langle\na w,\na P_\kappa\rangle$.
We consider the sets introduced in \eqref{es} above. Then,
\[
\int_{{B_{\rho_\alpha}(r)}} \langle\na w,\na P_\kappa\rangle   = \underset{\ve\to 0^+}{\lim} \left\{\int_{{B^+_{\rho_\alpha}(r,\ve)}} \langle\na w,\na P_\kappa\rangle  +  \int_{{B^-_{\rho_\alpha}(r,\ve)}} \langle\na w,\na P_\kappa\rangle \right\}.
\]
Integrating by parts, and using \eqref{div} and the fact that on the set $L^+(r,\ve)$ we have $\nu = - e_{n+1}$, we find
\begin{align*}
& \int_{{B^+_{\rho_\alpha}(r,\ve)}}  \langle\na w,\na P_\kappa\rangle   = \int_{{S^+_{\rho_\alpha}(r,\ve)}} w \sum_{j=1}^m \langle X_j,\nu\rangle X_j P_\kappa dH_{n}
\\
& + \int_{{L^+(r,\ve)}} w \sum_{j=1}^m \langle X_j,\nu\rangle X_j P_\kappa dH_{n}
 -  \int_{{B^+_{\rho_\alpha}(r,\ve)}} w \sum_{j=1}^m \text{div}(X_j P_\kappa\ X_j)
\\
& =  \int_{{S^+_{\rho_\alpha}(r,\ve)}} w \langle\na P_\kappa,\na \ra\rangle \frac{dH_{n}}{|\nabla \ra|}  - \int_{{L^+(r,\ve)}} w  D_z P_\kappa dx -  \int_{{B^+_{\rho_\alpha}(r,\ve)}} w \Ba P_\kappa
\\
& =  \frac{\kappa}{r}  \int_{{S^+_{\rho_\alpha}(r,\ve)}} w P_\kappa \frac{\psi_\alpha}{|\nabla \ra|} dH_{n} - \int_{{L^+(r,\ve)}} w  D_z P_\kappa dx,
\end{align*}
since $\Ba P_\kappa = 0$. In the last equality we have used the crucial identity \eqref{nara2} in Lemma \ref{l:gradalpharho} and the homogeneity of $P_\kappa$, which gives $\Za P_\kappa = \kappa P_\kappa$. If we now use the fact that $D_z P_\kappa(x,0) = 0$, by Lebesgue dominated convergence we see that $\int_{{L^+(r,\ve)}} w  D_z P_\kappa dx \to 0$ as $\ve \to 0^+$, and thus
\begin{equation}\label{keystep}
\underset{\ve\to 0^+}{\lim} \int_{{B^+_{\rho_\alpha}(r,\ve)}} \langle\na w,\na P_\kappa\rangle = \frac{\kappa}{r} \int_{{S^+_{\rho_\alpha}(r)}} w P_\kappa \frac{\psi_\alpha}{|\nabla \ra|}dH_{n}.
\end{equation}
Similarly, we obtain
\begin{equation}\label{keystep2}
\underset{\ve\to 0^+}{\lim} \int_{{B^-_{\rho_\alpha}(r,\ve)}} \langle\na w,\na P_\kappa\rangle = \frac{\kappa}{r} \int_{{S^-_{\rho_\alpha}(r)}} w P_\kappa \frac{\psi_\alpha}{|\nabla \ra|}dH_{n}.
\end{equation}
Combining \eqref{keystep} and \eqref{keystep2} we finally conclude that
\[
\int_{{B_{\rho_\alpha}(r)}} \langle\na w,\na P_\kappa\rangle   = \frac{\kappa}{r} \int_{{S_{\rho_\alpha}(r)}} w P_\kappa \frac{\psi_\alpha}{|\nabla \ra|}dH_{n}.
\]
This proves the claim \eqref{c}. If we now use \eqref{c} in \eqref{M3}, we obtain the following crucial identity
\begin{equation}\label{M5}
\mathcal W_\kappa(u,r)  = \mathcal W_\kappa(w,r)  = \frac{1}{r^{Q-1+2\kappa}} H(w,r),\ \ \ \ \ 0<r<r_0.
\end{equation}
From the defining equation  \eqref{M} of the functional $\mathcal M_\kappa(u,P_\kappa,r)$ we have
\begin{equation}\label{Hw}
\mathcal M_\kappa(u,P_\kappa,r) = \frac{1}{r^{Q-1+2\kappa}} \int_{{S_{\rho_\alpha}(r)}}  w^2 \frac{\psi_\alpha}{|\nabla \ra|}dH_{N-1} = \frac{1}{r^{Q-1+2\kappa}} H(w,r).
\end{equation}
Differentiating \eqref{Hw} we find
\begin{align*}
\frac{d}{dr} \mathcal M_\kappa(u,P_\kappa,r) & = - \frac{Q-1+2\kappa}{r^{Q+2\kappa}} H(w,r)
 +  \frac{1}{r^{Q-1+2\kappa}} H'(w,r).
\end{align*}
Now, similarly to the expression in the right-hand side of \eqref{Ds} we see that
\[
H'(w,r)  = \frac{Q-1}{r} H(w,r)  + 2 \int_{S_{\rho_\alpha}(r)} w\left(\frac{Z_\alpha w}{r}\right)\frac{\psi_\alpha}{|\nabla \ra|}dH_{n}.
\]
Combining the latter two equations we find
\begin{align*}
\frac{d}{dr} \mathcal M_\kappa(u,P_\kappa,r) & = - \frac{2\kappa}{r^{Q+2\kappa}} H(w,r) + \frac{2}{r^{Q-1+2\kappa}} \int_{S_{\rho_\alpha}(r)} w\left(\frac{Z_\alpha w}{r}\right)\frac{\psi_\alpha}{|\nabla \ra|}dH_{n}.
\end{align*}
We now claim that
\begin{equation}\label{c2}
- \frac{2\kappa}{r^{Q+2\kappa}} H(w,r) + \frac{2}{r^{Q-1+2\kappa}} \int_{S_{\rho_\alpha}(r)} w\left(\frac{Z_\alpha w}{r}\right)\frac{\psi_\alpha}{|\nabla \ra|}dH_{n} \ge \frac 2r \mathcal W_\kappa(w,r).
\end{equation}
From \eqref{c2} and \eqref{M5} we would conclude
\[
\frac{d}{dr} \mathcal M_\kappa(u,P_\kappa,r)  \ge \frac 2r \mathcal W_\kappa(w,r) = \frac 2r \mathcal W_\kappa(u,r).
\]
In view of \eqref{M2} this would complete the proof of the theorem.
We are thus left with proving \eqref{c2}.

With this objective in mind we consider
\begin{align*}
\int_{{B_{\rho_\alpha}(r)}} |\na w|^2   = \underset{\ve\to 0^+}{\lim} \left\{\int_{{B^+_{\rho_\alpha}(r,\ve)}} |\na w|^2 +  \int_{{B^-_{\rho_\alpha}(r,\ve)}} |\na w|^2 \right\}.
\end{align*}
In $B^\pm_{\rho_\alpha}(r,\ve)$ we have from \eqref{sb2} and Definition \ref{D:shBa}
\[
\Ba(w^2/2) = w \Ba w + |\na w|^2 = w(\Ba u - \Ba P_\kappa) + |\na w|^2 =  |\na w|^2.
\]
We thus find
\begin{align*}
& \int_{{B_{\rho_\alpha}(r)}} |\na w|^2  = \underset{\ve\to 0^+}{\lim} \left\{\sum_{j=1}^{n+1} \int_{{B^+_{\rho_\alpha}(r,\ve)}} X_j X_j (w^2/2) + \sum_{j=1}^{n+1}  \int_{{B^-_{\rho_\alpha}(r,\ve)}} X_j X_j (w^2/2) \right\}
\\
& = \underset{\ve\to 0^+}{\lim} \left\{\sum_{j=1}^{n+1} \int_{{\p B^+_{\rho_\alpha}(r,\ve)}} w \langle X_j,\nu\rangle  X_j w + \sum_{j=1}^{n+1}  \int_{{\p B^-_{\rho_\alpha}(r,\ve)}} w \langle X_j,\nu\rangle  X_j w  \right\}
\\
& = \underset{\ve\to 0^+}{\lim} \left\{\int_{{S^+_{\rho_\alpha}(r,\ve)}} w \langle\na w,\na \rho_\alpha\rangle \frac{dH_{n}}{|\nabla \ra|}+ \int_{{S^-_{\rho_\alpha}(r,\ve)}} w \langle\na w,\na \rho_\alpha\rangle \frac{dH_{n}}{|\nabla \ra|}\right\}
\\
& + \underset{\ve\to 0^+}{\lim} \left\{- \int_{{L^+(r,\ve)}} w D_z w dx + \int_{{L^-(r,\ve)}} w D_z w dx\right\}
\\
& = \int_{{S_{\rho_\alpha}(r)}} w \langle\na w,\na \rho_\alpha\rangle \frac{dH_{n}}{|\nabla \ra|} + \int_{{B'_{\rho_\alpha}(r)}} w(-D^+_z w + D^-_z w) dx
\\
& = \int_{S_{\rho_\alpha}(r)} w\left(\frac{Z_\alpha w}{r}\right)\frac{\psi_\alpha}{|\nabla \ra|}dH_{n} + \int_{{B'_{\rho_\alpha}(r)}} w(-D^+_z u + D^-_z u) dx,
\end{align*}
where in the last equality we have used \eqref{nara2} and the fact that $D_z^+ P_\kappa = D_z^-P_\kappa = 0$ on the thin ball $B'_{\rho_\alpha}(r)$. By the last two equations in \eqref{sb2} we now have on $B'_{\rho_\alpha}(r)$
\[
w(-D^+_z u + D^-_z u) = u(-D^+_z u + D^-_z u) - P_\kappa(-D^+_z u + D^-_z u)  = - P_\kappa(-D^+_z u + D^-_z u)  \le 0,
\]
since $P_\kappa(x,0) \ge 0$. In conclusion, we have proved that
\[
\int_{S_{\rho_\alpha}(r)} w\left(\frac{Z_\alpha w}{r}\right)\frac{\psi_\alpha}{|\nabla \ra|}dH_{n} \ge \int_{{B_{\rho_\alpha}(r)}} |\na w|^2 dx dz = D(w,r).
\]
Keeping in mind the definition of $\mathcal W_\kappa(w,r)$, see \eqref{WBa} above, we conclude that the latter inequality proves the claim \eqref{c2}, thus completing the proof of the theorem.

\end{proof}

With Theorem \ref{T:MBa} in hands we will now establish a corresponding result for solutions to the problem \eqref{sb} above. We begin with introducing the space of the relevant ``harmonic" polynomials. In what follows we have indicated with $\tilde Z$ the generator of the standard Euclidean dilations with respect to the variables $(x,y)$.

\begin{dfn}\label{D:shBa}
Let $\tilde \kappa\ge 0$. We denote by $\tilde{\mathfrak P}_{a,\tilde \kappa}(\R^{n+1})$ the space of all polynomials $\tilde P_{\tilde \kappa}$ in $\R^{n+1}$ such that $L_a \tilde P_{\tilde \kappa} = 0$ and $\tilde Z \tilde P_{\tilde \kappa} = \tilde \kappa \tilde P_{\tilde \kappa}$. The elements of such space will be called $L_a$-\emph{solid harmonics} of degree $\tilde \kappa$. We indicate by $\tilde{\mathfrak P}^+_{a,\tilde \kappa}(\R^{n+1})$  the space of those elements $\tilde P_{\tilde \kappa}\in \tilde{\mathfrak P}_{a,\tilde \kappa}(\R^{n+1})$  such that $\tilde P_{\tilde \kappa}(x,0) \ge 0$ and $\tilde P_{\tilde \kappa}(x,-y) = \tilde P_{\tilde \kappa}(x,y)$. In particular, if $\tilde P_{\tilde \kappa}\in \tilde{\mathfrak P}^+_{a,\tilde \kappa}(\R^{n+1})$, then $\lim_{y\to 0^+} y^aD_y \tilde P_{\tilde \kappa}(x,y) = 0$.
\end{dfn}

\begin{lemma}\label{L:iso}
Let $0 \le a<1$, $\tilde \kappa > 0$, and consider the numbers $0 \le \alpha <\infty$ and $\kappa >0$ defined by $\alpha = \frac{a}{1-a}$, and $\kappa = \frac{\tilde \kappa}{1-a}$. Then, the equation \eqref{tu} above, with $h$ given by \eqref{h}, establishes a one-to-one onto correspondence between the spaces $\tilde{\mathfrak P}^+_{a,\tilde \kappa}(\R^{n+1})$ and $\mathfrak P^+_{\alpha,\kappa}(\R^{n+1})$. By slightly abusing the notation, such correspondence can be expressed by the following equation
\begin{equation}\label{Ps}
\widetilde{(P_\kappa)} = \tilde P_{\tilde \kappa},
\end{equation}
where we emphasize that in the right-hand side of \eqref{Ps} we have indicated with $\tilde P_{\tilde \kappa}$ an element of the space $\tilde{\mathfrak P}^+_{a,\tilde \kappa}(\R^{n+1})$, and not the operation that takes $f\to \tilde f$.
\end{lemma}

\begin{proof}
It follows in a straightforward fashion from Proposition \ref{P:hom}.

\end{proof}

We are now ready to establish the following basic result for the extension problem \eqref{sb}.

\begin{thrm}[One-parameter monotonicity formulas of Monneau type for $L_a, -1<a<1$]\label{T:monneauext}
Let $\tilde u$ be a solution of \eqref{sb} in $B_e(R_0)$ with $\vf\equiv 0$. We denote by $\tilde \kappa = \tilde N(\tilde u,0^+)$. Let $\tilde P_{\tilde \kappa}\in \tilde{\mathfrak P}^+_{a,\tilde \kappa}(\R^{n+1})$,  and consider the functional defined for $0<r<R_0$
\begin{equation}\label{Me}
\tilde{\mathcal M}_{\tilde \kappa}(\tilde u,\tilde P_{\tilde \kappa},r) = \frac{1}{r^{\tilde Q-1+2\tilde \kappa}} \int_{{S_{e}(r)}} (\tilde u - \tilde P_{\tilde \kappa})^2 |y|^a dH_{n},
\end{equation}
where $\tilde Q$ is given by  \eqref{tQ} above. Then,
\begin{equation}\label{M'e}
\frac{d}{dr} \tilde{\mathcal M}_{\tilde \kappa}(\tilde u,\tilde P_{\tilde \kappa},r)  \ge \frac{2}{r} \tilde{\mathcal W}_{\tilde \kappa}(\tilde u,r),
\end{equation}
and therefore by \eqref{WkBa} and \eqref{Nk}, $r \to  \tilde{\mathcal M}_{\tilde \kappa}(\tilde u,\tilde P_{\tilde \kappa},r)$ is non-decreasing in $(0,R_0)$.
\end{thrm}

\begin{proof}
We assume first that $0\le a<1$, and thus $0\le \alpha <\infty$. Let $\tilde u$ and $\tilde P_{\tilde \kappa}\in \tilde{\mathfrak P}^+_{a,\tilde \kappa}(\R^{n+1})$ be as in the statement of the theorem, and denote by $u$ and $P_\kappa$ the functions defined by \eqref{tu}. Then, $u$ satisfies \eqref{sb2} with $\vf \equiv 0$ and by Lemma \ref{L:iso} we have $P_{\kappa}\in \mathfrak P^+_{\alpha,\kappa}(\R^{n+1})$. By Theorem \ref{T:MBa} we know that $r\to \mathcal M_\kappa(u,P_\kappa,r)$ is non-decreasing in $(0,r_0)$, and \eqref{M'} holds. We now claim that for every $r\in (0,r_0)$ one has
\begin{equation}\label{tMM}
\mathcal M_{\kappa}(u,P_{\kappa},h(r)) = (1-a)^{\tilde Q - 1 - 2 \tilde \kappa - 2a} \tilde{\mathcal M}_{\tilde \kappa}(\tilde u,\tilde P_{\tilde \kappa},r).
\end{equation}
Notice that, if true, the claim implies that $r\to \tilde{\mathcal M}_{\tilde \kappa}(\tilde u,\tilde P_{\tilde \kappa},r)$ is
also non-decreasing on $(0,R_0)$ and that moreover
\begin{equation}\label{MM'}
\frac{d}{dr} \tilde{\mathcal M}_{\tilde \kappa}(\tilde u,\tilde P_{\tilde \kappa},r) =  (1-a)^{-\tilde Q + 1 + 2 \tilde \kappa + 2a} h'(r) \left(\frac{d}{d\tau} \mathcal M_{\kappa}(u,P_{\kappa},\cdot)\right)(h(r)).
\end{equation}
Using \eqref{M'} we find
\[
\left(\frac{d}{d\tau}\mathcal M_\kappa(u,P_\kappa,\cdot)\right)(h(r)) \ge \frac{2}{h(r)} \mathcal W_\kappa(u,h(r)).
\]
Substituting this inequality in \eqref{MM'}, and using \eqref{h} and the basic identity \eqref{Ws}, we obtain
\begin{align}\label{Ms}
& \frac{d}{dr} \tilde{\mathcal M}_{\tilde \kappa}(\tilde u,\tilde P_{\tilde \kappa},r) \ge (1-a)^{-\tilde Q + 1 + 2 \tilde \kappa + 2a}  \frac{2h'(r)}{h(r)} \mathcal W_\kappa(u,h(r))
\\
& =  (1-a)^{-\tilde Q + 2 + 2 \tilde \kappa + 2a}  (1-a)^{\tilde Q-2a -2+2\tilde \kappa} \frac 2r  \tilde{\mathcal W}_{\tilde \kappa}(\tilde u,r) = \frac 2r  \tilde{\mathcal W}_{\tilde \kappa}(\tilde u,r),
\notag
\end{align}
which proves \eqref{M'e}. In order to complete the proof of the theorem we are thus left with proving the claim \eqref{tMM}. However, this follows easily by observing that, if we set $w = u - P_\kappa$, then \eqref{M} gives
\begin{align*}
& \mathcal M_{\kappa}(u,P_{\kappa},h(r)) = \frac{1}{h(r)^{Q-1+2\kappa}} H(w,h(r))
\\
& = \frac{1}{h(r)^{Q-1+2\kappa}} (1-a)^{-a} r^{-a} \tilde H(\tilde w,r) = \frac{(1-a)^{-a+(1-a)(Q-1+2\kappa)}}{r^{(1-a)(Q-1+2\kappa)}} \tilde H(\tilde w,r),
\end{align*}
where in the second to the last equality we have used Lemma \ref{L:hth}. To see that, if we let $\tilde Q$ be given by \eqref{tQ}  (so that \eqref{QtQ} hold), and $\tilde \kappa = (1-a)\kappa$, then the right-hand side of the latter equation equals $(1-a)^{\tilde Q - 1 - 2 \tilde \kappa - 2a} \tilde{\mathcal M}_{\tilde \kappa}(\tilde u,\tilde P_{\tilde \kappa},r)$, all we need at this point is to observe that $\tilde w = \widetilde{u-P_\kappa} = \tilde u - \widetilde{(P_\kappa)}$, and use \eqref{Ps} in Lemma \ref{L:iso}.

Once we have the correct guess of the Monneau functional \eqref{Me} for $L_a$, and of the remarkable monotonicity result \eqref{M'e},   it is easy to complete the  proof in the remaining case $-1<a<0$. We leave the relevant details to the reader.

\end{proof}

\section{Structure and regularity of the singular set: the case of zero obstacle}
\label{sec4}

We study in this section the structure and the regularity of the singular part of the free boundary in the obstacle problem \eqref{pb} above with \emph{zero obstacle} for the fractional  Laplacian $(-\Delta)^s$. This means that, with $a = 1-2s$, in the equivalent formulation as a thin obstacle problem for the extension operator $L_a$ in \eqref{sb} above, we take $\vf = 0$. We also localize the problem to a ball centered at the origin of fixed radius, say $R_0 = 1$. For the sake of simplicity, henceforth we write $B(R)$, instead of $B_e(R)$, for the Euclidean ball centered at the origin with radius $R$ in $\R^{n+1}$ (with respect to the variables $(x,y)$), and we denote by $B'(R) = B(R)\cap \{y=0\}$ the corresponding thin ball. Thus, we consider a function $\tilde u$ that solves
\begin{equation}\label{obst-0}
\begin{cases}
L_a \tilde u = \operatorname{div}(|y|^a \nabla \tilde u) = 0\ \ \ \ \ \ \ \ \text{in}\ B(1)_+\cup B(1)_-,
\\
\tilde u(x,0) \ge 0,\ \ \ \ \ \ \ \ \ \ \ \ \ \ \ \ \ \ \ \ \ \ \  \text{for}\ (x,0)\in B'(1),
\\
\tilde u(x,-y) = \tilde u(x,y),\ \ \ \ \ \ \ \ \ \  \ \ \ \  \text{for}\ (x,y)\in B(1),
\\
- \underset{y\to 0^+}{\lim} y^a D_y\tilde u(x,y) \ge 0, \ \ \ \ \  \ \ \  \text{in}\ B(1)^+,
\\
\underset{y\to 0^+}{\lim} y^a D_y \tilde u(x,y) = 0, \ \ \ \ \ \ \ \ \ \ \text{on the subset of $B'(1)$ where}\ \tilde u(x,0) > 0.
\end{cases}
\end{equation}

Recall that, by Theorem \ref{T:almgrenext}, the Almgren-type frequency function
\[\tilde N(\tilde u,r)=\frac{r\int_{B(r)} |\nabla \tilde u|^2 |y|^a}{\int_{\partial B(r)} \tilde u^2 |y|^a}\]
is nondecreasing in $r\in(0,1)$.
Moreover, $\tilde N(\tilde u,r)\equiv \tilde \kappa$ if and only if $\tilde u$ is homogeneous of degree~$\tilde \kappa$.
Furthermore, by Theorem \ref{T:monneauext}, if  for $\tilde \kappa>0$ we have $\tilde N(\tilde u,0^+)=\tilde \kappa$, then for any polynomial $P_{\tilde \kappa}\in \tilde{\mathfrak P}_{a,\tilde \kappa}^+$ the quantity
\[
\tilde{\mathcal M}_{\kappa}(\tilde u, \tilde P_{\tilde \kappa},r)
=\frac{1}{r^{n+a+2\tilde \kappa}}\int_{\partial B(r)} \bigl(\tilde u-\tilde P_{\tilde \kappa}\bigr)^2 |y|^a
\]
is nondecreasing in $r\in(0,1)$. Here, we notice that, according to \eqref{tQ}, the exponent in \eqref{Me} is given by $\tilde Q - 1 + 2\tilde \kappa = n+a+2\tilde \kappa$.
Using those two results, we will next study the structure and regularity of the singular set.

\subsection{Notation}

Throughout this section, we let $a=1-2s$, and for $0<r<1$ continue to use the notation \eqref{tD}, \eqref{tH} and \eqref{tN} above for $\tilde D(\tilde u,r)$, $\tilde H(\tilde u,r)$ and $\tilde N(\tilde u,r)$ respectively, except that in all integrals we will write for simplicity $B(r)$, instead of $B_e(r)$, and we will routinely omit the indication of the differential of volume $dxdy$, or surface measure $dH_{n}$.

\subsection{Blow-ups}

We establish first some results that will be needed later.

\begin{lemma}\label{bound-H-0}
Let $\tilde u$ be a solution of \eqref{obst-0}.
Assume that $0$ is a free boundary point, i.e., $0\in \Gamma(\tilde u)$, and suppose that $\tilde N(\tilde u,0^+)=\kappa$.
Then, for $0<r<1$ one has
\begin{equation}\label{growth-H-kappa-0}
\tilde H(\tilde u,r) \leq \tilde H(\tilde u,1) r^{n+a+2\kappa}.
\end{equation}
Moreover, for any $\varepsilon>0$ there exists $r_{\varepsilon}>0$ such that for $0<r<r_{\varepsilon}$ one has
\begin{equation}\label{growth-H-kappa-below-0}
\tilde H(\tilde u,r) \ge \frac{\tilde H(\tilde u,r_\ve)}{r_\ve^{n+a+2\tilde \kappa + \ve}} r^{n+a+2\kappa+\varepsilon}.
\end{equation}
\end{lemma}

\begin{proof}
First, from Lemma \ref{L:hth} we obtain
\[
\tilde H'(\tilde u,r) = \frac ar \tilde H(\tilde u,r) + (1-a)^a r^a h'(r) H'(u,h(r)).
\]
Using now Lemma \ref{l:hprime} and then Lemma \ref{L:dtd} we obtain
\[
\tilde H'(\tilde u,r)=\frac{\tilde Q -1}{r} \tilde H(\tilde u,r)+2\tilde D(\tilde u,r) = \frac{n+a}{r} \tilde H(\tilde u,r)+2\tilde D(\tilde u,r),
\]
where in the second equality we have used \eqref{tQ}.
Thus, we have
\begin{equation}\label{freq-H-0}
\frac{d}{dr}\log \tilde H(\tilde u,r)=2\frac{\tilde N(\tilde u,r)}{r}+\frac{n+a}{r}.
\end{equation}
Since by Theorem \ref{T:almgrenext} the function $r\to \tilde N(\tilde u,r)$ is non-decreasing in $(0,1)$, it follows that $\tilde N(\tilde u,r) \ge \tilde \kappa$. Therefore, for $r\in(0,1)$ we have
\[
\frac{n+a+2\tilde \kappa}{r}\leq \frac{d}{dr}\log \tilde H(\tilde u,r).
\]
Integrating from $r$ to $1$, we obtain \eqref{growth-H-kappa-0}.

To prove \eqref{growth-H-kappa-below-0} it suffices to observe that for every $\ve>0$ there exists $0<r_\ve<1$ such that $0\le \tilde N(\tilde u,r) - \tilde \kappa < \frac{\ve}{2}$ for $0<r<r_\ve$. We thus find from \eqref{freq-H-0}
\[
\frac{d}{dr}\log \tilde H(\tilde u,r) \le \frac{n+a+2\tilde \kappa + \ve}{r}.
\]
Integrating from $r$ to $r_\varepsilon$ we obtain the desired conclusion \eqref{growth-H-kappa-below-0}.

\end{proof}

\begin{dfn}\label{D:almgrenres}
Given a solution $\tilde u$ of \eqref{obst-0}, we define the \emph{Almgren rescalings} of $\tilde u$ as
\[
\tilde u_r(x,y):=\frac{\tilde u(rx,ry)}{d_r},
\]
where
\[
d_r:=\left(\frac{\tilde H(\tilde u,r)}{r^{n+a}}\right)^{1/2}.
\]
We note that for every $0<r<1$ one has
\begin{equation}\label{tilde H_1}
\int_{\p B(1)} \tilde u_r^2\ |y|^a = 1,
\end{equation}
and moreover we have the following scale invariant property
\begin{equation}\label{tilde N_scale}
\tilde N(\tilde u_r,\rho) = \tilde N(\tilde u,r\rho).
\end{equation}
\end{dfn}

We next show the following.

\begin{prop}\label{blow-ups-0}
Let $\tilde u$ be a solution of \eqref{obst-0}.
Assume that $0\in \Gamma(\tilde u)$ and that $\tilde N(\tilde u,0^+)=\tilde \kappa$.
Then, up to a subsequence, $\tilde u_r$ converge as $r\to 0^+$ to a homogeneous function $\tilde u_0$ which is a global solution to the zero obstacle problem and it is homogeneous of degree $\tilde \kappa$.
\end{prop}

\begin{proof}
The proof is standard and follows from the monotonicity of $\tilde N(\tilde u,r)$.
For completeness, we sketch the details here.
First, we have
\[
\tilde D(\tilde u_r,1) = \tilde N(\tilde u_r,1) = \tilde N(\tilde u,r) \le \tilde N(\tilde u,1).
\]
Notice that in the first equality we have used \eqref{tilde H_1}, in the second we have used \eqref{tilde N_scale}, whereas in the last inequality we have used the monotonicity of $r\to \tilde N(\tilde u,r)$ in Theorem \ref{T:almgrenext}. Since again by \eqref{tilde H_1} we know that $\|\tilde u_r\|_{L^2(\partial B(1))}=1$, we infer that the sequence $\{\tilde u_r\}$ is uniformly bounded in $H^1(B(1),|y|^a)$, so that (up to a subsequence) the functions $\tilde u_r$ converge to $\tilde u_0$ weakly in $H^1(B(1),|y|^a)$, strongly in $L^2(S(1))$, and a.e. in $B(1)$.

Moreover, the rescaled functions $\tilde u_r$ solve the zero obstacle problem \eqref{obst-0} in the ball $B(1/r)$.
Thus, by the $C^{1,\alpha}$ estimates in \cite{CSS} we conclude that the sequence $\tilde u_r$ is uniformly bounded in $C^{1,\alpha}_{\rm loc}(B(1))$, and hence $\tilde u_r\to \tilde u_0$ in $C^1_{\rm loc}(B(1))$.
Letting $r\to0$, we find that $\tilde u_0$ is a global solution to the zero obstacle problem, with $\|\tilde u_0\|_{L^2(\partial B(1))}=1$, thus in particular $\tilde u_0\not\equiv 0$.

Finally, the global solution $\tilde u_0$ is homogeneous of degree $\tilde \kappa$.
For this, we notice that from the above convergence properties and from \eqref{tilde N_scale}
\[\tilde N(\tilde u_0,r)=\lim_{\rho\to0^+}\tilde N(\tilde u_\rho,r)=\lim_{\rho\to0^+}\tilde N(\tilde u,r\rho)=\tilde N(\tilde u,0^+) = \tilde \kappa.\]
This means that $\tilde N(\tilde u_0,r)\equiv \tilde \kappa$. By the second part of Theorem \ref{T:almgrenext} we conclude that $\tilde u_0$ is homogeneous of degree $\kappa$.

\end{proof}

Thanks to the previous result, we may define $\Gamma_{\tilde \kappa}(\tilde u)$ as the set of free boundary points at which $\tilde N(\tilde u,0^+)=\tilde \kappa$.

\subsection{Characterization of singular points}

We next prove a characterization of singular points similar to \cite[Theorem 1.3.2]{GP}.
Recall that we denote $\Sigma(\tilde u)$ the set of \emph{singular} free boundary points, and that
\[
\Sigma_{\tilde \kappa}(\tilde u)=\Sigma(\tilde u)\cap\Gamma_{\tilde \kappa}(\tilde u).
\]

\begin{prop}\label{characterization}
Let $\tilde u$ be a solution of \eqref{obst-0}.
Assume that $0$ is a free boundary point, and that $\tilde N(\tilde u,0^+)=\tilde \kappa$.
Then, the following statements are equivalent:
\begin{itemize}
\item[(i)] $0\in \Sigma_{\tilde \kappa}(\tilde u)$
\item[(ii)] any blow-up of $\tilde u$ at the origin is a nonzero homogeneous polynomial $p_\kappa(x,y)$ of degree $\tilde \kappa$ satisfying
    \[L_ap_{\tilde \kappa}=0,\qquad p_{\tilde \kappa}(x,0)\geq0,\qquad p_\kappa(x,-y)=p_{\tilde \kappa}(x,y)\]
\item[(iii)] $\tilde \kappa=2m$ for some $m\in \{1,2,3,...\}$
\end{itemize}
\end{prop}

\begin{proof}
(i) $\Longrightarrow$ (ii)  By assumption, the rescalings $\tilde u_r$ satisfy $|\{\tilde u_r(x,0)=0\}\cap B_1\}|\to0$ as $r\to0$.
(Recall that $B_1$ denotes the Euclidean unit ball in $\R^n=\R^{n+1}\cap\{y=0\}$.)
Let us denote $\Gamma(\tilde u_r)=\{\tilde u_r(x,0)=0\}\cap \{y=0\}$.
Then, it is not difficult to see that $L_a \tilde u_r=c\bigl(\lim_{y\downarrow0}|y|^a\partial_y \tilde u_r\bigr)\left.\mathcal H^n\right|_{\Gamma(\tilde u_r)}$.
Moreover, by \cite[Proposition 4.3]{CSS} we have that $|y|^a\partial_y \tilde u_r$ is uniformly bounded in $B(1)$.
Since $|\Gamma(\tilde u_r)\cap B_1|\to0$, then we find that $L_a \tilde u_r$ converges weakly to 0, and therefore any blow-up $\tilde u_0$ satisfies $L_a\tilde u_0=0$ in $B(1)$.
Since $\tilde u_0$ is homogeneous, then $L_a\tilde u_0=0$ in $\R^{n+1}$.
By \cite[Lemma 5.3]{CSS} this means that $\tilde u_0$ is a polynomial of degree $\kappa$ satisfying $L_a\tilde u_0=0$ in $\R^{n+1}$.
The properties of $\tilde u_r$ also imply that $\tilde u_0$ is not identically zero, $\tilde u_0(x,0)\geq0$ in $\R^n$, and $\tilde u_0(x,-y)=\tilde u_0(x,y)$.

(ii) $\Longrightarrow$ (i) Suppose that 0 is not a singular point, that is, we have $|\{\tilde u_r(x,0)=0\}\cap B_1\}|\geq\delta>0$ over some sequence $r=r_j\to0^+$.
Taking a subsequence if necessary, we may assume that $\tilde u_{r_j}$ converges to a blow-up $\tilde u_0$.
Assume that $|\{\tilde u_0(x,0)=0\}\cap B_1\}|<\delta$, and let us see that this is not possible.
Indeed, if that happens then there exists an open set $A$ with $|A|<\delta$ and such that $\{\tilde u_0(x,0)=0\}\cap B_1\}\subset A$.
But then for large $j$ we have $\{\tilde u_{r_j}(x,0)=0\}\cap B_1\}\subset A$, a contradiction.

(ii) $\Longrightarrow$ (iii) If $\kappa$ is odd, then $p_\kappa(x,0)\geq0$ implies that $p_\kappa(x,0)\equiv0$ on $\R^n$.
Thus, by Lemma \ref{polynomials} this yields $p_\kappa\equiv0$ in $\R^{n+1}$.
Therefore, $\kappa$ is even.

(iii) $\Longrightarrow$ (ii) Let $\tilde u_0$ be any blow-up of $\tilde u$ at $0$, which by assumption is homogeneous of degree $\kappa=2m$.
Notice that, $\mu=L_a \tilde u_0$ is a nonpositive measure on $\{y=0\}$.
We want to show that $\mu=0$.

Let $P(x,y)$ be a $2m$-homogeneous polynomial satisfying $L_aP=0$ in $\R^{n+1}$, which is positive in $\{y=0\}\setminus\{0\}$ (take for example $q_{2m}(x)=|x|^{2m}$ in Lemma \ref{polynomials}).
Let $\Psi\in C^\infty_c(\R^{n+1}\setminus\{0\})$ be a nonnegative radial function.

Notice that
\[(x,y)\cdot \nabla P(x,y)=2mP(x,y),\qquad (x,y)\cdot \nabla \tilde u_0(x,y)=2m\tilde u_0(x,y),\]
and
\[\begin{split}
\tilde u_0(x,y)\nabla \Psi(x,y)\cdot \nabla P(x,y) &= 2m\tilde u_0(x,y)P(x,y)\nabla \Psi(x,y)\cdot \frac{(x,y)}{|x|^2+|y|^2}\\
&= P(x,y)\nabla \Psi(x,y)\cdot\nabla \tilde u_0(x,y).
\end{split}\]
Also, since $\mu$ is supported on $\{y=0\}$ and $P(x,0)\geq0$ we see that
\[\langle\mu,\Psi P\rangle \geq0.\]
Thus, integrating by parts and using that $L_aP=0$ we find
\[\begin{split}
0\leq \langle\mu,\Psi P\rangle &= \langle L_a\tilde u_0,\Psi P\rangle=\int_{\R^{n+1}} |y|^a \nabla \tilde u_0\cdot \nabla(\Psi P)\\
&=\int_{\R^{n+1}} |y|^a\bigl(\Psi \nabla P\cdot \nabla \tilde u_0+P\nabla \tilde u_0\cdot\nabla\Psi\bigr) \\
&=\int_{\R^{n+1}} |y|^a\bigl(-\Psi \tilde u_0\nabla P|y|^{-a}L_aP-\tilde u_0\nabla\Psi\cdot\nabla P+P\nabla \tilde u_0\cdot\nabla\Psi\bigr) =0,
\end{split}\]
that is,
\[\int_{\R^n} P(x,0)\Psi(x,0)d\mu(x)=0.\]
Since $\Psi\geq0$ is arbitrary, it follows that $\mu=c\delta_0$ for some $c\geq0$.
However, since $\mu$ is $(2m-2)$-homogeneous (it is a second derivative of a $2m$-homogeneous function), the only possibility is that $\mu\equiv0$, i.e., $L_a\tilde u_0=0$ in $\R^{n+1}$.
Finally, by \cite[Lemma 5.3]{CSS} $\tilde u_0$ is a polynomial.
\end{proof}

\subsection{Structure and regularity of the singular set in the zero obstacle case}

In this section we establish the structure and regularity of the singular set for the problem \eqref{obst-0}.
We start with the following nondegeneracy Lemma.
Recall that $\Sigma_\kappa(u)$ denotes the set of singular points with frequency $\kappa$.

\begin{lemma}\label{nondegeneracy-0}
Let $\tilde u$ be a solution of \eqref{obst-0}.
Assume that $0\in \Sigma_\kappa(u)$, $\kappa=2m$.
Then, for all $r\in(0,1)$ we have
\[C^{-1}r^\kappa\leq \sup_{\partial B(r)}|\tilde u|\leq Cr^\kappa\]
for some $C>0$.
\end{lemma}

\begin{proof}
The upper bound follows from \eqref{growth-H-kappa-0} and \cite[Lemma 3.4]{BFR}.

To prove the lower bound, we argue by contradiction and assume that for a sequence $r=r_j\to0$ we have $\sup_{\partial B(r)}|\tilde u|= o(r^\kappa)$.
Then,
\begin{equation}\label{contr-0}
d_r=\left(\frac{1}{r^{n+a}}\int_{\partial B(r)} \tilde u^2 |y|^a \right)^{1/2}=o(r^\kappa).
\end{equation}
Passing to a subsequence if necessary we may assume that
\[\tilde u_r(x,y)=\frac{\tilde u(rx,ry)}{d_r}\to p_\kappa(x,y)\qquad \textrm{uniformly on}\quad \partial B(1),\]
for some nonzero $p_\kappa\in\tilde{\mathfrak P}_{a,\kappa}^+$.
Now, for such $p_\kappa$ we use the Monneau formula in Theorem \ref{T:monneauext}.
If \eqref{contr-0} holds then we have
\[\tilde{\mathcal M_\kappa}(\tilde u,p_\kappa,0^+)=\int_{\partial B(1)} p_\kappa^2 |y|^a=\frac{1}{r^{n+a+\kappa}}\int_{\partial B(r)}p_\kappa^2|y|^a.
\]
Therefore, using the monotonicity of $\tilde{\mathcal M_\kappa}(\tilde u,p_\kappa,r)$, we will have that
\[\frac{1}{r^{n+a+\kappa}}\int_{\partial B(r)}|y|^a(\tilde u-p_\kappa)^2\geq \frac{1}{r^{n+a+\kappa}}\int_{\partial B(r)}|y|^ap_\kappa^2\]
or, equivalently,
\[\frac{1}{r^{n+a+\kappa}}\int_{\partial B(1)}|y|^a(\tilde u^2-2v p_\kappa)\geq0.\]
After rescaling we obtain
\[\frac{1}{r^{2\kappa}}\int_{\partial B(1)} |y|^a(d_r^2\tilde u_r^2-2d_rr^\kappa \tilde u_rp_\kappa)\geq 0,\]
and thus
\[\int_{\partial B(1)} |y|^a\left(\frac{d_r}{r^\kappa}\tilde u_r^2-2 \tilde u_rp_\kappa\right)\geq 0.\]
Since $\tilde u_r\to p_\kappa$ as $r\to0$, then letting $r\to0$ in the last inequality and using \eqref{contr-0} we find
\[-\int_{\partial B(1)}p_\kappa^2\geq0,\]
a contradiction.
\end{proof}

We next prove uniqueness and continuity of blow-ups.

\begin{prop}\label{uniqeness-0}
Let $\tilde u$ be a solution of \eqref{obst-0}.
Then, there exists a modulus of continuity $\omega:\R^+\to\R^+$ such that, for any $x_0\in \Sigma_\kappa(\tilde u)\cap B(1/2)'$, we have
\[\tilde u(x,0)=p_\kappa^{x_0}(x-x_0)+\omega(|x-x_0|)|x-x_0|^\kappa\]
for some polynomial $p_\kappa\in \tilde{\mathfrak P}_{a,\kappa}^+$.
In addition, the mapping $\Sigma_\kappa(\tilde u)\cap B(1/2)'\ni x_0\mapsto p_\kappa^{x_0}\in\tilde{\mathfrak P}_{a,\kappa}^+$ is continuous, with
\[\int_{\partial B(1)} |y|^a\bigl(p_\kappa^{x_0'}-p_\kappa^{x_0}\bigr)^2\leq \omega(|x_0'-x_0|)\]
for all $x_0,x_0'\in \Sigma_\kappa(\tilde u)\cap B(1/2)$.
\end{prop}

\begin{proof}
Let $v^{x_0}(x,y)=\tilde u(x_0+x,y)$, and let
\[v_r^{x_0}(x,y)=\frac{v^{x_0}(rx,ry)}{r^\kappa}.\]
By Lemma \ref{nondegeneracy-0} we have that
\[C^{-1}\rho^\kappa\leq \sup_{B(\rho)}|v_r^{x_0}|\leq C\rho^\kappa,\]
for all $\rho\in (0,1/r)$.
Thus, exactly as in Proposition \ref{blow-ups-0} we get
\begin{equation}\label{71-0}
v_{r_j}^{x_0}\rightarrow v_0^{x_0}\qquad\textrm{in}\quad C^1_{\rm loc}(\R^{n+1})\quad \textrm{along a subsequence}\ r_j\to0.
\end{equation}
Moreover, $v_0^{x_0}$ is not identically zero, and it is an homogeneous polynomial $p_\kappa\in\tilde{\mathfrak P}_{a,\kappa}^+$.

Hence, using \eqref{71-0} we get
\[\tilde{\mathcal M}_\kappa(v_{x_0},p_\kappa^{x_0},0^+) =\lim_{r_j\to0}\int_{\partial B(1)}|y|^a(v_{r_j}^{x_0}-p_\kappa^{x_0})^2=0.\]
Thus, the Monneau-type monotonicity formula in Theorem \ref{T:monneauext} implies
\begin{equation}\label{72-0}
\int_{\partial B(1)}|y|^a(v_r^{x_0}-p_\kappa^{x_0})^2=\tilde{\mathcal M}_\kappa(v_{x_0},p_\kappa^{x_0},r)\longrightarrow0\qquad \textrm{as}\quad r\downarrow0
\end{equation}
(not just along a subsequence).
This immediately implies that the blow-up is unique, and since $v^{x_0}(x,0)=\tilde u(x,0)$, we deduce that $\tilde u(x,0)=p_\kappa^{x_0}(x-x_0)+o(|x-x_0|^\kappa)$.
The fact that the rest $o(|x-x_0|^\kappa)$ is uniform with respect to $x_0$ follows from a simple compactness argument, as in \cite[Lemma 7.3 and Proposition 7.7]{PSU}.

We now prove continuous dependence of $p_\kappa^{x_0}$ with respect to $x_0$.
Given $\varepsilon>0$ it follows from \eqref{72-0} that there exists $r_\varepsilon=r_\varepsilon(x_0)>0$ such that
\[\tilde{\mathcal M}_\kappa(v^{x_0},p_\kappa^{x_0},r_\varepsilon)<\varepsilon.\]
Now, by continuous dependence of $v^{x_0}$ with respect to $x_0$, there exists $\delta_\varepsilon=\delta_\varepsilon(x_0)>0$ such that
\[\tilde{\mathcal M}_\kappa(v^{x_0'},p_\kappa^{x_0},r_\varepsilon)<2\varepsilon\]
for all $x_0'\in \Gamma_\kappa(\tilde u)$ satisfying $|x_0'-x_0|<\delta_\varepsilon$.
Then, it follows from Theorem \ref{T:monneauext} that
\[\tilde{\mathcal M}_\kappa(v^{x_0'},p_\kappa^{x_0},r)<2\varepsilon\]
for all $r\in(0,r_\varepsilon]$.
Letting $r\to0$ we obtain
\[\int_{\partial B(1)}|y|^a(p_\kappa^{x_0'}-p_\kappa^{x_0})^2=\tilde{\mathcal M}_\kappa(v_{x_0'},p_\kappa^{x_0},0^+)\leq 2\varepsilon.\]
Since $\varepsilon$ is arbitrary, we deduce that $p_\kappa^{x_0}$ is continuous with respect to $x_0$.
Finally, the uniform continuity follows exactly as in the proof of \cite[Theorem 1.5.4]{GP}.
\end{proof}

To end this section, we show the following.

\begin{thrm}\label{structure-0}
Let $\tilde u$ be a solution of \eqref{obst-0}.
Then,
\[\Sigma(\tilde u)\cap B(1/2)=\cup_{m=1}^\infty \Sigma_{2m}(\tilde u)\cap B(1/2).\]
Moreover, the blow-up of $\tilde u$ at any $x_0\in \Sigma_{2m}(\tilde u)\cap B(1/2)'$ is a {unique} homogeneous polynomial $p_{2m}^{x_0}$ of degree $2m$, and
\[\Sigma_{2m}(\tilde u)\cap B(1/2)=\cup_{d=1}^{n-1}\Sigma_{2m}^d(\tilde u)\cap B(1/2),\]
where
\begin{equation}\label{Sigma-2m-d-0}
\qquad\qquad\qquad\qquad\Sigma_{2m}^d(\tilde u):=\bigl\{x_0\in \Sigma_{2m}(\tilde u)\,:\, d(p_{2m}^{x_0})=d\bigr\},\qquad d=0,1,...,n-1.
\end{equation}
Furthermore, every set $\Sigma_{2m}^d(\tilde u)\cap B(1/2)$ is contained in a countable union of $d$-dimensional $C^1$ manifolds.
\end{thrm}

\begin{proof}
Let $x_0\in \Sigma(\tilde u)\cap B(1/2)$.
Then, by Proposition \ref{characterization} we have that $x_0\in \Sigma_{2m}(\tilde u)$ for some $m\in\{1,2,3,...\}$.
Moreover, the blow-up of $\tilde u$ at $x_0$ is an homogeneous polynomial $p_{2m}^{x_0}$ of degree $2m$, and it is unique by Proposition \ref{uniqeness-0}.

Finally, the fact that every set $\Sigma_{2m}^d(\tilde u)\cap B(1/2)$ is contained in a countable union of $d$-dimensional $C^1$ manifolds follows Proposition \ref{uniqeness-0} above combined with Whitney's extension theorem and the implicit function theorem, exactly as in \cite[Theorem 1.3.8]{GP}.

\end{proof}

\section{Analytic obstacles: reduction to the zero obstacle case}
\label{sec5}

We next study problem \eqref{pb} in case that $\varphi$ is \emph{analytic}.
First, we have the following.

\begin{lemma}\label{analytic}
Let $\varphi(x)$ be an analytic function in $B_1(x_0)\subset \R^n$.
Then, there exists an analytic function $\tilde \varphi(x,y)$, defined for $(x,y)$ in a neighborhood of $(x_0,0)$, satisfying:
\begin{itemize}
\item $\tilde \varphi(x,0)=\varphi(x)$ on $\{y=0\}$
\item $\tilde\varphi(x,y)$ is analytic in a neighborhood of $(x_0,0)$
\item $L_a\tilde\varphi=0$ in a neighborhood of $(x_0,0)$
\end{itemize}
\end{lemma}

Lemma \ref{analytic} will be proved below.
But first, using this result, we give the:

\begin{proof}[Proof of Theorem \ref{th-analytic}]
Let $\tilde u$ be the Caffarelli-Silvestre extension of $u$ in $\R^{n+1}$.

Let $x_0\in \Gamma(u)\subset \R^n$ be any free boundary point, and let $\tilde \varphi$ be the analytic extension of $\varphi$ near $(x_0,0)$ given by Lemma \ref{analytic}.
Let $r_0$ be such that $\tilde \varphi$ is analytic and satisfies $L_a\tilde\varphi=0$ in the ball (in $\R^{n+1}$) of radius $r_0$ centered at $(x_0,0)$.
Then, the function
\begin{equation}\label{v-analytic}
v^{x_0}(x,y):=\tilde{u}(x_0+r_0x,r_0y)-\tilde\varphi(x_0+r_0x,r_0y)
\end{equation}
solves the zero obstacle problem \eqref{obst-0} in the unit ball $B(1)$ of $\R^{n+1}$.

Since $v^{x_0}$ solves the zero obstacle problem \eqref{obst-0}, it follows from Theorem \ref{structure-0} that:
\[\Sigma(v^{x_0})\cap B(1/2)=\cup_{m=1}^\infty \Sigma_{2m}(v^{x_0})\cap B(1/2).\]
Moreover, the blow-up of $v^{x_0}$ at any $x_1\in \Sigma_{2m}(v^{x_0})\cap B(1/2)$ is a {unique} homogeneous polynomial $p_{2m}^{x_1}$ of degree $2m$, and
\[\Sigma_{2m}(v^{x_0})\cap B(1/2)=\cup_{d=1}^{n-1}\Sigma_{2m}^d(v^{x_0})\cap B(1/2),\]
where
\[\qquad\qquad\qquad\qquad\Sigma_{2m}^d(v^{x_0}):=\bigl\{x_1\in \Sigma_{2m}(v^{x_0})\,:\, d(p_{2m}^{x_1})=d\bigr\},\qquad d=0,1,...,n-1.\]
Furthermore, every set $\Sigma_{2m}^d(v^{x_0})\cap B(1/2)$ is contained in a countable union of $d$-dimensional $C^1$ manifolds.

Translating such result into $u$, we find that
\[\Sigma(u)\cap B_{r_0/2}(x_0)=\cup_{m=1}^\infty \Sigma_{2m}(u)\cap B_{r_0/2}(x_0).\]
Moreover, the blow-up of $u$ at any $x_1\in \Sigma_{2m}(u)\cap B_{r_0/2}(x_0)$ is a {unique} homogeneous polynomial $p_{2m}^{x_1}$ of degree $2m$, and
\[\Sigma_{2m}(u)\cap B_{r_0/2}(x_0)=\cup_{d=1}^{n-1}\Sigma_{2m}^d(u)\cap B_{r_0/2}(x_0),\]
where
\[\qquad\qquad\qquad\qquad\Sigma_{2m}^d(u):=\bigl\{x_1\in \Sigma_{2m}(u)\,:\, d(p_{2m}^{x_1})=d\bigr\},\qquad d=0,1,...,n-1.\]
Furthermore, every set $\Sigma_{2m}^d(u)\cap B_{r_0/2}(x_0)$ is contained in a countable union of $d$-dimensional $C^1$ manifolds.

Since this can be done for every $x_0\in \Sigma(u)$, the result follows.
\end{proof}

To prove Lemma \ref{analytic}, we need the following.

\begin{lemma}\label{polynomials}
Let $q_\kappa(x)$ be an homogeneous polynomial of degree $\kappa$ on $\R^n$.
Then, there exists a unique homogeneous polynomial $\tilde q_\kappa(x,y)$ of degree $\kappa$ on $\R^{n+1}$ such that
\[L_a \tilde q_\kappa=0\qquad \textrm{in}\quad \R^{n+1},\]
\[\tilde q_\kappa(x,0)=q_\kappa(x)\qquad \textrm{on}\quad \R^n,\]
and
\[\tilde q_\kappa(x,-y)=\tilde q_\kappa(x,y)\qquad \textrm{in}\quad \R^{n+1}.\]
\end{lemma}

\begin{proof}
It suffices to show the theorem in case $q_\kappa(x)=x^\alpha/\alpha!$, with $|\alpha|=\kappa$.
Here, $\alpha=(\alpha_1,...,\alpha_n)\in \mathbb N^n$, $x^\alpha=x_1^{\alpha_1}\cdots x_n^{\alpha_n}$, and $\alpha!=\alpha_1!\cdots \alpha_n!$.

\emph{Existence}. We claim that
\[E_\alpha(x,y):=\frac{x^\alpha}{\alpha!}-\Delta\left(\frac{x^\alpha}{\alpha!}\right)\frac{y^2}{2!}c_2+ \Delta^2\left(\frac{x^\alpha}{\alpha!}\right)\frac{y^4}{4!}c_4-...\]
with
\[c_{2k}:=\prod_{i=1}^k \frac{2i-1}{2i-2s},\]
is a polynomial of order $|\alpha|$ in $(x,y)$ that satisfies
\begin{itemize}
\item[(i)] $E_\alpha(x,0)=x^\alpha/\alpha!$ on $\{y=0\}$
\item[(ii)] $L_aE_\alpha=0$ in $\R^{n+1}$
\item[(iii)] $E_\alpha(x,-y)=E_\alpha(x,y)$ in $\R^{n+1}$.
\end{itemize}
Indeed, notice that
\[E_\alpha(x,y)=\sum_{k\geq0}p_{2k}(x)\frac{y^{2k}}{(2k!)},\]
where
\[p_{2k}(x):=(-1)^k c_{2k}\Delta^k\frac{x^\alpha}{\alpha!}.\]
Then, by definition of $c_{2k}$ we have
\[p_{2k+2}=-\frac{2k-1}{2k-2s}\Delta p_{2k}\]
or, equivalently,
\[\Delta p_{2k}+\left(1+\frac{a}{2k-1}\right)p_{2k+2}=0.\]
Thus, for $k\geq1$ we have
\[\begin{split}&\left(\Delta_{x,y}+\frac{a}{y}\partial_y\right)\left(p_{2k}(x)\frac{y^{2k}}{(2k)!}\right) \\
&=-\left(1+\frac{a}{2k-1}\right) p_{2k+2}(x)\frac{y^{2k}}{(2k)!}+ p_{2k}(x)\frac{y^{2k-2}}{(2k-2)!}+ a\,p_{2k}(x)\frac{y^{2k-2}}{(2k-1)!}\\
&=-\left(1+\frac{a}{2k-1}\right)p_{2k+2}(x)\frac{y^{2k}}{(2k)!} +\left(1+\frac{a}{2k-1}\right)p_{2k}(x)\frac{y^{2k-2}}{(2k-2)!}.
\end{split}\]
Adding up the previous expressions, and using that $\Delta p_{2k}=0$ when $2k>|\alpha|$, we get
\[\left(\Delta_{x,y}+\frac{a}{y}\partial_y\right)\sum_{k\geq0}p_{2k}(x)\frac{y^{2k}}{(2k)!}=\Delta_x p_0-\left(1+\frac{a}{2k-1}\right)p_2(x)=0,\]
and thus
\[L_a E_\alpha(x,y)=0,\]
as claimed.

\emph{Uniqueness}. By linearity of $L_a$, it suffices to show that the only extension of $q_\kappa=0$ is $\tilde q_\kappa=0$.
Assume that there is a polynomial $\tilde q_\kappa$ of degree $\kappa$ such that $L_a\tilde q_\kappa=0$ in $\R^{n+1}$, $\tilde q_\kappa$ is even in $y$, and $\tilde q_\kappa$ vanishes on $\{y=0\}$.
Then, we have
\[\tilde q_\kappa(x,y)=\sum_{k\geq1} p_{2k}(x)\frac{y^{2k}}{(2k)!},\]
for some polynomials $p_{2k}$ of order $\kappa-2k$.
Now, exactly as before, we have
\[\left(\Delta_{x,y}+\frac{a}{y}\partial_y\right)\left(p_{2k}(x)\frac{y^{2k}}{(2k)!}\right)=\Delta_x p_{2k}(x)\frac{y^{2k}}{(2k)!}+\frac{2k-2s}{2k-1}\,p_{2k}(x)\frac{y^{2k-2}}{(2k-2)!}.\]
Therefore, setting $p_0(x)\equiv0$, we have
\[\begin{split}
|y|^{-a}L_a\tilde q_\kappa(x,y)&=
 \left(\Delta_{x,y}+\frac{a}{y}\partial_y\right)\sum_{k\geq1}p_{2k}(x)\frac{y^{2k}}{(2k)!}\\
&=\sum_{k\geq1}\Delta_x p_{2k}(x)\frac{y^{2k}}{(2k)!} +\left(1+\frac{a}{2k-1}\right)\,p_{2k}(x)\frac{y^{2k-2}}{(2k-2)!}\\
&=\sum_{k\geq0} \left(\Delta_x p_{2k}(x)+\frac{2k-2s}{2k-1}\,p_{2k+2}(x)\right)\frac{y^{2k}}{(2k)!}.
\end{split}\]
This means that
\[p_{k+2}(x)=-\frac{2k-1}{2k-2s}\Delta_x p_{2k}(x)\qquad\textrm{for all}\quad k\geq0,\]
and since $p_0\equiv0$ this yields $p_{2k}\equiv0$ for all $k$.
Hence, $\tilde q_\kappa\equiv0$ in $\R^{n+1}$, as desired.
\end{proof}

We now give the:

\begin{proof}[Proof of Lemma \ref{analytic}]
We may assume that $x_0=0$.

Since $\varphi(x)$ is analytic near the origin, then we have
\[\varphi(x)=\sum_{\alpha\in \mathbb N^n} a_\alpha \frac{x^\alpha}{\alpha!},\]
with
\[|a_\alpha|\leq M^{|\alpha|}\]
for some $M$ independent of $\alpha$.

Let
\[E_\alpha(x,y):=\frac{x^\alpha}{\alpha!}-\Delta\left(\frac{x^\alpha}{\alpha!}\right)\frac{y^2}{2!}c_2+ \Delta^2\left(\frac{x^\alpha}{\alpha!}\right)\frac{y^4}{4!}c_4-...\]
with
\[c_{2k}:=\prod_{i=1}^k \frac{2i-1}{2i-2s},\]
be given by Lemma \ref{polynomials}.
Then, $E_\alpha$ satisfies:
\begin{itemize}
\item $E_\alpha(x,0)=x^\alpha/\alpha!$ on $\{y=0\}$
\item $L_aE_\alpha=0$ in $\R^{n+1}$
\item $E_\alpha(x,-y)=E_\alpha(x,y)$ in $\R^{n+1}$
\end{itemize}
It is easy to check that $|c_{2k}|\leq C\sqrt{k}$ and thus, in particular, $E_\alpha(x,y)$ is a polynomial of order $|\alpha|$ in $(x,y)$, with coefficients bounded by $M^\alpha$.

We now consider the function
\[\tilde \varphi(x,y)=\sum_{\alpha} a_\alpha E_\alpha(x,y).\]
By the above properties of $E_\alpha$, it is easy to check that the power series defining $\tilde\varphi$ converges in a neighborhood of the origin, and thus we are done.
\end{proof}

\section{$C^{k,\gamma}$ obstacles: preliminaries and a generalized Almgren frequency formula}
\label{sec6}

We start now our study of singular points in the obstacle problem for the fractional Laplacian in $\R^n$.
The obstacle $\varphi$ is assumed to be $C^{k,\gamma}$, with $k\geq2$ and $\gamma\in(0,1)$.

As before, given the solution $u(x)$ of the obstacle problem \eqref{pb}, we consider its extension $\widetilde u(x,y)$, for $y\geq0$, defined as the solution of
\[\left\{ \begin{array}{rcll}
\widetilde{u}(x,0)&=&u(x) &\textrm{on}\ \{y=0\},\\
L_a \widetilde{u}(x,y)&=& 0&\textrm{in}\ \{y>0\},
\end{array}\right.\]
where
\begin{equation}\label{eq:La}
L_a\widetilde u:=-\divv_{x,y}\bigl(|y|^{a}\nabla_{x,y}\widetilde u\bigr),\qquad a:=1-2s.
\end{equation}
We may extend the function $\widetilde u(x,y)$ to the whole $\R^{n+1}$ as $\widetilde u(x,-y)=\widetilde u(x,y)$.
Then, $\widetilde u$ solves~\eqref{sb}.

\addtocontents{toc}{\protect\setcounter{tocdepth}{1}} 

\subsection{Subtracting the Taylor polynomial}

We will prove a generalized Almgren frequency formula to study free boundary points where the blow-ups of our solution $u$ have homogeneity $\kappa<k+\gamma$.
In particular, we will study singular free boundary points, at which $\kappa=2m$ with $m\in\mathbb N$ and $2\leq \kappa\leq k$.

For this, given a free boundary point $x_0$ for our solution $\tilde u$, it will be important to replace $\widetilde u-\varphi$ with a suitable variant of it for which the $L_a$ operator is very small near $x_0$.

More precisely, given a free boundary point $x_0\in\Gamma(u)$ we define
\[\tilde \varphi(x,y):=\varphi(x)-q_k(x)+\tilde q_k(x,y),\]
where $q_k(x)$ the Taylor polynomial of $\varphi$ at $x_0$ of degree $k$, and $\tilde q_k(x,y)$ is the $s$-harmonic extension of $q_k$ given by Lemma \ref{polynomials}.
Then, we have $L_a \tilde q_k(x,y)=0$, $\tilde q_k(x,0)=q_k(x)$, $\tilde q_k(x,-y)=\tilde q_k(x,y)$, and $|\varphi(x)-q_k(x)|\leq C|x-x_0|^{k+\gamma}$.

Moreover, we define
\begin{equation}\label{v}
v^{x_0}(x,y):=\widetilde{u}(x_0+x,y)-\tilde \varphi(x_0+x,y).
\end{equation}
Notice that
\[v^{x_0}(x,0)=u(x_0+x)-\varphi(x_0+x)\]
for all $x\in \R^n$, hence $v^{x_0}(x,0)\geq0$.
Furthermore, we have $v^{x_0}(x,-y)=v^{x_0}(x,y)$, and
\begin{eqnarray}\label{ext2}
|L_a v^{x_0}(x,y)|&=&|y|^{a}\bigl|\Delta_x(\varphi-q_k)(x_0+x)\bigr|\nonumber\\
&\leq& C\,|y|^{a}|x|^{k+\gamma-2},
\end{eqnarray}
for every $(x,y)\in \mathbb{R}^{n+1}\setminus\left\{(x,0):\, v^{x_0}(x,0)=0\right\}$.
It is also important to observe that $v^{x_0}$ depends continuously on $x_0$.

Throughout the rest of the paper, we shall use $v$ instead of $v^{x_0}$ whenever the dependence on point $x_0$ is clear.
Also, as before, we denote by $\BB$ the ball in $\mathbb{R}^{n+1}$ of radius $r$ centered at $(0,0)$.

\subsection{Almgren-type frequency formula}

We now establish the following generalized Almgren-type frequency formula, which extends the ones in \cite{CSS,GP,BFR,CDS}.

\begin{prop}[Generalized Almgren's frequency formula]\label{Almgren}
Let $u$ solve the obstacle problem for the fractional Laplacian, with $\varphi\in C^{k,\gamma}(\R^n)$,  $k\geq2$ and $\gamma\in(0,1)$.

Let $x_0\in \Gamma(u)$ be a free boundary point, let $v=v^{x_0}$ be defined as in \eqref{v},
and set
\begin{equation}\label{Hx_0}
\tilde{H}^{x_0}(r,v):=\int_{\partial \BB}{|y|^{a}v^2 }.
\end{equation}
Let $\theta\in(0,\gamma)$.
Then there exist constants $C_0,\, r_0>0$, independent of $x_0$, such that the function
\begin{equation}\label{generalized-frequency}
r\mapsto \Phi^{x_0}(r,v):=\bigl(r+C_0\,r^{1+\theta}\bigr)\,\frac{d}{dr}\log \max\left\{\tilde{H}^{x_0}(r,v),\ r^{n+a+2(k+\gamma-\theta)}\right\},
\end{equation}
is monotone nondecreasing on $(0,r_0)$.
In particular the limit $\lim_{r\downarrow0}\Phi^{x_0}(r,v):=\Phi^{x_0}(0^+,v)$ exists.
\end{prop}

To simplify the notation we shall denote $\Phi=\Phi^{x_0}$ and $\tilde{H}=\tilde{H}^{x_0}$ when no confusion is possible.

\begin{rmrk}
Similar Almgren-type frequency formulas have been previously established in \cite{CSS}, \cite{GP}, \cite{BFR}, and \cite{CDS}.
In particular, the formula established in \cite[Theorem 6.2]{CDS} corresponds to the case $k=1$ and $\gamma=s+\delta$ in our Proposition \ref{Almgren}.
\end{rmrk}

Before proving Proposition \ref{Almgren} we establish an auxiliary lemma that provides us with some upper bounds for the functions
\begin{equation}\label{HG}
\tilde{G}(r,v):=\int_{\BB}{|y|^{a}v^2 }
\qquad\mbox{and}\qquad   \tilde{H}(r,v)=\int_{\partial \BB}{|y|^{a}v^2 }=\tilde{G}'(r,v).
\end{equation}

\begin{lemma}\label{auxiliar_frecuency}
Let $v$ be as in Proposition \ref{Almgren}, and define
\begin{equation}\label{D}
\tilde{D}(r,v):=\int_{\BB}{|y|^{a}|\nabla v|^{2}}.
\end{equation}
Then there exist constants $\bar{C},\, \bar{r}>0$, independent of $x_0$, such that
\begin{equation}\label{H1}
\tilde{H}(r,v)\leq \bar{C}\left(r\,\tilde{D}(r,v)+ r^{n+a+2(k+\gamma)}\right)\qquad \textrm{for all}\ r\in (0,\bar{r}),
\end{equation}
and
\begin{equation}\label{G1}
\tilde{G}(r,v)\leq \bar{C}\left(r^2\,\tilde{D}(r,v)+ r^{n+a+1+2(k+\gamma)}\right)\qquad \textrm{for all}\ r\in (0,\bar{r}).
\end{equation}
\end{lemma}

\begin{proof}
Let $r\in(0,1)$.
By \cite[Lemma 2.9]{CSS} it follows that
$$v(0)\geq \frac{1}{\omega_{n+a}r^{n+a}}\int_{\partial \BB}{|y|^{a} v}- C\, r^{k+\gamma},$$
so one can follow the proof of \cite[Lemma 2.13]{CSS} to get
$$
\int_{\partial \BB}{|y|^{a} v^2}\leq C\, r\int_{\BB}{|y|^{a}|\nabla v|^{2}}+ C \,r^{(n+a)+2(k+\gamma)}.
$$
The previous inequality proves \eqref{H1}, and integrating it with respect to $r$ we obtain \eqref{G1}.
\end{proof}

We now prove the main result of this section.

\begin{proof}[Proof of Proposition \ref{Almgren}]
As observed in \cite[Proof of Theorem 3.1]{CSS},
in order to prove that $\Phi(r,v)$ is increasing one can
concentrate in each of the two values for the maximum separately.

In case
\[\Phi(r,v)=\bigl(r+C_0r^{1+\theta}\bigr)\,\frac{d}{dr}\log r^{n+a+2(k+\gamma-\theta)}=(1+C_0r^\theta)\bigl(n+a+2(k+\gamma-\theta)\bigr),\]
the function $\Phi(\cdot,v)$ is clearly monotonically increasing.
Thus, we only need to prove that $\Phi'(r,v)\geq0$ in case $\tilde{H}(r,v)>r^{n+a+2(k+\gamma-\theta)}$.

First, notice that
$$\tilde{H}(r,v)=r^{n+a}\int_{\partial\BBone}{|y|^{a}v^{2}(rx,ry)},$$
and thus
\begin{eqnarray}
\tilde{H}'(r,v)&=&(n+a)\,\frac{\tilde{H}(r,v)}{r}+2\, r^{n+a}\int_{\partial\BBone}{|y|^{a}v(rx,ry)\,\nabla v(rx,ry)\cdot(x,y)}\nonumber\\
&=&(n+a)\,\frac{\tilde{H}(r,v)}{r}+2\,\mathcal{I}(r,v),\label{h'}
\end{eqnarray}
where
\begin{eqnarray}
\mathcal{I}(r,v)&:=&\int_{\partial \BB}{|y|^{a} v \,v_{\nu}}=\tilde{D}(r,v)+\int_{\BB}{v\, \divv (|y|^{a}\nabla v)}\nonumber \\
&=&\tilde{D}(r,v)-\int_{\BB}{v\, L_a v}.
\label{i}
\end{eqnarray}
Hence
\begin{equation}\label{mastarde}
\Phi(r,v)=(n+a)\,(1+C_0r^\theta)+2\,r\,(1+C_0r^\theta)\,\frac{\mathcal{I}(r,v)}{\tilde{H}(r,v)},
\end{equation}
and it is enough to show that $r\,(1+C_0r)\,\frac{\mathcal{I}(r,v)}{\tilde{H}(r,v)}$ is monotone or, equivalently,
\[\frac{d}{dr}\log\left( r\,(1+C_0r^\theta)\,\frac{\mathcal{I}(r,v)}{\tilde{H}(r,v)} \right) \geq0.\]
To show this, we notice that that, since
\[\tilde{D}'(r,v)=\frac{n+a-1}{r}\,\tilde{D}(r,v)-\frac{2}{r}\int_{\BB}{\bigl((x,y)\cdot\nabla v\bigr)\, \divv(|y|^{a}\nabla v)}+2\int_{\partial \BB}{|y|^{a}v_{\nu}^{2}},\]
it follows by \eqref{i} that
\begin{eqnarray*}
\mathcal{I}'(r,v)&=&\frac{n+a-1}{r}\,\mathcal{I}(r,v)-\frac{n+a-1}{r}\int_{\BB}v\, \divv (|y|^{a}\nabla v)\\
& &-\,\frac{2}{r}\int_{\BB}{\bigl((x,y)\cdot\nabla v\bigr)\, \divv(|y|^{a}\nabla v)}+2\int_{\partial \BB}{|y|^{a}v_{\nu}^{2}}+\int_{\partial \BB}{v\, \divv (|y|^{a}\nabla v)}.
\end{eqnarray*}
Thus, recalling that $\divv (|y|^{a}\nabla v)=-L_a v$, by \eqref{h'} and the Cauchy-Schwarz inequality, we obtain
\[\begin{split}
\frac{d}{dr}\log\left( r\,(1+C_0r^\theta)\,\frac{\mathcal{I}(r,v)}{\tilde{H}(r,v)} \right)&=\\
&= \frac{\theta C_0r^{\theta-1}}{1+C_0r^\theta} + \frac1r +\frac{\mathcal{I}'(r,v)}{\mathcal{I}(r,v)}-\frac{\tilde{H}'(r,v)}{\tilde{H}(r,v)}\\
&= \frac{\theta C_0r^{\theta-1}}{1+C_0r^\theta}+ 2\left(\frac{\int_{\partial\BB}|y|^av_\nu^2}{\mathcal{I}(r,v)}-\frac{\mathcal{I}(r,v)}{\tilde{H}(r,v)}\right) -\mathcal{E}(r,v)\\
&\geq\frac{\theta C_0r^{\theta-1}}{1+C_0r^\theta}-\mathcal{E}(r,v),
\end{split}\]
where
\begin{equation}\label{E}
\mathcal{E}(r,v):=\frac{-\frac{1}{r}\left(\int_{\BB}{\left[2\bigl((x,y)\cdot\nabla v\bigr)+(n+a-1)v\right]L_a v}\right)+\int_{\partial \BB}{vL_a v}}{\mathcal{I}(r,v)}.
\end{equation}
Now, since
\[\frac{C_0r^{\theta-1}}{1+C_0r^\theta}\geq \frac{C_0}{2}r^{\theta-1}\]
provided $r \leq r_0$, with $r_0$ small enough,
and since $C_0$ can be chosen arbitrarily large, to conclude the proof we only need to show that
\begin{equation}\label{we-want}
\mathcal{E}(r,v)\leq Cr^{\theta-1}.
\end{equation}
For this, we will estimate separately each term of the numerator and denominator of $\mathcal{E}$.

Since $v$ satisfies \eqref{ext2} outside $\{v=0\}\cap \{y=0\}$ while $v\,L_av=0$ on the set $\{v=0\}\cap \{y=0\}$
(because $L_av$ is a signed measure),
using the Cauchy-Schwarz inequality, \eqref{i}, and Lemma \ref{auxiliar_frecuency}, we obtain that
\begin{eqnarray}
\mathcal{I}(r,v)&=&\tilde{D}(r,v)-\int_{\BB} v\,L_av =\tilde{D}(r,v)-\int_{\BB\setminus \{v=0\}} v\,L_av\nonumber\\
&\geq& \tilde{D}(r,v)-2\left(\int_{\BB}{|y|^{a} v^2}\right)^{1/2}\left(\int_{\BB\setminus \{v=0\}}{|y|^{-a}\left(L_av\right)^2}\right)^{1/2}\label{milveces}\\
&\geq& \tilde{D}(r,v)-2\,\tilde{G}(r,v)^{1/2}\left(\int_{\BB}{|y|^{a}|x|^{2(k+\gamma-2)}}\right)^{1/2}\nonumber\\
&\geq& \tilde{D}(r,v)-2\,\tilde{G}(r,v)^{1/2}r^{\frac{n+1+a}{2}+k+\gamma-2}\nonumber\\
&\geq& \tilde{D}(r,v)
-C\left(\tilde{D}(r,v)^{1/2}r^{\frac{n+1+a}{2}+k+\gamma-1}+r^{(n+1+a)+2(k+\gamma-1)}\right).  \label{denominator}
\end{eqnarray}
Similarly, since $(x,y)\cdot\nabla v=0$ on the set $\{v=0\}\cap \{y=0\}$ we get
\begin{equation}\label{numerador1}
\left|\frac{1}{r}\int_{\BB}{\bigl((x,y)\cdot\nabla v\bigr)L_a v}\right|
    \leq C\, \tilde{D}(r,v)^{1/2} r^{\frac{n+1+a}{2}+k+\gamma-2}
\end{equation}
and
\begin{equation}\label{numerador2}
\max\left\{\left|\frac{1}{r}\int_{\BB} v\, L_av\right|, \left|\int_{\partial \BB}{v\,L_av}\right|\right\}
    \leq C\, \Bigl(\tilde{D}(r,v)^{1/2} r^{\frac{n+1+a}{2}+k+\gamma-2}+r^{n+a+2(k+\gamma-1)}\Bigr).
\end{equation}
Thus, it follows by \eqref{E}-\eqref{numerador2} that
\[|\mathcal{E}(r,v)| \leq C\,
\frac{ \tilde{D}(r,v)^{1/2}r^{\frac{n+1+a}{2}+k+\gamma-2}+r^{n+a+2(k+\gamma-1)} }{ \tilde{D}(r,v)-C\left(\tilde{D}(r,v)^{1/2}r^{\frac{n+1+a}{2}+k+\gamma-1}+r^{(n+a+1)+2(k+\gamma-1)}\right)}.\]
Now, recalling that $\tilde{H}(r,v)>r^{n+a+2(k+\gamma-\theta)}$,
thanks to $\eqref{H1}$ we get
\[\tilde{D}(r,v)\geq c \,r^{n+a+2(k+\gamma-\theta)-1}.\]
This yields
\[\begin{split}
|\mathcal{E}(r,v)|&\leq C\,\frac{\sqrt{\tilde{D}(r,v)}r^{\frac{n+a+1}{2}+k+\gamma-2}+r^{n+a+2(k+\gamma-1)}}{\frac12 \tilde{D}(r,v)}\\
&= C\,\frac{r^{\frac{n+a+1}{2}+k+\gamma-2}}{\sqrt{\tilde{D}(r,v)}}+ C\frac{r^{n+a+2(k+\gamma-1)}}{\tilde{D}(r,v)}\\
&\leq C\,\frac{r^{\frac{n+a+1}{2}+k+\gamma-2}}{r^{\frac{n+a+1}{2}+k+\gamma-\theta-1}}+ C\frac{r^{n+a+2(k+\gamma-1)}}{r^{n+a+2(k+\gamma-\theta)-1}}\\
&=Cr^{\theta-1}+Cr^{2\theta-1}\leq Cr^{\theta-1}.
\end{split}\]
Therefore, \eqref{we-want} is proved, as desired.
\end{proof}

\subsection{Growth near the free boundary and blow-ups}

We establish here some results that will be needed later.
First, we show the following.

\begin{lemma}\label{bound-H}
Let $u$ be a solution of the obstacle problem for the fractional Laplacian, with $\varphi\in C^{k,\gamma}(\R^n)$,  $k\geq2$ and $\gamma\in(0,1)$.

Let $x_0\in \Gamma(u)$ be a free boundary point, $v=v^{x_0}$ be defined as in \eqref{v}, and $\Phi^{x_0}(r,v)$ be defined as in \eqref{generalized-frequency}.

Suppose that
\[\Phi^{x_0}(0+,v)=n+a+2\kappa,\qquad \textrm{with}\quad \kappa<k+\gamma.\]
Then,
\begin{equation}\label{growth-H-kappa}
\tilde H(r,v)=\int_{\partial \BB}{|y|^{a}v^2 }\leq Cr^{n+a+2\kappa}
\end{equation}
for $0<r<r_0$.
Moreover, for any $\varepsilon>0$ there exists $r_{\varepsilon,x_0}>0$ such that
\begin{equation}\label{growth-H-kappa-below}
\tilde H(r,v)=\int_{\partial \BB}{|y|^{a}v^2 }\geq Cr^{n+a+2\kappa+\epsilon}
\end{equation}
for $0<r<r_{\varepsilon,x_0}$.
\end{lemma}

\begin{proof}
Since $\Phi$ is monotone increasing in $r$, it follows from the definition of $\kappa$ and $\Phi$ that
\[n+a+2\kappa\leq r(1+C_0r^\theta)\frac{d}{dr}\log\max\left\{\tilde{H}^{x_0}(r,v),\ r^{n+a+2(k+\gamma-\theta)}\right\}\]
for $r\in(0,r_0)$.
Therefore,
\[\frac{n+a+2\kappa}{r}-\frac{(n+a+2\kappa)C_0r^{\theta-1}}{1+C_0r^\theta}=\frac{n+a+2\kappa}{r(1+C_0r^\theta)}\leq \frac{d}{dr}\log\max\left\{\tilde{H}^{x_0}(r,v),\ r^{n+a+2(k+\gamma-\theta)}\right\}\]
Integrating from $r$ to $r_0$, we get
\[\log r^{n+a+2\kappa}+C_1\geq \log\max\left\{\tilde{H}^{x_0}(r,v),\ r^{n+a+2(k+\gamma-\theta)}\right\},\]
and thus \eqref{growth-H-kappa} follows.

The proof of \eqref{growth-H-kappa-below} is analogous.
\end{proof}

We will also need the following.

\begin{lemma}\label{bound-???}
Let $u$ be a solution of the obstacle problem for the fractional Laplacian, with $\varphi\in C^{k,\gamma}(\R^n)$,  $k\geq2$ and $\gamma\in(0,1)$.

Let $x_0\in \Gamma(u)$ be a free boundary point, $v=v^{x_0}$ be defined as in \eqref{v}, and $\Phi^{x_0}(r,v)$ be defined as in \eqref{generalized-frequency}.

Suppose that
\[\Phi^{x_0}(0+,v)=n+a+2\kappa,\qquad \textrm{with}\quad \kappa<k+\gamma.\]
Then,
\begin{equation}\label{???}
r\left|\int_{\BB}v\,L_av\right|\leq Cr^{n+a+\kappa+(k+\gamma)}
\end{equation}
for $0<r<r_0$.
\end{lemma}

\begin{proof}
Let $\tilde{G}(r,v)$ be given by \eqref{HG}.
Since $\tilde{G}'(r,v)=\tilde{H}(r,v)$, then the previous Lemma yields
\[\tilde{G}(r,v)\leq Cr^{n+a+2\kappa+1}.\]
Using this, the Cauchy-Schwarz inequality, and \eqref{ext2}, we get
\[\begin{split}
\left|\int_{\BB}v\,L_av\right|&\leq C r^{k+\gamma-2}\int_{\BB}|v|\,|y|^a\leq C r^{k+\gamma-2} \left(\int_{\BB}|y|^a\right)^{1/2} \left(\int_{\BB}v^2|y|^a\right)^{1/2}\\
&\leq C r^{k+\gamma-2} r^{\frac{n+a+1}{2}} r^{\frac{n+a+2\kappa+1}{2}}=Cr^{n+a+\kappa+(k+\gamma)-1}.\end{split}\]
Thus, \eqref{???} follows.
\end{proof}

We next show the following.

\begin{prop}\label{blow-ups}
Let $u$ be a solution of the obstacle problem \eqref{pb}, with $\varphi\in C^{k,\gamma}(\R^n)$,  $k\geq2$ and $\gamma\in(0,1)$.

Let $x_0\in \Gamma(u)$ be a free boundary point, $v=v^{x_0}$ be defined as in \eqref{v}, and $\Phi^{x_0}(r,v)$ be defined as in \eqref{generalized-frequency}.

Suppose that
\[\Phi^{x_0}(0+,v)=n+a+2\kappa,\qquad \textrm{with}\quad \kappa<k+\gamma,\]
and let
\[v_r(x,y)=v_r^{x_0}(x,y):=\frac{v(x_0+rx,ry)}{d_r},\]
where
\[d_r:=\left(\frac{\tilde H^{x_0}(r,v)}{r^{n+a}}\right)^{1/2}.\]
Then, up to a subsequence, $v_r$ converge as $r\to 0^+$ to a homogeneous function $v_0$ which is a global solution to the zero obstacle problem and it is homogeneous of degree $\kappa$.
\end{prop}

\begin{proof}
The proof is a minor modification of that in \cite[Lemma 6.2]{CSS}; see also \cite[Proposition 5.3]{BFR}.
For completeness, we sketch the proof here.

Let $r_1>0$ be such that $\tilde H(r,v)>r^{n+a+2(k+\gamma-\theta)}$ for $r\in(0,r_1)$.
Then, the inequality $\Phi^{x_0}(r,v)\leq \Phi^{x_0}(r_1,v)$ yields
\[r\frac{\tilde H'(r,v)}{\tilde H(r,v)}\leq C.\]
Using \eqref{h'} and \eqref{i}, we find
\[r\frac{\int_{\BB}|y|^a|\nabla v|-\int_{\BB}v\,L_av}{\tilde H(r,v)}\leq C.\]
Now, it follows from \eqref{???} and \eqref{growth-H-kappa-below} that
\begin{equation}\label{1234}
r\frac{\int_{\BB}v\,L_av}{\tilde H(r,v)}\leq Cr^{k+\gamma-\kappa-\epsilon}\rightarrow0
\end{equation}
as $r\to0^+$, and hence
\[r\frac{\int_{\BB}|y|^a|\nabla v|}{\tilde H(r,v)}\leq C\]
for $r>0$ small enough.

By definition of $v_r$, the previous inequality is equivalent to
\[\int_{B(1)}|y|^a|\nabla v_r|\leq C.\]
Also, it follows from the definition of $v_r$ that $\|v_r\|_{L^2(\partial B(1))}=1$.

This implies that the sequence $\{v_r\}$ is uniformly bounded in $H^1(\BBone,|y|^a)$, so that (up to a subsequence) the functions $v_r$ converge to $v_0$ weakly in $H^1(\BBone,|y|^a)$, strongly in $L^2(\partial\BBone)$, and a.e. in $\BBone$.

Moreover, by \eqref{ext2} and \eqref{growth-H-kappa-below}, we have that
\begin{equation}\label{eq-blow-ip}
|L_a v_r(x,y)|\leq C\frac{r^2}{d_r} r^{k+\gamma-2}|y|^a|x|^{k+\gamma-2}\leq Cr^{k+\gamma-\kappa-\varepsilon/2}|y|^a|x|^{k+\gamma-2}
\end{equation}
in $\R^{n+1}\setminus(\{v_r=0\}\cap \{y=0\})$.
Since $\kappa<k+\gamma$, then we may take $\epsilon>0$ such that $k+\gamma-\kappa-\varepsilon/2>0$.

Thus, by $C^{1,\alpha}$ estimates (see \cite{CSS}), we get that the sequence $v_r$ is uniformly bounded in $C^{1,\alpha}_{\rm loc}(\BBone)$, and hence $v_r\to v_0$ in $C^1_{\rm loc}(\BBone)$.

Now, letting $r\to0$ in \eqref{eq-blow-ip}, we find that $v_0$ is a global solution to the zero obstacle problem, with $\|v_0\|_{L^2(\partial B(1))}=1$.

Finally, let us see that $v_0$ is homogeneous of degree $\kappa$.
For this, we consider the ``pure'' Almgren frequency formula
\[\tilde N(r,v):=\frac{r\int_{\BB}|y|^a|\nabla v|^2}{\int_{\partial\BB}|y|^av^2}=\frac{r\,\tilde D(r,v)}{\tilde H(r,v)}.\]
Then, for $r>0$ small enough we have $\tilde H(r,v)>r^{n+a+2(k+\gamma-\theta)}$ and thus
\[\Phi(r,v)=(1+C_0r^\theta)\left((n+a)+2\tilde N(r,v)-2r\frac{\int_{\BB}v\,L_av}{\tilde H(r,v)}\right).\]
Using \eqref{1234} we see that
\[\Phi(0+,v)=(n+a)+2\tilde N(0+,v).\]
Therefore,
\[\tilde N(\rho,v_0)=\lim_{r\to0^+}\tilde N(\rho,v_r)=\lim_{r\to0^+}\tilde N(r\rho,v)=\kappa.\]
This means that the Almgren's frequency formula $\tilde N(\cdot,v_0)$ is constant, and hence by Theorem \ref{T:almgrenEO} $v_0$ is homogeneous.
\end{proof}

\section{Singular points and Monneau-type monotonicity formulas}
\label{sec7}

We start here our study of the singular set of the free boundary for problem \eqref{pb}.

Given $\kappa<k+\gamma$, we define
\[\Gamma_\kappa(u):=\{x_0\in \Gamma(u)\,:\, \Phi(0+,v^{x_0})=n+a+2\kappa\}.\]

Moreover, recall that a free boundary point $x_0\in \Gamma(u)$ is said to be singular if
\[\lim_{r\downarrow 0}\frac{\bigl|\{u=\varphi\}\cap B_r(x_0)\bigr|}{|B_r(x_0)|}=0.\]
We will denote $\Sigma(u)$ the set of singular points, and
\[\Sigma_\kappa(u):=\Gamma_\kappa(u)\cap \Sigma(u).\]
We next start the study of the structure and regularity of the singular set.
For this, we first need a characterization of singular points similar to Proposition \ref{characterization} above.

\begin{prop}
Let $u$ be a solution of the obstacle problem \eqref{pb}, with $\varphi\in C^{k,\gamma}(\R^n)$,  $k\geq2$ and $\gamma\in(0,1)$.
Let $x_0\in \Gamma(u)$ be a free boundary point, $v=v^{x_0}$ be defined as in \eqref{v}.
Assume $x_0\in \Gamma_\kappa(u)$, with $\kappa<k+\gamma$.
Then, the following statements are equivalent:
\begin{itemize}
\item[(i)] $x_0\in \Sigma_\kappa(u)$
\item[(ii)] any blow-up of $v$ at the origin is a nonzero homogeneous polynomial $p_\kappa(x,y)$ of degree $\kappa$ satisfying
    \[L_ap_\kappa=0,\qquad p_\kappa(x,0)\geq0,\qquad p_\kappa(x,-y)=p_\kappa(x,y)\]
\item[(iii)] $\kappa=2m$ for some $m\in \{1,2,3,...\}$
\end{itemize}
\end{prop}

\begin{proof}
The proof is a minor modification of that of Proposition \ref{characterization} and is therefore omitted.
\end{proof}

We next start the study of the regularity of the set $\Sigma_\kappa(u)$, with $\kappa=2m$, $\kappa\leq k$.

\subsection{Monneau-type monotonicity formulas}

Throughout this Section, we denote
\begin{equation}\label{kappa}
\kappa=2m,\qquad \textrm{with}\quad m\in \mathbb N \quad \textrm{and}\quad 2\leq \kappa\leq k.
\end{equation}
By the previous Proposition, $x_0\in\Gamma_\kappa(u)$ is a singular point if and only if its frequency is \eqref{kappa}.
Recall also that $\varphi\in C^{k,\gamma}(\R^n)$, with $k\geq2$ and $\gamma\in(0,1)$.

As before, we denote by $\tilde{\mathfrak P}_{a,\kappa}^+=\tilde{\mathfrak P}_{a,2m}^+$ the set of $\kappa$-homogeneous polynomials $p_\kappa(x,y)$ satisfying
\[L_ap_\kappa=0\quad \textrm{in}\ \R^{n+1},\qquad p_\kappa\geq0\quad \textrm{for}\ \{y=0\},\qquad p_\kappa(x,y)=p_\kappa(x,-y).\]

\begin{prop}\label{monneau}
Let $u$ be a solution of the obstacle problem \eqref{pb}, with $\varphi\in C^{k,\gamma}(\R^n)$,  $k\geq2$ and $\gamma\in(0,1)$.

Let $x_0\in \Gamma_\kappa(u)$, with $\kappa$ as in \eqref{kappa}, let $v=v^{x_0}$ be defined as in \eqref{v}, and let $p_\kappa\in \tilde{\mathfrak P}_{a,\kappa}^+$.
There exists $C_M>0$ such that the quantity
\[\tilde{\mathcal{M}}^{x_0}(r,v, p_\kappa):=\frac{1}{r^{n+a+2\kappa}}\int_{\partial \BB}|y|^{a}\bigl(v(x,y)-p_\kappa(x-x_0,y)\bigr)^2\]
satisfies
\[\frac{d}{dr}\tilde{\mathcal{M}}^{x_0}(r,v, p_\kappa)\geq -C_M\,r^{\gamma-1}\qquad \forall\,r \in (0,r_0),\]
where $r_0$ is as in Proposition \ref{Almgren}.
\end{prop}

To prove Proposition \ref{monneau} we need the following bound on a Weiss-type energy.

\begin{lemma}\label{weiss}
Let $v$ be as in Proposition \ref{monneau}, and let $r_0$ be as in Proposition \ref{Almgren}.
Then there exists a constant $C_W>0$ such that the following holds:

The quantity
\[\tilde{\mathcal{W}}^{x_0}(r,v):=\frac{1}{r^{n+a-1+2\kappa}}\int_{\BB}|y|^{a}|\nabla v|^2 -\frac{\kappa}{r^{n+a+2\kappa}}\int_{\partial \BB}|y|^{a}v^2\]
satisfies
\begin{equation}\label{ineq-weiss}
\tilde{\mathcal{W}}^{x_0}(r,v)\geq -C_W\,r^{\gamma}\qquad \forall\, r \in (0,r_0).
\end{equation}
\end{lemma}

\begin{proof}
We will use the Almgren-type monotonicity formula proved above.
Throughout this proof, we denote $\Phi=\Phi^{x_0}$, $\tilde{H}=\tilde{H}^{x_0}$, $\mathcal{I}=\mathcal{I}^{x_0}$, and $\tilde{D}=\tilde{D}^{x_0}$.

By definition of $\Gamma_{\kappa}(u)$, we have $\Phi(0^{+},v)=n+a+2\kappa$.
Thus, by the monotonicity of $\Phi(\cdot,v)$ on $(0,r_0)$ (see Proposition \ref{Almgren}), for any $r \in (0,r_0)$ we have that either
\begin{equation}\label{1}
\Phi(r,v)=\bigl(r+C_0r^{1+\theta}\bigr)\,\frac{\tilde{H}'(r,v)}{\tilde{H}(r,v)}\geq n+a+2\kappa
\end{equation}
or
\begin{equation}\label{2}
\tilde{H}(r,v)\leq r^{n+a+2(k+\gamma-\theta)}.
\end{equation}
We split the proof of \eqref{ineq-weiss} in two cases.

\vspace{2mm}

\noindent
\emph{- Case 1}. If \eqref{1} holds then it follows by \eqref{h'} that
$$
(r+C_0r^2)\,\left(\frac{n+a}{r}+2\,\frac{{\mathcal{I}(r,v)}}{\tilde{H}(r,v)}\right)\geq n+a+2\kappa,
$$
that is
$$
n+a+2\,r\,\frac{{\mathcal{I}(r,v)}}{\tilde{H}(r,v)}\geq
 n+a+2\kappa- C_0r^2\,\left(\frac{n+a}{r}+2\,\frac{{\mathcal{I}(r,v)}}{\tilde{H}(r,v)}\right),
$$
and since $r\left(\frac{n+a}{r}+2\,\frac{{\mathcal{I}(r,v)}}{\tilde{H}(r,v)}\right)\leq \Phi(r,v)\leq C$ we get
$$
r\,\frac{{\mathcal{I}(r,v)}}{\tilde{H}(r,v)}\geq \kappa-\frac{C_0}{2}\,r^2\left(\frac{n+a}{r}+2\frac{{\mathcal{I}(r,v)}}{\tilde{H}(r,v)}\right)\geq \kappa-C\,r.
$$
Hence, using \eqref{i}, \eqref{growth-H-kappa}, and \eqref{???}, we obtain
\[r\tilde{D}(r,v)-\kappa \tilde{H}(r,v)\geq -Cr\tilde{H}(r,v)-Cr^{n+a+\kappa+k+\gamma}\geq -Cr^{n+a+2\kappa+\gamma}.\]
Here we used that $\kappa\leq k$ and $\gamma\leq1$.
This gives
\[\tilde{\mathcal{W}}^{x_0}(r,v)\geq -C\,\frac{r^{n+a+2\kappa+\gamma}}{r^{n+a+2\kappa}}\geq -C\,r^{\gamma},\]
as desired.

\vspace{2mm}

\noindent
\emph{- Case 2}. If \eqref{2} holds then we simply use that $\tilde{D}(r,v)\geq0$ to obtain
\[\frac{1}{r^{n+a-1+2\kappa}}{\tilde{D}(r,v)}-\frac{2}{r^{n+a+2\kappa}}\tilde{H}(r,v)\geq -C\,r^{2(k+\gamma-\theta-\kappa)}\geq -C\,r^{\gamma},\]
provided that we take $\theta\leq \gamma/2$.
This concludes the proof of \eqref{ineq-weiss}.
\end{proof}

We can now prove Proposition \ref{monneau}.

\begin{proof}[Proof of Proposition \ref{monneau}]
Set $w:=v-p_\kappa$ and let us use the notation $z=(x,y) \in \R^{n+1}$.
We have
\begin{eqnarray}
\frac{d}{dr}\tilde{\mathcal{M}}^{x_0}(r,v, p_{\kappa})&=&\frac{d}{dr}\int_{\partial \BBone}\frac{|y|^{a}|w(rz)|^2}{r^{2\kappa}}\nonumber\\
&=&\int_{\partial \BBone}|y|^{a}\,\frac{2\,w(rz)\bigl(rz\cdot\nabla w(rz)-2\,w(rz)\bigr)}{r^{2\kappa+1}} \nonumber\\
&=&\frac{2}{r^{n+a+2\kappa+1}}\int_{\partial \BB}|y|^{a}w\bigl(z\cdot\nabla w-2\,w\bigr).\label{d dr M}
\end{eqnarray}
We now claim that
\begin{equation}\label{claim-weiss}
\tilde{\mathcal{W}}^{x_0}(r,v)\leq  \frac{1}{r^{n+a+2\kappa}}\int_{\partial \BB}|y|^{a}w(z\cdot \nabla w-2w)+C\,r^{\gamma}.
\end{equation}
Indeed, since $L_ap_{\kappa}=0$ in $\R^{n+1}$ and $p_{\kappa}$ is $\kappa$-homogeneous,
we then have $\tilde{\mathcal{W}}^{x_0}(r,p_{\kappa})\equiv0$.
Hence, using again that $L_ap_{\kappa}=0$ and that $z\cdot \nabla p_{\kappa}=\kappa\,p_{\kappa}$ (by homogeneity), integrating by parts we get
\begin{eqnarray}
\tilde{\mathcal{W}}^{x_0}(r,v)&=&\tilde{\mathcal{W}}^{x_0}(r,v)-\tilde{\mathcal{W}}^{x_0}(r,p_{\kappa}) \nonumber\\
&=&\frac{1}{r^{n+a-1+2\kappa}}\int_{\BB}|y|^{a}\left(|\nabla w|^2+2\,\nabla w\cdot\nabla p_{\kappa}\right)
        -\frac{\kappa}{r^{n+a+2\kappa}}\int_{\partial \BB}|y|^{a}\left(w^2+2\,w\,p_{\kappa}\right)\nonumber\\
&=&\frac{1}{r^{n+a-1+2\kappa}}\int_{\BB}|y|^{a}|\nabla w|^2
        +\frac{1}{r^{n+a-1+2\kappa}}\int_{\BB}2\,w\,L_ap_{\kappa}\nonumber\\
&& +\frac{1}{r^{n+a+2\kappa}}\int_{\partial \BB}|y|^{a}w\left(2z\cdot \nabla p_{\kappa}-2\kappa\,p_{\kappa}\right)
        -\frac{\kappa}{r^{n+a+2\kappa}}\int_{\partial \BB}|y|^{a}w^2\nonumber\\
&=&\frac{1}{r^{n+a-1+2\kappa}}\int_{\BB}|y|^{a}|\nabla w|^2
        -\frac{\kappa}{r^{n+a+2\kappa}}\int_{\partial \BB}|y|^{a}w^2.
\label{split}
\end{eqnarray}
Using now that $p_{\kappa}\leq C\,r^{\kappa}$ in $\BB$, that $L_ap_{\kappa}=0$, the growth of $v$ in $\BB$, and \eqref{ext2}, we get
\[\left|\int_{\BB}{w\, L_a w}\right|=\left|\int_{\BB}{(p_2-v)\, L_a v}\right|\leq C\, r^{n+a+\kappa+k+\gamma-\theta}.\]
Integrating by parts in \eqref{split} and using the previous bound, we conclude that
\begin{eqnarray*}
\tilde{\mathcal{W}}^{x_0}(r,v)&=&\frac{1}{r^{n+a-1+2\kappa}}\int_{\BB}w\,L_aw
                            +\frac{1}{r^{n+a+2\kappa}}\int_{\partial \BB}|y|^{a}w\,(z\cdot \nabla w-\kappa\,w)\\
&\leq& \frac{1}{r^{n+a+2\kappa}}\int_{\partial \BB}|y|^{a}w(z\cdot \nabla w-\kappa\, w)+ C\,r^{\gamma},
\end{eqnarray*}
and \eqref{claim-weiss} follows.

Finally, combining \eqref{d dr M}, \eqref{claim-weiss}, and Lemma \ref{weiss}, we get
\[\frac{d}{dr}\tilde{\mathcal{M}}^{x_0}(r,v,p_2) \geq \frac{2}{r}\,\tilde{\mathcal{W}}^{x_0}(r,v)-C\,r^{\gamma}\geq -C\,r^{\gamma-1},\]
as desired.
\end{proof}

\section{Proof of Theorem \ref{th-main}}
\label{sec8}

In this Section we prove Theorem \ref{th-main}.
We start with the following nondegeneracy Lemma.
Recall that $\Sigma_\kappa(u)$ denotes the set of singular points with frequency $\kappa$.

\begin{lemma}\label{nondegeneracy}
Let $u$ be a solution of the obstacle problem \eqref{pb}, with $\varphi\in C^{k,\gamma}(\R^n)$,  $k\geq2$ and $\gamma\in(0,1)$.

Let $x_0\in \Gamma_\kappa(u)$, with $\kappa$ as in \eqref{kappa}, and let $v=v^{x_0}$ be defined as in \eqref{v}.
Then, for all $r\in(0,r_0)$ we have
\[C^{-1}r^\kappa\leq \sup_{\partial\BB}|v|\leq Cr^\kappa\]
for some $C>0$.
\end{lemma}

\begin{proof}
We may assume that $x_0=0$.
The upper bound follows from \eqref{growth-H-kappa} and \cite[Lemma 3.4]{BFR}.

To prove the lower bound, we assume that for a sequence $r=r_j\to0$ we have $\sup_{\partial\BB}|v|= o(r^\kappa)$.
Then,
\begin{equation}\label{contr}
d_r=\left(\frac{1}{r^{n+a}}\int_{\partial\BB}|y|^a v^2\right)^{1/2}=o(r^\kappa).
\end{equation}
Passing to a subsequence if necessary we may assume that
\[v_r(x,y)=\frac{v(rx,ry)}{d_r}\to p_\kappa(x,y)\qquad \textrm{uniformly on}\quad \partial\BBone,\]
for some nonzero $p_\kappa\in\tilde{\mathfrak P}_{a,\kappa}^+$.
Now, for such $p_\kappa$ we use the Monneau formula in Proposition \ref{monneau}.
If \eqref{contr} holds then we have
\[\tilde{\mathcal M}_\kappa(0+,v,p_\kappa)=\int_{\partial\BBone}|y|^a p_\kappa^2=\frac{1}{r^{n+a+\kappa}}\int_{\partial\BB}p_\kappa^2.\]
Therefore, using the monotonicity of $\tilde{\mathcal M}_\kappa(r,v,p_\kappa)+C_Mr^{\gamma}$, we will have that
\[C_mr^\gamma+\frac{1}{r^{n+a+\kappa}}\int_{\partial\BB}|y|^a(v-p_\kappa)^2\geq \frac{1}{r^{n+a+\kappa}}\int_{\partial\BB}|y|^ap_\kappa^2\]
or, equivalently,
\[\frac{1}{r^{n+a+\kappa}}\int_{\partial\BB}|y|^a(v^2-2v p_\kappa)\geq -C_mr^\gamma.\]
After rescaling we obtain
\[\frac{1}{r^{2\kappa}}\int_{\partial\BBone} |y|^a(d_r^2v_r^2-2d_rr^\kappa v_rp_\kappa)\geq -C_Mr^\gamma,\]
and thus
\[\int_{\partial\BBone} |y|^a\left(\frac{d_r}{r^\kappa}v_r^2-2 v_rp_\kappa\right)\geq -C_M\frac{r^{\kappa+\gamma}}{d_r}.\]
By Lemma \ref{bound-H} we have $r^{\kappa+\gamma}/d_r\to0$ as $r\to0$.
Moreover, recall that $v_r\to p_\kappa$ as $r\to0$.
Thus, letting $r\to0$ in the last inequality and using \eqref{contr}, we find
\[-\int_{\partial\BBone}p_\kappa^2\geq0,\]
a contradiction.
\end{proof}

We next prove uniqueness and continuity of blow-ups.

\begin{thrm}\label{uniqeness}
Let $u$ solve the obstacle problem for the fractional Laplacian, and let $\kappa$ be as in \eqref{kappa}.

Then, there exists a modulus of continuity $\omega:\R^+\to\R^+$ such that, for any $x_0\in \Gamma_\kappa(u)$, we have
\[u(x)-\varphi(x)=p_\kappa^{x_0}(x-x_0)+\omega(|x-x_0|)|x-x_0|^\kappa\]
for some polynomial $p_\kappa\in \tilde{\mathfrak P}_{a,\kappa}^+$.
In addition, the mapping $\Gamma_\kappa\ni x_0\mapsto p_\kappa^{x_0}\in\tilde{\mathfrak P}_{a,\kappa}^+$ is continuous, with
\[\int_{\partial\BBone} |y|^a(p_\kappa^{x_0'}-p_\kappa^{x_0})^2\leq \omega(|x_0'-x_0|)\]
for all $x_0,x_0'\in \Gamma_\kappa(u)$.
\end{thrm}

\begin{proof}
Let $v^{x_0}$ be defined as in \eqref{v}, and let
\[v_r^{x_0}(x,y)=\frac{v^{x_0}(x_0+rx,ry)}{r^\kappa}.\]
By Lemma \ref{nondegeneracy} we have that
\[C^{-1}\rho^\kappa\leq \sup_{B(\rho)}|v_r^{x_0}|\leq C\rho^\kappa,\]
for all $\rho\in (0,r_0/r)$.
Thus, exactly as in Proposition \ref{blow-ups} we get
\begin{equation}\label{71}
v_{r_j}^{x_0}\rightarrow v_0^{x_0}\qquad\textrm{in}\quad C^1_{\rm loc}(\R^{n+1})\quad \textrm{along a subsequence}\ r_j\to0.
\end{equation}
Moreover, $v_0^{x_0}$ is not identically zero, and it is an homogeneous polynomial $p_\kappa\in\tilde{\mathfrak P}_{a,\kappa}^+$.

Hence, using \eqref{71} we get
\[\tilde{\mathcal M}_\kappa(0+,v_{x_0},p_\kappa^{x_0}) =\lim_{r_j\to0}\int_{\partial\BBone}|y|^a(v_{r_j}^{x_0}-p_\kappa^{x_0})^2=0.\]
Thus, the monotonicity formula in Proposition \ref{monneau} implies
\begin{equation}\label{72}
\int_{\partial\BBone}|y|^a(v_r^{x_0}-p_\kappa^{x_0})^2= \tilde{\mathcal M}_\kappa(r,v_{x_0},p_\kappa^{x_0}) \longrightarrow0\qquad \textrm{as}\quad r\downarrow0
\end{equation}
(not just along a subsequence).
This immediately implies that the blow-up is unique, and since $v^{x_0}(x,0)=u(x)-\varphi(x)$, we deduce that $u(x)-\varphi(x)=p_\kappa^{x_0}(x-x_0)+o(|x-x_0|^\kappa)$.
The fact that the rest $o(|x-x_0|^\kappa)$ is uniform with respect to $x_0$ follows from a simple compactness argument, as in \cite[Lemma 7.3 and Proposition 7.7]{PSU}.

We now prove continuous dependence of $p_\kappa^{x_0}$ with respect to $x_0$.
Given $\varepsilon>0$ if follows from \eqref{72} that there exists $r_\varepsilon=r_\varepsilon(x_0)>0$ such that
\[\tilde{\mathcal M}_\kappa(r_\varepsilon,v^{x_0},p_\kappa^{x_0})<\varepsilon.\]
Now, by continuous dependence of $v^{x_0}$ with respect to $x_0$, there exists $\delta_\varepsilon=\delta_\varepsilon(x_0)>0$ such that
\[\tilde{\mathcal M}_\kappa(r_\varepsilon,v^{x_0'},p_\kappa^{x_0})<2\varepsilon\]
for all $x_0'\in \Gamma_\kappa(u)$ satisfying $|x_0'-x_0|<\delta_\varepsilon$.
Then, it follows from Proposition \ref{monneau} that
\[\tilde{\mathcal M}_\kappa(r,v^{x_0'},p_\kappa^{x_0})<2\varepsilon+C_Mr_\varepsilon^\gamma\]
for all $r\in(0,r_\varepsilon]$.
Letting $r\to0$ we obtain
\[\int_{\partial\BBone}|y|^a(p_\kappa^{x_0'}-p_\kappa^{x_0})^2=\tilde{\mathcal M}_\kappa(r0+,v_{x_0'},p_\kappa^{x_0}) \leq 2\varepsilon+C_Mr_\varepsilon^\gamma.\]
Since $\varepsilon$ is arbitrary, we deduce that $p_\kappa^{x_0}$ is continuous with respect to $x_0$.
Finally, the uniform continuity follows exactly as in the proof of \cite[Theorem 2.8.4]{GP}.
\end{proof}

Finally, we give the:

\begin{proof}[Proof of Theorem \ref{th-main}]
The result follows from Theorem \ref{uniqeness} and the exact same argument as in the proof of \cite[Theorem 2.6.5]{GP}.
\end{proof}

\end{document}